\documentclass[a4paper,regno]{amsart}

\usepackage{inputenc}
\usepackage{amssymb,amsthm,amsmath}
\usepackage{enumitem}
\usepackage{hyperref}
\usepackage{mathtools}
\usepackage{mathrsfs}
\usepackage{relsize}
\usepackage{placeins}
\usepackage{afterpage}
\usepackage{float}

\usepackage[all]{xy}
\usepackage{tikz-cd}
\usetikzlibrary{calc,decorations.markings,decorations.pathmorphing}
\usepackage{ifthen}
\usepackage{bm}
\usepackage[margin=10pt,font=small,labelfont=bf,
labelsep=period,indention=.7cm]{caption}

\numberwithin{equation}{section}

\newtheorem{thm}{Theorem}[section]
\newtheorem{prop}[thm]{Proposition}
\newtheorem{lem}[thm]{Lemma}
\newtheorem{cor}[thm]{Corollary}
\newtheorem{quest}[thm]{Question}
\theoremstyle{definition}
\newtheorem{rem}[thm]{Remark}
\newtheorem{defn}[thm]{Definition}
\newtheorem{example}[thm]{Example}
\newtheorem{assum}[thm]{Assumption}

\def\cal#1{\mathcal{#1}}

\def\bf#1{\mathbf{#1}}
\def\ol#1{\overline{#1}}
\def\ul#1{\underline{#1}}

\renewcommand{\epsilon}{\varepsilon}

\DeclareMathOperator{\Hom}{Hom}

\DeclareMathOperator{\Ext}{Ext}

\DeclareMathOperator{\modules}{mod} \renewcommand{\mod}{\modules}
\DeclareMathOperator{\determinant}{det} \renewcommand{\det}{\determinant}

\DeclareMathOperator{\Ind}{Ind}
\DeclareMathOperator{\Res}{Res}
\DeclareMathOperator{\id}{id}
\newcommand{\op}{{\rm op}}
\DeclareMathOperator{\Irr}{Irr}
\DeclareMathOperator{\SL}{SL}
\DeclareMathOperator{\GL}{GL}
\DeclareMathOperator{\Tr}{Tr}
\DeclareMathOperator{\supp}{supp}
\DeclareMathOperator{\gl}{gl.\!dim}

\newcommand{\Z}{\mathbb{Z}}
\newcommand{\R}{\mathbb{R}}
\newcommand{\C}{\mathbb{C}}
\newcommand{\PP}{\mathbb{P}}

\renewcommand{\to}{\rightarrow}

\newcommand{\minus}{\smallsetminus}

\begin{document}

\title[Higher representation infinite algebras from metacyclic groups]{Higher representation infinite algebras from McKay quivers of metacyclic groups}
\author{Simone Giovannini}
\address{Simone Giovannini: Dipartimento di Matematica, Universit\`{a} di Padova, Via Trieste 63, 35121 Padova (Italy)}
\email{sgiovann@math.unipd.it}
\thanks{The author is supported by the PhD programme in Mathematical Sciences, Department of Mathematics, University of Padova, and by grants BIRD163492 and DOR1690814 of University of Padova}

\keywords{$n$-representation infinite algebra, higher preprojective algebra, skew group algebra, McKay quiver, metacyclic group}
\subjclass[2010]{16G20; 16S35; 16E65}

\begin{abstract}
For each prime number $s$ we introduce examples of $(s-1)$- and $s$-representation infinite algebras in the sense of Herschend, Iyama and Oppermann, which arise from skew group algebras of some metacyclic groups embedded in $\SL(s,\C)$ and $\SL(s+1,\C)$. For this purpose, we give a description of the McKay quiver with a superpotential of such groups. Moreover we show that for $s=2$ our examples correspond to the classical tame hereditary algebras of type $\tilde{{\mathrm D}}$.
\end{abstract}

\maketitle
\tableofcontents

\section*{Introduction}

In the late seventies McKay \cite{McKay} observed that the finite subgroups of $\SL(2,\C)$ are closely related to simply laced Dynkin diagrams. More precisely, we have the following bijection:
\[
\begin{array}{ccc}
\{ \text{Finite subgroups of $\SL(2,\C)$} \} & \longleftrightarrow & \{ \text{Simply laced Dynkin diagrams} \} \\
\text{Cyclic} & & \text{Type }{\mathrm A}_m \\
\text{Binary dihedral} & & \text{Type }{\mathrm D}_m \\
\text{Binary tetrahedral} & & \text{Type }{\mathrm E}_6 \\
\text{Binary octahedral} & & \text{Type }{\mathrm E}_7 \\
\text{Binary icosahedral} & & \text{Type }{\mathrm E}_8
\end{array}
\]
which has become known as \emph{McKay correspondence}. There are many ways to achieve this bijection.

One is related with the geometry of the quotient space $\C^2/G$, where $G$ is a finite subgroup of $\SL(2,\C)$ acting naturally on $\C^2$. This space has a singularity at the origin $0$ and we can take a minimal resolution $Y\to\C^2/G$: then the preimage of $0$ is a union of irreducible components isomorphic to $\PP^1_\C$, and the corresponding intersection graph is exactly the Dynkin diagram associated to $G$.

The McKay correspondence can also be realized in the context of representation theory of associative algebras by means of preprojective algebras, which were introduced by Gel'fand and Ponomarev in \cite{GP}. Let $G$ be a non-trivial finite subgroup of $\SL(2,\C)$ and let $\C[V]$ be the algebra of polynomial functions over the natural representation $V$ of $G$. Then $G$ acts naturally on $\C[V]$ and it was proven in \cite{RV} that the skew group algebra $\C[V]*G$ is Morita equivalent to the preprojective algebra of a quiver whose underlying graph is the extended variant of the Dynkin diagram corresponding to $G$ under the McKay correspondence.

Now we can consider a preprojective algebra obtained as above and define a grading on it by putting, for all double arrows $\bullet\!\leftrightarrows\!\bullet$, one arrow of degree $1$ and the other of degree $0$. Then its degree $0$ part is the path algebra of an extended Dynkin quiver. In this way (excluding the case of type $\tilde{{\mathrm A}}$ with cyclic orientation) we obtain all (basic connected) tame hereditary finite dimensional $\C$-algebras (\cite{DF},\cite{N}). In particular, every skew group algebra of some finite subgroup $G\leq\SL(2,\C)$ is Morita equivalent to the preprojective algebra of some tame hereditary algebra.

The notion of tame hereditary algebras, and more in general of representation infinite hereditary algebras, has been recently generalized in the context of higher dimensional Auslander-Reiten theory by Herschend, Iyama and Oppermann, who gave the following definition.

\begin{defn}[\cite{HIO}]
Let $n\geq1$. A finite dimensional algebra $\Lambda$ is \emph{$n$-representation infinite} if $\gl \Lambda \leq n$ and $\nu_n^{-i}(\Lambda)$ is a complex concentrated in degree $0$ for all $i\geq0$, where $\nu_n:=\nu\circ[-n] \colon \mathrm{D}^{\rm b}(\mod\Lambda) \to \mathrm{D}^{\rm b}(\mod\Lambda)$ and $\nu$ is the Serre functor \cite{BK} of $\mathrm{D}^{\rm b}(\mod\Lambda)$, the bounded derived category of finitely generated modules over $\Lambda$.
\end{defn}

This is indeed a generalization of the classical case, since the $1$-representation infinite algebras are exactly the hereditary representation infinite algebras.

In \cite{HIO} the authors showed that many of the nice properties of hereditary representation infinite algebras can be generalized in dimension $n$. For example, for an $n$-representation infinite algebra $\Lambda$, they introduce subcategories of $n$-preprojective, $n$-preinjective and $n$-regular modules in $\mod \Lambda$, which behave very similarly to their classical analogues. Moreover, they study the $(n+1)$-preprojective algebra of $\Lambda$, which is defined as the tensor algebra
$$ \Pi_{n+1}(\Lambda) := T_\Lambda \Ext^n_\Lambda(D\Lambda,\Lambda)$$
of the $\Lambda$-bimodule $\Ext^n_\Lambda(D\Lambda,\Lambda)$, where $D$ denotes the standard duality. For $n=1$ it is known (even in more generality, see for example \cite{R}) that the $2$-preprojective algebra is isomorphic to the classical preprojective algebra we mentioned above.

Higher preprojective algebras play a fundamental role in the following characterization of $n$-representation infinite algebras.

\begin{thm}[\cite{AIR},\cite{K},\cite{MM},\cite{HIO}]\label{thm_n-RI_CYGP1_intro}
There is a bijection between isomorphism classes of $n$-representation infinite algebras and isomorphism classes of bimodule $(n+1)$-Calabi-Yau algebras of Gorenstein parameter $1$ with finite dimensional degree $0$ part. This bijection is realized by sending an $n$-representation infinite algebra $\Lambda$ to its higher preprojective algebra $\Pi_{n+1}(\Lambda)$ and a bimodule $(n+1)$-Calabi-Yau algebra $B$ of Gorenstein parameter $1$ to its degree $0$ part $B_0$.
\end{thm}

Bimodule $(n+1)$-Calabi-Yau algebras of Gorenstein parameter $1$ are a special class of positively graded algebras. Their definition will be given in Section~\ref{sec_CYGP}: here we just mention that being bimodule $(n+1)$-Calabi-Yau is a property of the algebra itself and does not involve the grading, while the Gorenstein parameter is a property of the grading.

If on one side the $n$-representation infinite algebras have an interesting behaviour and provide new insights in higher dimensional Auslander-Reiten theory, on the other side very few examples of them are known: hence it is of a certain interest to look for new ones. The only known examples at the moment come from tensor product constructions \cite[\S~2.1]{HIO}, from non-commutative algebraic geometry \cite{MM,HIMO,BH} and from skew group algebras. The latter is the family of examples which we would like to enrich in this paper.

Since, as we already mentioned, all tame hereditary representation infinite algebras can be obtained from finite subgroups of $\SL(2,\C)$, it is natural to wonder if, in a similar way, one can obtain examples of $n$-representation infinite algebras from finite subgroups of $\SL(n+1,\C)$. If $G\leq\SL(n+1,\C)$ and $V$ is the natural representation of $G$, then it is known that the skew group algebra $\C[V]*G$ is Morita equivalent to the path algebra of a quiver $Q_G$ (called the McKay quiver of $G$) modulo some relations: we will denote this algebra by $\Pi_G$. It is proved in \cite{BSW} that the relations in $\Pi_G$ are induced by a superpotential $\omega_G$ and that both $\C[V]*G$ and $\Pi_G$ are bimodule $(n+1)$-Calabi-Yau algebras.

Now we can ask the following question.
\begin{quest}\label{quest_preproj}
If $G\leq\SL(n+1,\C)$, then is $\C[V]*G$ Morita equivalent to the $(n+1)$-preprojective algebra of some $n$-representation infinite algebra?
\end{quest}

In other words, according to Theorem~\ref{thm_n-RI_CYGP1_intro}, can we always find a grading on $\Pi_G$ which gives it Gorenstein parameter $1$ and such that the degree $0$ part is finite dimensional?

Unfortunately, the answer to this question can be negative if $n\geq2$. For example, Thibault proved in \cite{T} that if $G$ is conjugate to a subgroup of $\SL(n_1,\C)\times\SL(n_2,\C)$ for some $n_1,n_2\geq1$ such that $n_1+n_2=n+1$, then $\C[V]*G$ cannot be Morita equivalent to a higher preprojective algebra.

On the other end, there are also situations where the answer to Question~\ref{quest_preproj} is affirmative. In \cite{HIO} the authors describe a method to obtain examples of $n$-representation infinite algebras from abelian subgroups of $\SL(n+1,\C)$. They called them \emph{$n$-representation infinite algebras of type $\tilde{{\mathrm A}}$}, since for $n=1$ they coincide with path algebras of Dynkin quivers of type $\tilde{{\mathrm A}}$.

The aim of this paper is to find other examples of groups for which the answer to Question~\ref{quest_preproj} is affirmative. Given a prime number $s$ and positive integers $m,r,t$ satisfying certain conditions, we will consider the finite subgroup $G$ of $\GL(s,\C)$ generated by the following matrices:
\[
\alpha = \left(
\begin{array}{cccc}
\epsilon_m & 0 & \cdots & 0 \\
0 & \epsilon_m^r & \cdots & 0 \\
\vdots & \vdots & \ddots & \vdots \\
0 & 0 & \cdots & \epsilon_m^{r^{s-1}}
\end{array}\right),
\qquad
\beta = \left(
\begin{array}{ccccc}
0 & 0 & \cdot & 0 & \epsilon_m^t \\
1 & 0 & \cdots & 0 & 0 \\
\vdots & \ddots & \ddots & \vdots & \vdots \\
\vdots & \ddots & \ddots & \vdots & \vdots \\
0 & 0 & \cdots & 1 & 0
\end{array}\right),
\]
where $\epsilon_m$ is a primitive $m$-th root of unity. The conditions we impose imply that $G$ satisfies the following properties:
\begin{itemize}
\item the subgroup $A$ generated by $\alpha$ is normal in $G$, and it is cyclic of order $m$;
\item the quotient $G/A$ is cyclic of order $s$ and so it is simple, since $s$ is a prime number;
\item the conjugation by $\beta$ induces an action of $G/A$ on $A$, where the generator of $G/A$ sends $\alpha$ to $\alpha^r$.
\end{itemize}
In particular we have that $G$ is a metacyclic group, i.e. it is an extension of cyclic groups. Moreover, we will assume that $\alpha$ is not a scalar multiple of the identity: this implies that $G$ is not abelian, and so our examples are different from the ones studied in \cite{HIO}.

For such a group $G$ the following two cases can occur:
\begin{itemize}
\item[(SL)] $G$ is contained in $\SL(s,\C)$;
\item[(GL)] $G$ is not contained in $\SL(s,\C)$.
\end{itemize}
In the case (GL), $\Pi_G$ is the path algebra of the McKay quiver of $G$ modulo some relations which are induced by a \emph{twisted} superpotential $\omega_G$. In particular, $\Pi_G$ is not bimodule $s$-Calabi-Yau: however, there is a natural embedding of $\GL(s,\C)$ in $\SL(s+1,\C)$ and we can consider the image of $G$ under it, which we denote by $G'$. Then $\Pi_{G'}$ is $(s+1)$-Calabi-Yau, since $G'\leq\SL(s+1,\C)$, and $\Pi_{G'}$ and the superpotential $\omega_{G'}$ can be easily obtained from $\Pi_G$ and $\omega_G$.

In Section~\ref{sec_sga-of-metacyclic-groups} we will give a description of the McKay quiver $Q_G$ of a metacyclic group $G$ and of a (twisted) superpotential $\omega_G$. For this purpose, we will rely on the already known description of the McKay quiver $Q_A$ of the abelian subgroup $A\leq G$, and we will exploit a $G/A$-action on $Q_A$. Moreover, we will prove in Theorem~\ref{thm_Psi(omegaG)cont_inPhi(omegaA)} that every path in the superpotential $\omega_G$ of $G$ in a certain sense ``comes from a path in the superpotential of $A$'': in this way we obtain only a partial description of $\omega_G$, but it will suffice for our purposes. In fact we will use this result to prove, in Proposition~\ref{prop_grading-A-->grading-G}, that if we have a grading on $Q_A$ such that $\omega_A$ is homogeneous of degree $a$ and which satisfies an invariance hypothesis, then there exists a grading on $Q_G$ such that $\omega_G$ is homogeneous of degree $a$. An analogous result is proven in Proposition~\ref{prop_grading-A'-->grading-G'} for the embedded group $G'\leq\SL(s+1,\C)$.

In Section~\ref{sec_cuts} we will describe explicitly some gradings satisfying the above results. Following \cite{HIO}, we will give a geometric picture of the McKay quiver of $Q_A$ and we will show that we can obtain gradings on this quiver by considering particular subsets of arrows called \emph{cuts}. Moreover, we will show that gradings satisfying Proposition~\ref{prop_grading-A-->grading-G} can be obtained from cuts which are invariant under the $G/A$-action.

In Section~\ref{sec_examples} we will give some examples where the previous results can be applied. In particular, for each prime number $s$ and for each integer $b\geq1$ we will define two metacyclic groups $M(s,b)\leq\SL(s,\C)$ and $\hat{M}(s,b)\leq\GL(s,\C)$. We will show that both this families of groups give a positive answer to Question~\ref{quest_preproj}: more precisely, we have the following result (see Corollary~\ref{cor_RI-from-metacyclic}).

\begin{thm}
Let $s$ be a prime number.
\begin{enumerate}[label=(\alph*)]
\item For each integer $b\geq 1$, there exists an $(s-1)$-representation infinite algebra which is the degree 0 part of $\Pi_G$, where $G=M(s,b)$.
\item For each integer $b\geq 2$ such that $(b,s)=1$, there exists an $s$-representation infinite algebra which is the degree 0 part of $\Pi_{G'}$, where $G=\hat{M}(s,b)$ and $G'$ is its embedding in $\SL(s+1,\C)$.
\end{enumerate}
\end{thm}

Finally, we will compute some examples for $s=2,3$. In the (SL) case for $s=2$ we will show that with our construction we obtain all tame hereditary algebras of type $\tilde{{\mathrm D}}$. For $s=3$, the groups we considered belong to the family of \emph{trihedral groups}, which have already been studied from a geometric point of view (see for example \cite{I,IR,Leng}).

Other groups which also raised some interest in geometry are binary dihedral groups in $\GL(2,\C)$ (see \cite{NdC,Leng}): they can be obtained from our construction in the (GL) case with $s=2$.

\subsection*{Acknowledgement}
Most of this paper was written while the author was visiting Uppsala University. The author would like to thank Martin Herschend for the precious comments and suggestions about this paper and for helpful discussions.
The author would like to thank the anonymous referee for the helpful comments which contributed to improve the readability of the paper.

\section{Notations and preliminaries}\label{sec_prelim}

Throughout the paper we will work over the field $\C$ of complex numbers. Unless stated otherwise, all the modules we will consider will be right modules.

\subsection{Algebras and bimodules}

Let $A$ be a $\C$-algebra and denote by $A^e := A^\op\otimes_\C A$ its enveloping algebra. Note that an $A^e$-module $M$ can be considered as an $A$-bimodule by putting $amb=m(a\otimes b)$ for all $m\in M$, $a,b\in A$.

Given an $A^e$-module $M$, we define the bimodule dual $M^\vee := \Hom_{A^e}(M,A^e)$: we will regard it as an $A^e$-module with action given by $(\psi(a\otimes b))(m)=(b\otimes a)\psi(m)$. We will denote by $M^* := \Hom_\C(M,\C)$ the complex dual, which is again an $A^e$-module if we set $(\psi(a\otimes b))(m)=\psi(m(b\otimes a))$.

\subsection{Gradings}

All the gradings we will consider are over $\Z$, unless stated otherwise. If $A$ is a graded $\C$-algebra and $V=\bigoplus_{d\in\Z}V_d$ is a finite dimensional graded $A$-module, then we can give $V^*$ a structure of graded $A$-module by setting $(V^*)_d=V_{-d}^*$. If $M=\bigoplus_{d\in\Z}M_d$ is a finitely generated graded projective $A$-module, then we have a natural grading on the dual $\Hom_A(M,A)$, where the degree $d$ part $\Hom_A(M,A)_d$ is given by morphisms $M\to A$ which are homogeneous of degree $d$.

If $A$ is a graded $\C$-algebra, then $A^e$ inherits naturally a graded structure. The functor $(\_)^\vee$ preserves the finitely generated graded projective $A^e$-modules, where the graded structure on $M^\vee$ is given as above.

For a homogeneous element $x$ in a graded algebra or module, we will often denote its degree by $|x|$.

\subsection{Representations of groups}\label{subsec_repgroups}

Given a finite group $G$, we will denote by $\C G$ its group algebra. We will sometimes identify representations of $G$ with modules over $\C G$. Given two representation $M$ and $N$ of $G$, their tensor product $M \otimes_\C N$ will be regarded as a representation of $G$ via the action $g(m\otimes n)=gm\otimes gn$.

If $H$ is a subgroup of $G$ and $M$ is a representation of $H$, we call $\Ind^G_H(M)$ the representation of $G$ induced by $M$, which, as left modules, is defined by $\C G \otimes_{\C H} M$. If we fix a set $\{g_1,\dots,g_n\}$ of left coset representatives of $G/H$, then $\Ind^G_H(M)$ is isomorphic to $\bigoplus_{i=1}^n g_i \otimes M$ with $G$-action given by $g(g_i\otimes m)=g_j\otimes hm$ for all $g\in G$, where $g_j$ is the coset representative which satisfy $gg_i=g_jh$ for $h\in H$. In order to simplify the notation we will write $g_i m$ for $g_i\otimes m$.

If $N$ is a representation of $G$, we call $\Res^G_H(N)$ the restriction of $M$ to $H$. It is well known (see for example \cite{Ben}) that the restriction functor is a right adjoint to the induction functor. More precisely we have an isomorphism of vector spaces
\begin{align*}
\Hom_H(N,\Res^G_H(M)) & \longrightarrow \Hom_G(\Ind^G_H(N),M) \\
\phi & \longmapsto \left(gx \mapsto g\phi(x)\right)
\end{align*}
which is functorial in $M$ and $N$.

We also have an isomorphism
\begin{align*}
\Ind^G_H(\Res^G_H(M)\otimes_\C N) & \longrightarrow M\otimes_\C \Ind^G_H(N) \\
g(m\otimes n) & \longmapsto gm\otimes gn
\end{align*}
of representations of $G$.

\section{Derivation quotient algebras and gradings}\label{sec_derivation}

In this section we describe superpotentials and derivation quotient algebras following \cite{BSW}, and we consider gradings on them.

\subsection{Superpotentials}

Let $Q=(Q_0,Q_1,s,t)$ be a finite quiver. This is the data of a finite set of vertices $Q_0$, a finite set of arrows $Q_1$ and two maps $s,t\colon Q_1 \to Q_0$ which assign to an arrow respectively its source and its target. We will denote by $\C Q$ the path algebra of $Q$, where the product $ab$ of two arrows has to be intended as ``first do $b$, then $a$''.

We denote by $\C Q_k$ the subspace of $\C Q$ generated by paths of length $k$ and set $S:=\C Q_0$, $V:=\C Q_1$. Then $S$ is a finite dimensional semisimple $\C$-algebra. We have a trace function $\Tr\colon S\to\C$ defined by $\Tr(e)=1$ for all primitive idempotents $e\in S$. We will give $V$ the unique $S$-bimodule structure where $t(a)as(a)=a$ for all $a\in Q_1$: in this way $\C Q$ is identified with the tensor algebra $T_S V$. 

Given a path $p\in\C Q_m$, we can define for any $k\leq m$ the left and right partial derivatives with respect to $p$ as the maps $\partial_p,\delta_p\colon \C Q_k \to \C Q_{m-k}$ given by
\[
\partial_p q := \left\{ \begin{array}{cc}
r & \text{ if } q=pr,\\
0 & \text{ otherwise,}
\end{array} \right.
\qquad
q \delta_p := \left\{ \begin{array}{cc}
r & \text{ if } q=rp,\\
0 & \text{ otherwise.}
\end{array} \right.
\]

Note that the set of arrows $Q_1$ is a basis of $V$, and we have a dual basis $\{a^* \,,\, a\in Q_1\}$ of $V^*$. These yield bases for the vector spaces $\C Q_k$ and $\C Q_k^*$. For a path $p=a_1\cdots a_k \in \C Q_k$, we will denote by $p^*$ the element $a^*_k\cdots a^*_1 \in \C Q_k^*$.

\begin{defn}[\cite{BSW}]
An element $\omega\in \C Q_n$ is called a \emph{superpotential of degree $n$} if it satisfies the following two conditions:
\begin{enumerate}[label=(\arabic*)]
\item $\omega$ is a linear combination of cyclic paths (equivalently, $\omega s = s \omega$ for any $s\in S$);
\item $\sigma(\omega)=(-1)^{n-1}\omega$, where $\sigma$ is the map defined on paths as $\sigma(a_1\cdots a_n) = a_n a_1\cdots a_{n-1}$.
\end{enumerate}
\end{defn}

For later purposes, it will be convenient to consider the twisted analogue of this definition. We call \emph{twist} a $\C$-algebra automorphism $\tau$ of $\C Q$ which satisfies $\tau(\C Q_k) \subseteq \C Q_k$ and permutes the primitive idempotents. If $M$ is an $S^e$-module, we define its (right) twist $M_\tau$ as the vector space $M$ with $S^e$-action given by $m.(s\otimes t):=m(s\otimes \tau(t))$.

\begin{defn}
An element $\omega\in \C Q_n$ is called a \emph{twisted superpotential of degree $n$} if it satisfies the following two conditions:
\begin{enumerate}[label=(\arabic*)]
\item $\omega$ is a linear combination of paths $p$ satisfying $t(p)=\tau(s(p))$ (equivalently, $\omega s = \tau(s) \omega$ for any $s\in S$);
\item  $\sigma^\tau(\omega)=(-1)^{n-1}\omega$, where $\sigma^\tau$ is the map defined on paths as $\sigma^\tau(a_1\cdots a_n) = \tau(a_n) a_1\cdots a_{n-1}$.
\end{enumerate}
\end{defn}
Note that if the twist is trivial we recover the definition of superpotential.

Given a twisted superpotential $\omega\in \C Q_n$ and an integer $k\leq n$, we can define an $S^e$-module morphism
$$ \Delta^\omega_k \colon \C Q_k^* \otimes_S S_\tau \to \C Q_{n-k} $$
by $\Delta^\omega_k(p^* \otimes s) := \partial_p \omega s$. We denote by $W_{n-k}$ the image of $\Delta^\omega_k$.

\begin{defn}
Let $\omega\in \C Q_n$ be a (twisted) superpotential. The \emph{derivation quotient algebra} of $\omega$ of order $k$ is defined as
$$ \cal D(\omega,k) := \C Q / \langle W_{n-k} \rangle = \C Q / \langle \partial_p \omega \,,\, p \text{ path of length } k \rangle, $$
where $\langle W_{n-k} \rangle$ denotes the smallest two-sided ideal of $\C Q$ which contains $W_{n-k}$.
\end{defn}

For later use we give the following definition.

\begin{defn}
Write a superpotential $\omega\in \C Q_n$ as
$$\omega = \sum_{|p|=n} c_p p,$$
for some scalars $c_p\in\C$. Then we define the \emph{support of $\omega$} to be the set
$$\supp(\omega) := \{ p \,|\, p \text{ path of length $n$, } c_p\neq0 \}.$$
\end{defn}

\subsection{Graded quivers}

We recall that a morphism of quivers $\phi\colon Q \to Q'$ consists in two maps $Q_0 \to Q'_0$, $Q_1 \to Q'_1$ (which, abusing of notation, we will both denote again by $\phi$) which are compatible with the source and target maps.

\begin{defn}
\begin{itemize}
\item A ($\Z-$)\emph{graded quiver} is a couple $(Q,g)$ consisting of a quiver $Q$ and a map $g\colon Q_1 \to \Z$. A \emph{morphism of graded quivers} $\phi\colon (Q,g) \to (Q',g')$ is given by a morphism of quivers $\phi\colon Q \to Q'$ such that the following diagram commutes:
\[
\xymatrix{
Q_1 \ar[rr]^{\phi} \ar[dr]_{g} & & Q'_1 \ar[dl]^{g'} \\
 & \Z & }
\]
\item Let $\phi\colon Q \to Q'$ be a morphism of quivers and suppose that $Q$ is graded by $g$. We will say that $\phi$ is \emph{$g$-gradable} if, for every arrow $a\in Q'_1$, $\phi^{-1}(a)$ is either empty or a homogeneous subset of $Q_1$ (i.e. all its elements have the same degree).
\end{itemize}
\end{defn}

Let $\phi\colon Q \to Q'$ be a morphism of quivers. There is a natural way to induce a grading on $Q$ from a grading on $Q'$, and vice versa, which we now illustrate.

\begin{defn}\label{def_graded-quiver-morphisms}
\begin{enumerate}[label=(\arabic*)]
\item Suppose that $g$ is a grading on $Q'$. Then we define a grading $\phi^*g$ by putting $(\phi^*g)(a)=g(\phi(a))$ for all $a\in Q_1$. Note that this is the unique grading on $Q$ which makes $\phi$ a morphism of graded quivers.
\item Suppose that $g$ is a grading on $Q$ and that $\phi$ is $g$-gradable. It is clear that there always exists a grading $g'$ on $Q'$ which makes $\phi$ a morphism of graded quivers: indeed, it is enough to put $g'(a)$ equal to $k$ if the elements of $\phi^{-1}(a)$ have all degree $k$, and to any integer if $\phi^{-1}(a)=\varnothing$. If, in addition, we assume that $\phi\colon Q_1 \to Q'_1$ is surjective, then such a grading is unique and we will denote it by $\phi_*g$.
\end{enumerate}
\end{defn}

\begin{rem}
A grading on $Q$ induces in a natural way a grading on the path algebra $\C Q$. Note that in general this grading does not coincide with the natural one on $\C Q$ given by path length (i.e., the one obtained by putting all arrows in degree $1$).
\end{rem}

From now on we will fix a grading on $\C Q$ which comes from a grading on $Q$. Note that in this case all elements of $S$ have degree $0$. We will now show that if $\omega$ is a superpotential which is homogeneous with respect to this grading, then we get a graded structure on $\cal D(\omega,k)$.

\begin{lem}\label{lem_D(w,k)-graded}
Suppose that $\omega\in \C Q_n$ is a superpotential which is homogeneous of degree $a$. Then $W_i$ is a graded $S^e$-submodule of $\C Q$ for all $i=0,\dots,n$. In particular the derivation quotient algebra $\cal D(\omega,k)$ inherits in a natural way a structure of graded $\C$-algebra.
\end{lem}
\begin{proof}
Firstly we show that the morphism $\Delta_k^\omega$ is homogeneous of degree $a$. Indeed, if $p^*\in\C Q_k^*$ has degree $d$, then $p$ has degree $-d$ and by definition of partial derivative we have $|\Delta_k^\omega(p^*)|=|\partial_p\omega|=a-d$.

Hence each $W_i$ is a graded $S^e$-submodule of $\C Q$ because it is the image of an homogeneous morphism. This implies that the ideal generated by it is homogeneous and so we have a well defined grading on the quotient $D(\omega,k) = \C Q / \langle W_{n-k} \rangle$.
\end{proof}

\section{Graded Calabi-Yau algebras and the Gorenstein parameter}\label{sec_CYGP}

We will now show how we can obtain $n$-representation infinite algebras from derivation quotient algebras.

\begin{defn}
Let $a$ be an integer and $n\geq 2$. A positively graded $\C$-algebra $A = \bigoplus_{i\geq 0} A_i$ is called \emph{bimodule $n$-Calabi-Yau of Gorenstein parameter $a$} if the $A^e$-module $A$ has a bounded resolution $P_\bullet$ of finitely generated graded projective $A^e$-modules such that we have an isomorphism of complexes
$$ \Phi\colon P_\bullet \xrightarrow{\sim} P_\bullet^\vee[n](-a). $$
Here $[n]$ denotes the shift of complexes, while $(-a)$ denotes the shift of the grading.
\end{defn}

In other words, we want a commutative diagram
\[
\begin{tikzcd}
\dots \arrow[r] & 0 \arrow[d] \arrow[r] & P_n \arrow[r,"d_n"] \arrow[d,"\Phi_n"] & \dots \arrow[r,"d_2"] & P_1 \arrow[r,"d_1"] \arrow[d,"\Phi_1"] & P_0 \arrow[r] \arrow[d,"\Phi_0"] & 0 \arrow[d] \arrow[r] & \dots \\
\dots \arrow[r] & 0 \arrow[r] & P_0^\vee \arrow[r,"d_1^\vee"] & \dots \arrow[r,"d_{n-1}^\vee"] & P_{n-1}^\vee \arrow[r,"d_n^\vee"] & P_n^\vee \arrow[r] & 0 \arrow[r] & \dots
\end{tikzcd}
\]
where the maps $\Phi_i$ are isomorphisms and homogeneous of degree $-a$.

The above definition is motivated by the following theorem.

\begin{thm}[\cite{AIR},\cite{K},\cite{MM},\cite{HIO}]\label{thm_n-RI_CYGP1}
If $A$ is a bimodule $n$-Calabi-Yau of Gorenstein parameter $1$ such that $\dim_\C A_0 < \infty$, then $A_0$ is $(n-1)$-representation infinite in the sense of \cite{HIO}. Viceversa, every $(n-1)$-representation infinite algebra is the degree $0$ part of a bimodule $n$-Calabi-Yau algebra of Gorenstein parameter $1$.
\end{thm}

In view of this, we will be interested in finding examples of bimodule $n$-Calabi-Yau algebras of Gorenstein parameter $1$. In the next subsection, following \cite{BSW}, we will describe a way to obtain such examples from derivation quotient algebras.

\subsection{Examples from derivation quotient algebras}

Let $Q$ be a quiver and retain the notation of Section~\ref{sec_derivation}. We put $A:=\cal D(\omega,n-2)$ and fix a positive grading on $\C Q$ such that $\omega$ is homogeneous of degree $a$. This induces, by Lemma~\ref{lem_D(w,k)-graded}, a grading on $A$ which is again positive.

Set $P_i:=A \otimes_S W_i \otimes_S A$ if $0\leq i\leq n$ and $P_i:=0$ otherwise. Then we have a complex of projective $A^e$-modules
\begin{equation}\label{eq_resolution}
P_\bullet = (\dots \to 0 \to P_n \xrightarrow{d_n} P_{n-1} \xrightarrow{d_{n-1}} \dots \xrightarrow{d_1} P_0 \to 0 \to \dots),
\end{equation}
with differentials $d_i\colon P_i \to P_{i-1}$ defined by
$$ d_i = \varepsilon_i (d_i^l + (-1)^i d_i^r), $$
where
\[ d_i^l(1 \otimes \partial_p \omega \otimes 1) := \sum_{b\in Q_1} b \otimes \partial_b\partial_p\omega \otimes 1,\]
\[ d_i^r(1 \otimes \partial_p \omega \otimes 1) := \sum_{b\in Q_1} 1 \otimes \partial_p\omega\delta_b \otimes b,\]
\[ \varepsilon_i:= \left\{ \begin{array}{cc}
(-1)^{i(n-i)} & \text{ if } i<(n+1)/2, \\
1 & \text{ otherwise.}
\end{array}\right.\]
Clearly each $P_i$ is a graded $A^e$-module, because $W_i$ is a graded $S^e$-submodule of $\C Q$. Moreover it is easy to see that the differentials $d_i$ have degree zero, so $P_\bullet$ is a complex of graded projective $A^e$-modules. The following is a graded version of \cite[Theorem~6.2]{BSW}.

\begin{thm}\label{thm_potential_deg_a}
Suppose that the complex $P_\bullet$ defined above is a resolution of $A$ (i.e. it is exact in positive degrees and $H^0(P_\bullet)=A$). Then $A$ is $n$-bimodule Calabi-Yau of Gorenstein parameter $a$.
\end{thm}
\begin{proof}
It is proved in \cite{BSW} that if $P_\bullet$ is a projective resolution of $A$, then it is self dual, i.e. $P_\bullet \cong P_\bullet^\vee[n]$. The duality isomorphism given in [loc.cit.] is described as follows (see also \cite{Kar} for more details). For each $i=0,\dots,n$ we have a perfect pairing
$$ \langle,\rangle \colon W_i \otimes W_{n-i} \to \C $$
given by $\langle \partial_p\omega,\partial_q\omega \rangle := \Tr(\partial_{qp}\omega)$. The isomorphism of $A^e$-modules
$$ \Phi_i \colon P_i = A \otimes_S W_i \otimes_S A \to \Hom_{A^e}(A \otimes_S W_{n-i} \otimes_S A, A^e) = P_{n-i}^\vee $$
is given by $\Phi_i(\alpha \otimes \partial_p\omega \otimes \alpha')(\beta \otimes \partial_q\omega \otimes \beta') = \alpha'\beta \otimes \langle \partial_p\omega,\partial_q\omega \rangle \otimes \beta'\alpha$. These maps commute with the differentials and thus yield an isomorphism of complexes $\Phi_\bullet\colon P_\bullet \to P^\vee_\bullet[n]$.

Now we are only left to show that $\Phi_i$ is homogeneous of degree $-a$. It is enough if we prove that if $\partial_p\omega\in W_i$ has degree $d$, then $\Phi_i(1\otimes\partial_p\omega\otimes 1)$ has degree $d-a$. For this it suffices to show that $\Phi_i(1\otimes\partial_p\omega\otimes 1)(1\otimes\partial_q\omega\otimes 1)$ has degree $e+d-a$ whenever $|\partial_q\omega|=e$. Note that we have $|p|=|\omega|-|\partial_p\omega|=a-d$, $|q|=a-e$ and so $|\partial_{qp}\omega|=a-|p|-|q|=d+e-a$. Suppose now that $d+e\neq a$: then $|\partial_{qp}\omega|\neq0$, but this implies that $\partial_{qp}\omega=0$ because $\partial_{qp}\omega\in S$, so $\Phi_i(1\otimes\partial_p\omega\otimes 1)(1\otimes\partial_q\omega\otimes 1)=0$. If instead $d+e=a$, then it is clear that $\Phi_i(1\otimes\partial_p\omega\otimes 1)(1\otimes\partial_q\omega\otimes 1)$ has degree $e+d-a=0$.
\end{proof}

\section{Skew group algebras and McKay quivers}\label{sec_Skew group algebras and McKay quivers}

In \cite{BSW} it is shown that a source of bimodule Calabi-Yau algebras is provided by a family of skew group algebras. In this section we will summarize this construction and consider the graded case in order to apply Theorem~\ref{thm_potential_deg_a}.

\begin{defn}
Let $R$ be a $\C$-algebra and let $G$ be a finite group acting on $R$ by algebra automorphisms. We define the \emph{skew group algebra} $R*G$ as the $\C$-algebra whose underlying vector space is $R \otimes_\C \C G$, with product given by
$$ (r_1 \otimes g_1)(r_2 \otimes g_2) = r_1 g_1(r_2) \otimes g_1 g_2. $$
\end{defn}

Let $V$ be a $\C$-vector space of dimension $n$ and let $G$ be a finite subgroup of $\GL(V)$. Call $\C[V]$ the algebra of polynomial functions on $V$: then $G$ acts on it in a natural way and we can form the skew group algebra $\C[V]*G$. We have the following description of the latter.

\begin{thm}[\cite{BSW}]
The skew group algebra $\C[V]*G$ is Morita equivalent to a basic algebra $\Pi_G$ which is a derivation quotient algebra of order $n-2$ with a (twisted) superpotential $\omega$ of degree $n$. More explicitly, we have
$$\Pi_G = \C Q / \langle \partial_p \omega \,,\, |p|=n-2 \rangle, $$
where $Q$ is the McKay quiver of $G$ (see Definition \ref{def_McKay}).
\end{thm}

\begin{defn}\label{def_McKay}
Let $\Irr(G)$ be a complete set of representatives for the irreducible representations of $G$. The \emph{McKay quiver $Q$ of $G$} relative to $V$ is described as follows. Its set of vertices is $\Irr(G)$ and, for any $S,T\in \Irr(G)$, the set of arrows going from $S$ to $T$ is given by a basis of the vector space $\Hom_G(S,V\otimes_\C T)$.
\end{defn}

We now give an explicit description, following \cite{BSW}, of the superpotential $\omega$ of the algebra $\Pi_G$.

We denote by $\det_V$ the $1$-dimensional representation of $G$ where each $g\in G$ acts as the multiplication by $\det(g)$. Clearly $\det_V$ is isomorphic to the exterior product $\bigwedge^n V$. Now consider the functor $\tau:=\det_V \otimes_\C$\underline{\hspace{1em}} and note that it sends irreducible representations to irreducible representations. Hence we get a bijection $\tau\colon \Irr(G) \to \Irr(G)$: it is easy to see that its inverse, which we denote by $\tau^-$, is given by tensoring by $\det_{V^*}$. Clearly we have that $\tau(V\otimes_\C S)=V\otimes_\C \tau(S)$, so $\tau$ extends to an automorphism of the path algebra $\tau\colon \C Q \to \C Q$ which preserves the path length: this will be our twist.

Consider now a path
$$ p\colon v_1 \xrightarrow{a_1} v_2 \xrightarrow{a_2} \dots \to v_n \xrightarrow{a_n} v_{n+1} $$
of length $n$ in the McKay quiver $Q$. If we call $S_i$ the representation corresponding to the vertex $v_i$, then we can view the arrows as morphisms $a_i\colon S_i \to V\otimes S_{i+1}$. This induces a composition
\begin{equation}\label{eq_composition_superpotential}
S_1 \xrightarrow{a_1} V\otimes S_2 \xrightarrow{\id_{V}\otimes a_2} \dots \xrightarrow{\id_{V^{\otimes n-1}}\otimes a_n} V^{\otimes n}\otimes S_{n+1} \xrightarrow{\alpha^n_V\otimes \id_{S_{n+1}}} \det_V \otimes S_{n+1} = \tau(S_{n+1}),
\end{equation}
where, for each $m\leq n$, we denote by $\alpha^m_V$ the antisymmetrizer $V^{\otimes m}\to \bigwedge^m V$, $x_1\otimes \dots \otimes x_m \mapsto x_1\wedge \dots \wedge x_m$. By Schur's Lemma, the composition morphism \eqref{eq_composition_superpotential} is the multiplication by a scalar, which will be denoted by $c_p$. Note in particular that $c_p\neq 0$ only when $\tau(t(p))=s(p)$.

\begin{thm}[\cite{BSW}]\label{thm_superpotential_McKay}
The superpotential $\omega$ of the algebra $\Pi_G$ is given by
$$ \omega = \sum_{|p|=n} (c_p \dim t(p))p. $$
\end{thm}

\begin{rem}\label{rem_superpotential-depends}
It is worth pointing out that the superpotential described above depends on the choice of the basis for the arrows in $Q$.
\end{rem}

Suppose that the basis we choose for $\Hom_G(S,V\otimes_\C T)$ is invariant under the twist. Then the automorphism $\tau$ of $\C Q$ is actually induced by an automorphism of the quiver $Q$ and, by \cite[Lemma~4.3]{BSW}, the coefficients of the superpotential have the property that $$c_{a_1\cdots a_n} = (-1)^{n-1}c_{\tau(a_n)a_1\cdots a_{n-1}}$$ for all paths $a_1\cdots a_n$ in $Q$. In case our basis is invariant only up to multiplication by a non-zero scalar, we can still define an automorphism of $Q$, which we call $\tau'$, by putting $\tau'(v):=\tau(v)$ if $v\in Q_0$ and $\tau'(a):=a'$ if $a\in Q_1$, where $a'$ is the only arrow such that $\tau(a)$ is a multiple of $a'$. It is easy to see that in this case we have the following weaker version of \cite[Lemma~4.3]{BSW}.

\begin{lem}\label{lem_cyclic-permutation}
If $p=a_1\cdots a_n$ is a path of length $n$ in $Q$, then the coefficient $c_p$ satisfies
$$c_{a_1\cdots a_n} = \mu c_{\tau'(a_n)a_1\cdots a_{n-1}}$$
for a non-zero scalar $\mu\in\C^{\times}$. In particular, $a_1\cdots a_s\in\supp(\omega)$ if and only if $\tau'(a_n)a_1\cdots a_{n-1}\in\supp(\omega)$.
\end{lem}

The main reason we are interested in skew group algebras comes from the following corollary to Theorem \ref{thm_potential_deg_a}.

\begin{cor}\label{cor_superpot-homog-->GP1}
Suppose that $G$ is contained in $\SL(V)$. Put a grading on the path algebra $\C Q$ of the McKay quiver of $G$ in such a way that the superpotential $\omega$ described in Theorem \ref{thm_superpotential_McKay} is homogeneous of degree $1$.

Then the algebra $\Pi_G = \C Q / \langle \partial_p \omega \,,\, |p|=n-2 \rangle$, equipped with the grading induced by $\C Q$, is $n$-bimodule Calabi-Yau of Gorenstein parameter $1$. In particular, its degree zero part, if finite dimensional, is an $(n-1)$-representation infinite algebra.
\end{cor}
\begin{proof}
In \cite{BSW} the authors prove that any skew group algebra $\C[V]*G$, for a finite subgroup $G$ of $\SL(V)$, has the property of being $n$-Calabi-Yau and Koszul. Moreover they show that, for any derivation quotient algebra $A$ satisfying these two properties, the complex \eqref{eq_resolution} is a resolution of $A$. This applies in particular to $\C[V]*G$, and also to $\Pi_G$ since being $n$-Calabi-Yau and Koszul is invariant under Morita equivalence. Hence we can conclude by Theorem \ref{thm_potential_deg_a} that $\Pi_G$ is $n$-bimodule Calabi-Yau of Gorenstein parameter $1$.

The last assertion then follows directly from Theorem \ref{thm_n-RI_CYGP1}.
\end{proof}

\subsection{Subgroups of $\GL(n,\C)$ embedded in $\SL(n+1,\C)$}\label{sec_embedding_gen}

We saw previously that we can obtain examples of $(n-1)$-representation infinite algebras from skew group algebras of finite subgroups of $\SL(n,\C)$. If, instead, we start from a subgroup of $\GL(n,\C)$ not contained in $\SL(n,\C)$, then we cannot apply Corollary~\ref{cor_superpot-homog-->GP1} anymore. However, every subgroup of $\GL(n,\C)$ can be regarded as a subgroup of $\SL(n+1,\C)$ by means of the natural embedding $\GL(n,\C) \hookrightarrow \SL(n+1,\C)$ given by
\[
X \mapsto \left(
\begin{array}{c|c}
X  & 0 \\ \hline
0 & \frac{1}{\det(X)}
\end{array}\right).
\]

For a finite subgroup $G$ of $\GL(V)$, where $V$ is a $\C$-vector space of dimension $n$, we denote by $G'$ its image under the above embedding, so that $G'$ is a subgroup of $\SL(W)$ for a vector space $W\supseteq V$ of dimension $n+1$. Let $Q$ be the McKay quiver of $G$ relative to $V$ and $\omega$ be the associated twisted superpotential: if $Q'$ denotes the McKay quiver of $G'$ relative to $W$, then it is known (see for example \cite{Guo}) that a superpotential $\omega'$ for $\Pi_{G'}$ can be obtained from $\omega$ by adding some arrows. More precisely we have the following.

\begin{prop}\label{prop_emb-quiver}
Fix a basis for the arrows of $Q$ which is invariant, up to multiplication by a non-zero scalar, under the twist. Then $Q$ can be viewed as a subquiver of $Q'$, and the latter can be obtained from the former by adding an arrow $i\to \tau(i)$ for each vertex $i\in Q_0$. Moreover, the support of the superpotential $\omega'$ is obtained from the one of $\omega$ by adding these arrows. More precisely, we have that a path
$$ \tau(i_n) \to i_1 \to \dots \to i_n $$
is in $\supp(\omega)$ if and only if the path
$$ \tau(i_n) \to i_1 \to \dots \to i_n \to \tau(i_n)$$
is in $\supp(\omega')$, and all the paths in $\supp(\omega')$, up to cyclic permutation, are obtained in this way.
\end{prop}

\section{Metacyclic groups}

We now introduce a family of groups to which we will later apply the results of the previous sections. All groups in this family will satisfy the following property.

\begin{defn}
A group $G$ is metacyclic if it has a normal cyclic subgroup $A$ such that $G/A$ is cyclic.
\end{defn}

Some generalities about metacyclic groups can be found in \cite[\S~47]{CR}. In particular, it is shown in loc. cit. that one can associate to a metacyclic group some integers which must satisfy certain conditions. On the contrary, in the following we will start with integers satisfying such conditions and associate to them a metacyclic group embedded in a general linear group.

\begin{defn}
Let $m,r,s,t$ be positive integers satisfying the following conditions:
\begin{enumerate}[label=(M\arabic*)]
\item $(m,r)=1$, where $(m,r)$ indicates the greatest common divisor of $m$ and $r$; \label{item_(m,r)=1}
\item $r^s\equiv 1 \pmod m$; \label{item_r^s=1(m)}
\item $(r-1)t\equiv 0 \pmod m$. \label{item_(r-1)t=0(m)}
\end{enumerate}
Define $G$ to be the finite subgroup of $\GL(s,\C)$ generated by the following matrices:
\[
\alpha = \left(
\begin{array}{cccc}
\epsilon_m & 0 & \cdots & 0 \\
0 & \epsilon_m^r & \cdots & 0 \\
\vdots & \vdots & \ddots & \vdots \\
0 & 0 & \cdots & \epsilon_m^{r^{s-1}}
\end{array}\right),
\qquad
\beta = \left(
\begin{array}{ccccc}
0 & 0 & \cdots & 0 & \epsilon_m^t \\
1 & 0 & \cdots & 0 & 0 \\
\vdots & \ddots & \ddots & \vdots & \vdots \\
\vdots & \ddots & \ddots & \vdots & \vdots \\
0 & 0 & \cdots & 1 & 0
\end{array}\right),
\]
where $\epsilon_m$ is a fixed primitive $m$-th root of unity. We will refer to $G$ as the \emph{metacyclic group associated to $m,r,s,t$} and we will denote by $A$ the subgroup of $G$ generated by $\alpha$.
\end{defn}

We now make some comments on the above definition. By \ref{item_(m,r)=1} we have that $A$ is cyclic of order $m$, and it is normal in $G$ because, by \ref{item_r^s=1(m)}, $\beta^{-1}\alpha\beta=\alpha^r$. Condition \ref{item_(r-1)t=0(m)} implies that $\beta^s=\alpha^t$, hence it easy to see that $G/A$ is a cyclic group of order $s$ generated by the class of $\beta$. This shows that $G$ is metacyclic; in particular, the order of $G$ is $sm$.

At a later stage we will need to consider additional conditions on the integers $m,r,s,t$, in order to apply the results of the previous sections or to simplify some calculations. For convenience we will list all of them now and refer to them in the following whenever they will be needed:
\begin{enumerate}[label=(M\arabic*)]\setcounter{enumi}{3}
\item $s$ is a prime number; \label{item_s-prime}
\item $r \not\equiv 1 \pmod m$; \label{item_r-not1}
\item $m=sn$ for an integer $n$; \label{item_m=sn}
\item $r-1=sb$ for an integer $b$. \label{item_r=1(s)}
\end{enumerate}

Let us make a brief comment on these additional conditions. We will need \ref{item_s-prime} because it implies that $G/A$ has prime order, and so it is a simple group: this will simplify a lot the description of the irreducible representation of $G$, allowing us to use \cite[Corollary~47.14]{CR} in the next subsection. Condition \ref{item_r-not1} will be used from Section~\ref{sec_sga-of-metacyclic-groups}: in particular it implies that $G$ is not abelian, thus ensuring that we are not in the case already studied in \cite{HIO}. Another consequence of \ref{item_r-not1} is Proposition~\ref{prop_r_not_equiv_1}, thanks to which we will have fewer cases to analyse in the study of the superpotential associated to $G$ in Section~\ref{sec_sga-of-metacyclic-groups}. Conditions \ref{item_m=sn} and \ref{item_r=1(s)} will be introduced in Section~\ref{sec_cuts} in order to prove the existence of gradings which satisfy the hypotheses of Corollary~\ref{cor_superpot-homog-->GP1}.

\subsection{Irreducible representations}

The representation theory of metacyclic groups is well known. Here we summarize the description of their irreducible representations, in order to fix some notation for the following sections.

Let $G$ be the metacyclic group associated to integers $m,r,s,t$. From now on we will suppose that $G$ satisfies also condition \ref{item_s-prime}. Note that this implies that the quotient $G/A$ is simple, because it is cyclic of prime order.

Let $\Irr(A)=\{S_i\}_{i=0}^{m-1}$ be a complete set of non-isomorphic irreducible representations of $A$. Since $A$ is abelian, the $S_i$'s are all $1$-dimensional: we put $S_i=\C v_i$ and assume that the action of $A$ on $S_i$ is given by $\alpha v_i=\epsilon_m^i v_i$. We will often consider the indices $i$ as integers modulo $m$ and identify $\Irr(A)$ with $\Z/m\Z$ via $S_i \leftrightarrow i$. Note in particular that this gives $S_i \otimes S_j \cong S_{i+j}$.

Consider now the induced representations $\Ind^G_A(S_i)=:T_i$. We choose as a basis of $T_i$ the set $\{v_i,\beta v_i,\dots,\beta^{s-1}v_i\}$, so $G$ acts on $T_i$ in the following way:
$$\alpha(\beta^k v_i) =  \epsilon_m^{r^ki} \beta^k v_i, \qquad k=0,\dots,s-1, $$
$$\beta(\beta^k v_i) =  \beta^{k+1} v_i, \qquad k=0,\dots,s-2, $$
$$\beta(\beta^{s-1} v_i) =  \beta^s v_i = \alpha^t v_i = \epsilon_m^{ti} v_i.$$

Thus the matrices of $\alpha$ and $\beta$ with respect to the action on $T_i$ are
\[
\alpha \mapsto \left(
\begin{array}{cccc}
\epsilon_m^i & 0 & \cdots & 0 \\
0 & \epsilon_m^{ri} & \cdot & 0 \\
\vdots & \vdots & \ddots & \vdots \\
0 & 0 & \cdots & \epsilon_m^{r^{s-1}i}
\end{array}\right),
\qquad
\beta \mapsto \left(
\begin{array}{ccccc}
0 & 0 & \cdots & 0 & \epsilon_m^{ti} \\
1 & 0 & \cdots & 0 & 0 \\
\vdots & \ddots & \ddots & \vdots & \vdots \\
\vdots & \ddots & \ddots & \vdots & \vdots \\
0 & 0 & \cdots & 1 & 0
\end{array}\right).
\]

The following proposition tells us when the representation $T_i$ is irreducible.
\begin{prop}[\cite{CR}]\label{prop_irrG}
\begin{enumerate}[label=(\alph*)]
\item $T_i$ is irreducible if and only if $r^ki \not\equiv i \pmod m$ for all $k=1,\dots,s-1$;
\item $T_i \cong T_j$ if and only if there exist a $k$ such that $r^ki=j$.
\end{enumerate}
\end{prop}

\begin{rem}\label{rem_a'}
Note that, since $s$ is prime, we can replace the condition (a) above by:
\textit{\begin{enumerate}[label=(\alph*')]
\item $T_i$ is irreducible if and only if $ri \not\equiv i \pmod m$.
\end{enumerate}}
\end{rem}

Suppose now that $ri \equiv i \pmod m$, so that $T_i$ is not irreducible. The matrix
\[\left(
\begin{array}{ccccc}
0 & 0 & \cdots & 0 & \epsilon_m^{ti} \\
1 & 0 & \cdots & 0 & 0 \\
\vdots & \ddots & \ddots & \vdots & \vdots \\
\vdots & \ddots & \ddots & \vdots & \vdots \\
0 & 0 & \cdots & 1 & 0
\end{array}\right),
\]
which represents the action of $\beta$, has characteristic polynomial $p(\lambda)=(-1)^s(\lambda^s-\epsilon_m^{ti})$ and so it is diagonalisable. Its eigenvalues are $\{ \lambda_{i,\ell} \,|\, \ell=0,\dots,s-1 \}$, where $\lambda_{i,\ell}:=\eta_i \epsilon_s^\ell$ and $\eta_i:=\epsilon_m^{\frac{t}{s}i}$ is a fixed $s$-th root of $\epsilon_m^{ti}$. The elements
$$w_i^{(\ell)}:= \sum_{k=0}^{s-1} \lambda_{i,\ell}^{s-1-k} \beta^k v_i, \qquad \ell=0,\dots,s-1,$$
provide a basis of eigenvectors for $\beta$. Indeed,
\[
\beta w_i^{(\ell)} = \sum_{k=0}^{s-1} \lambda_{i,\ell}^{s-k-1} \beta(\beta^k v_i) = \sum_{h=1}^{s-1} \lambda_{i,\ell}^{s-h} \beta^h v_i + \epsilon_m^{ti}v_i = \lambda_{i,\ell} \sum_{h=1}^{s-1} \lambda_{i,\ell}^{s-h-1} \beta^h v_i + \lambda_{i,\ell}^s v_i = \lambda_{i,\ell}\, w_i^{(\ell)}
\]
and
\[
\alpha w_i^{(\ell)} = \sum_{k=0}^{s-1} \lambda_{i,\ell}^{s-k-1} \alpha(\beta^k v_i) = \sum_{k=0}^{s-1} \lambda_{i,\ell}^{s-k-1} \epsilon_m^{r^k i} \beta^k v_i = \sum_{k=0}^{s-1} \lambda_{i,\ell}^{s-k-1} \epsilon_m^i \beta^k v_i = \epsilon_m^i w_i^{(\ell)},
\]
where, in the second equality, we used the fact that $r^k i \equiv i \pmod m$.

Hence the $1$-dimensional subspace $T_i^{(\ell)}:= \C w_i^{(\ell)}$ is a subrepresentation of $T_i$, and we have a decomposition
$$T_i \cong \bigoplus_{\ell=0}^{s-1} T_i^{(\ell)}.$$

\section{The skew group algebra of a metacyclic group}\label{sec_sga-of-metacyclic-groups}

In this section, $G$ will be the metacyclic group associated to some integers $m,r,s,t$ satisfying conditions \ref{item_(m,r)=1},\dots,\ref{item_r-not1}. Note that \ref{item_r-not1} implies that $G$ is not abelian: we are assuming it because the case of skew group algebras of abelian groups has already been studied in \cite{HIO}. Moreover, this assumption will simplify a lot some calculations at a later stage (see Proposition~\ref{prop_r_not_equiv_1}).

The action of $G/A$ on $G$ by conjugation induces an automorphism $\varphi$ of $G$ given by $\varphi(g)=\beta^{-1}g\beta$. This in turn induces an action of $G/A$ on $\Irr(A)$ given by $\varphi(i)=ri$. From now on we will fix a set $\cal D$ of representatives in $\Z/m\Z$ of this action and we will denote by $\cal F$ the set of fixed points. Note that the orbits of the fixed points have cardinality $1$, so we have $\cal F\subseteq\cal D$ regardless of what choice for $\cal D$ we made. By Proposition~\ref{prop_irrG} we have the following result.

\begin{prop}\label{prop_irrG2}
The set $\Irr(G) = \{T_i \,|\, i\in \cal D \smallsetminus \cal F\} \cup  \{T_i^{(\ell)} \,|\, i\in \cal F \,, \ell=0,\dots,s-1\}$ is a complete set of nonisomorphic irreducible representations of $G$. Moreover we have that $\dim_\C T_i=s$ and $\dim_\C T_i^{(\ell)}=1$.
\end{prop}

\paragraph*{\textbf{Notation. }} For each $i\in\Z/m\Z$ we call $\ul{i}$ its representative in $\cal D$, i.e. the only element of the $G/A$-orbit of $i$ which is contained in $\cal D$ (note that $\ul i=i$ if $i\in\cal D$). We fix an integer $\kappa_i\in\{0,\dots,s-1\}$ such that $r^{\kappa_i}\ul{i} \equiv i$. If $i\in \cal D \smallsetminus \cal F$ it is clear, by Proposition~\ref{prop_irrG}(a), that we can choose $\kappa_i$ in a unique way; otherwise, for a fixed point $i\in\cal F$, we set $\kappa_i:=0$.

\begin{example}
Let $m=21$, $r=4$, $s=3$, $t=0$, so the action of $G/A$ on $\Z/21/\Z$ is given by the multiplication by $4$. A possible choice of representatives is
$$\cal D = \{0,4,7,8,9,12,13,14,17\} \subseteq \Z/21\Z.$$
In this case we have, for example, $\ul1=\ul4=\ul{16}=4$ and $\kappa_1=2$, $\kappa_4=0$, $\kappa_{16}=1$. We will analyse this example in detail in Example~\ref{ex_s=3}.
\end{example}

Call $V$ the $s$-dimensional natural representation of $G$, which coincides with $T_1$. We may note that, by Remark~\ref{rem_a'}, condition \ref{item_r-not1} is equivalent to say that $V$ is irreducible. It is easy to see that, for all $i$, we have an isomorphism $\Res^G_A(T_i) \cong \bigoplus_{k=0}^{s-1} S_{r^ki}$, $\beta^k v_i \mapsto v_{r^ki}$. In the case where $i$ is a fixed point, we have an isomorphism $\Res^G_A(T_i^{(\ell)}) \cong S_i$, $w_i^{(\ell)} \mapsto v_i$. We will call again $V$ the restriction of $V$ to $A$, which clearly coincides with the natural representation of $A$ and is isomorphic to $\bigoplus_{k=0}^{s-1} S_{r^k}$.

Our strategy for describing the skew group algebra of $G$ will exploit the action of $G/A$ on the McKay quiver of $A$. The structure of the latter, and also the relations which give an isomorphism with the skew group algebra of $A$, are well known. The following description is a particular case of \cite[Corollary~4.1]{BSW}.

\begin{prop}[\cite{BSW}]\label{prop_quiver_A}
Let $Q_A$ be the McKay quiver of $A$. Then the following holds.
\begin{itemize}
\item $Q_A$ has vertices $\Z/m\Z$ and an arrow $x_k^i \colon i \to i-r^k$ for each $i\in \Z/m\Z$ and $k=0,\dots,s-1$ (sometimes, if this does not cause confusion, we will omit the superscript and write $x_k$ in place of $x_k^i$).
\item Let $\mathfrak{S}_s$ be the symmetric group on $\{0,\dots,s-1\}$. For each permutation $\sigma\in\mathfrak{S}_s$ and each vertex $i$, define a path
$$ \bf p_\sigma^i := x_{\sigma(s-1)} \cdots x_{\sigma(0)} \colon i \to i-\textstyle\sum_{k=0}^{s-1}r^{\sigma(k)}. $$
Then the superpotential of $Q_A$ is given by
$$ \omega_A = \sum_{i\in(Q_A)_0} \sum_{\sigma\in\mathfrak{S}_s} (-1)^{\sigma}\bf p_\sigma^i,$$
where $(-1)^{\sigma}$ is the sign of the permutation $\sigma$.
\end{itemize}
In particular, the skew group algebra $\C[V]*A$ is isomorphic to the path algebra of $Q_A$ modulo the relations
$$ \{x_h^{i-r^k} x_k^i = x_k^{i-r^h} x_h^i \,|\, i\in \Z/m\Z,\, 0 \leq k,h \leq s-1\}.$$
\end{prop}
\begin{proof}
The proof is given in \cite[Corollary~4.1]{BSW}. Here we only mention, since it will be needed later, which basis of arrows we choose.

For each $i\in\Z/m\Z$ and $k=0,\dots,s-1$, we will take $x^i_k$ to be the morphism in $\Hom_A(S_i,V\otimes S_{i-r^k}) \cong \Hom_A(S_{r^k}\otimes S_{i-r^k},V\otimes S_{i-r^k})$ defined by $v_{r^k}\otimes v_{i-r^k} \mapsto v_{r^k}\otimes v_{i-r^k}$.
\end{proof}

\subsection{The vertices and arrows of $Q_G$}

Now we will consider the McKay quiver of $G$, which we denote by $Q_G$. Proposition \ref{prop_irrG2} gives us a description of its vertices: we will call $i$ the vertex corresponding to $T_i$ for $i\in\cal D \smallsetminus \cal F$, and $i^{(\ell)}$ the vertex corresponding to $T_i^{(\ell)}$ for $i\in\cal F$, $\ell=0,\dots,s-1$. By an abuse of terminology, for $i\in\cal F$, we will call ``fixed points'' both the vertex $i\in (Q_A)_0$ and each of the vertices $i^{(\ell)}\in (Q_G)_0$, for $\ell=0,\dots,s-1$.

We will now describe the arrows of $Q_G$. Recall that in order to do this we must choose a basis of the vector space $\Hom_G(S,V\otimes T)$ for all $S,T\in \Irr(G)$.

First note that for $i,j\in\Z/m\Z$ we have, by the isomorphisms discussed in Section~\ref{subsec_repgroups},
\begin{align}
\Hom_G(T_i, V\otimes T_j) & = \Hom_G(\Ind^G_A(S_i), V\otimes T_j) \cong \Hom_A(S_i, \Res^G_A(V\otimes T_j)) \nonumber \\
 & \cong \Hom_A(S_i, V\otimes \Res^G_A(T_j)) \cong \Hom_A(S_i, V\otimes (\textstyle\bigoplus_{k=0}^{s-1} S_{r^kj})). \label{eq_indres-for-arrows}
\end{align}
By Proposition \ref{prop_quiver_A} $\Hom_A(S_i, V\otimes (\bigoplus_{k=0}^{s-1} S_{r^kj}))$ is generated by elements of the form $x_p^i$, for some $a,p\in \{0,\dots,s-1\}$ which satisfy $r^aj \equiv i-r^p \pmod m$. Now take such an element $x_p^i\in \Hom_A(S_i, V\otimes (\bigoplus_{k=0}^{s-1} S_{r^kj}))$ and call $x^i_{p,a}$ its image in $\Hom_G(T_i, V\otimes T_j)$ under the inverses of the isomorphisms \eqref{eq_indres-for-arrows}. Then, recalling that we identify $\beta^a v_j\in\Res^G_A(T_j)$ with $v_{r^aj}\in \bigoplus_{k=0}^{s-1} S_{r^kj}$, it is easy to see that this homomorphism is explicitly given by
\begin{align*}
x^i_{p,a} \colon  T_i & \longrightarrow V\otimes T_j \\
  v_i & \longmapsto \beta^pv_1 \otimes \beta^av_j.
\end{align*}

Before going on we make the following observation.

\begin{prop}\label{prop_r_not_equiv_1}
In the quivers $Q_G$ and $Q_A$ there are no arrows between two fixed points.
\end{prop}
\begin{proof}
Suppose that we have an arrow $i^{(\ell)}\to j^{(\ell')}$ in $Q_G$, where $i,j\in\cal F$ and $0\leq\ell,\ell'\leq s-1$. Then the space $\Hom_G(T_i^{(\ell)}, V\otimes T_j^{(\ell')}) \cong \Hom_G(T_i^{(\ell)}, T_{1+j})$ is different from zero. By \ref{item_r-not1}, $V$ is irreducible and thus so is $V\otimes T_j^{(\ell')} \cong T_{1+j}$. Hence, by Schur's Lemma, we must have $T_i^{(\ell)} \cong T_{1+j}$, which is a contradiction because these representations have different dimensions as vector spaces.

Now take two fixed points $i,j\in(Q_A)_0$. The arrows $i\to j$ in $Q_A$ are given by a basis of $\Hom_A(S_i,V\otimes S_j) \cong \Hom_A(S_i,\bigoplus_{k=0}^{s-1} S_{r^k+j})$, so if there exists such an arrow then we must have $i \equiv r^a+j \pmod m$ for some $a$. This implies that $r^a$ is a fixed point, because so is $i-j$. Hence we have $r \equiv 1$, which is again a contradiction by \ref{item_r-not1}.
\end{proof}

So for choosing a basis of the arrows we only have to consider the following three possibilities.
\begin{enumerate}[label=(\arabic*)]
\item $i,j\in \cal D \smallsetminus \cal F$. Then we choose $\{x^i_{p,a} \,|\, r^aj \equiv i-r^p \pmod m\}$ as a basis of $\Hom_G(T_i, V\otimes T_j)$. These elements form indeed a basis because they are the image under an isomorphism of a basis of $\Hom_A(S_i, V\otimes (\bigoplus_{k=0}^{s-1} S_{r^kj}))$.
\item $i\in \cal D \smallsetminus \cal F$ and $j\in \cal F$. For any $\ell = 0,\dots,s-1$, we call $\pi_j^{(\ell)}\colon T_j \to T_j^{(\ell)}$ the projection morphism and, given $a,p$ which satisfy $r^aj \equiv i-r^p \pmod m$, we define
$$x_{p,a}^{i(\ell)} := (\id_V \otimes \pi_j^{(\ell)}) \circ x_{p,a}^i \colon T_i \to V\otimes T_j^{(\ell)}.$$
Note that, since $j$ is a fixed point, $r^aj\equiv j$ and so every choice of $a$ satisfies the condition above. We choose $\{x_{p,0}^{i(\ell)} \,|\, j \equiv i-r^p \pmod m\}$ to be our basis for $\Hom_G(T_i,V\otimes T_j^{(\ell)})$. This set is indeed a basis, because it has cardinality at most $1$ and $\Hom_G(T_i,V\otimes T_j^{(\ell)})$ has dimension at most $1$.
\item $i\in \cal F$ and $j\in \cal D \smallsetminus \cal F$. For any $\ell = 0,\dots,s-1$, we call $\iota_i^{(\ell)}\colon T_i^{(\ell)} \to T_i$ the inclusion morphism and, given $a,p$ which satisfy $r^aj \equiv i-r^p \pmod m$, we define
$$x_{p,a}^{(\ell)i} := x_{p,a}^i \circ \iota_i^{(\ell)} \colon T_i^{(\ell)} \to V\otimes T_j.$$
Note that, since $i$ is a fixed point, we have that $r^{a+b}j \equiv i-r^{p+b}$ for any $b$, so we can assume that the exponent of $r$ is zero in the condition above. We choose $\{x_{p,0}^{(\ell)i} \,|\, j \equiv i-r^p \pmod m\}$ to be our basis for $\Hom_G(T_i^{(\ell)},V\otimes T_j)$. This set is indeed a basis, because it has cardinality at most $1$ and $\Hom_G(T_i^{(\ell)},V\otimes T_j)$ has dimension at most $1$.
\end{enumerate}

The following lemma consists in some calculations which will be used in the next subsections.

\begin{lem}\label{lem_calculations-arrows}
\begin{enumerate}[label=(\alph*)]
\item We have $\pi_i^{(\ell)}(v_i) = \frac{1}{s} \epsilon_m^{-ti} \lambda_{i,\ell} w_i^{(\ell)}$ and $\iota_i^{(\ell)}(w_i^{(\ell)}) = \sum_{k=0}^{s-1} \lambda_{i,\ell}^{s-1-k} \beta^k v_i$.
\item $x_{p,0}^{i(\ell)}(v_i) = \frac{\epsilon_m^{-tj}}{s}\lambda_{j,\ell} \beta^p v_1 \otimes w_j^{(\ell)}$.
\item $x_{p,0}^{(\ell)i}(w_i^{(\ell)}) = \sum_{k=0}^{s-1} \lambda_{i,\ell}^{s-1-k} \beta^{k+p}v_1 \otimes \beta^{k}v_j$.
\item If $j$ is a fixed point, the composition
$$ T_i \xrightarrow{x_{p,0}^{i(\ell)}} V \otimes T_j^{(\ell)} \xrightarrow{\id_V \otimes x_{q,0}^{(\ell)j}} V \otimes V \otimes T_h $$
sends $v_i$ to
$$ \frac{1}{s} \sum_{k=0}^{s-1} \lambda_{j,\ell}^{s-k} \beta^p v_1 \otimes \beta^{k+q} v_1 \otimes \beta^{k} v_h. $$
\end{enumerate}
\end{lem}
\begin{proof}
(a) The second claim is immediate from the definition of $w_i^{(\ell)}$. For the first claim, it is enough to show that $v_i=\frac{1}{s} \sum_{\ell=0}^{s-1} \epsilon_m^{-ti} \lambda_{i,\ell} w_i^{(\ell)}$. Indeed, we have
\[
\frac{1}{s} \sum_{\ell=0}^{s-1} \epsilon_m^{-ti} \lambda_{i,\ell} w_i^{(\ell)} = \frac{1}{s} \sum_{\ell=0}^{s-1} \epsilon_m^{-ti} \lambda_{i,\ell} \sum_{k=0}^{s-1} \lambda_{i,\ell}^{s-1-k} \beta^k v_i = \frac{1}{s} \sum_{k=0}^{s-1} \epsilon_m^{-ti} \left(\sum_{\ell=0}^{s-1} \lambda_{i,\ell}^{s-k}\right) \beta^k v_i.
\]
Recalling that $\lambda_{i,\ell}=\eta_i \epsilon_s^\ell$, where $\eta_i^s=\epsilon_m^{ti}$, we get
\[
\sum_{\ell=0}^{s-1} \lambda_{i,\ell}^{s-k} = \eta_i^{s-k} \sum_{\ell=0}^{s-1} \epsilon_s^{l(s-k)} = \eta_i^{s-k} \sum_{\ell=0}^{s-1} \epsilon_s^{-kl} = \left\{ \begin{array}{cc}
s\eta_i^s=s\epsilon_m^{ti} & \text{ if } k=0, \\
\frac{1-\epsilon_s^{-ks}}{1-\epsilon_s^{-k}}=0 & \text{ if } k\neq0,
\end{array} \right.
\]
and the result follows.

(b), (c), (d) follow immediately from (a).
\end{proof}

\begin{rem}
We may note that in the previous lemma we made an abuse of notation. Indeed, the basis of $T_i$ is made of the elements $\beta^k v_1$ for $k\in\{0,\dots,s-1\}$, so when we write $\beta^k v_i$ for some $k\not\in\{0,\dots,s-1\}$ this should be intended as $\beta^{\bar{k}} v_i$ (where $\bar{k}$ denotes the smallest non-negative integer in the equivalence class of $k$ modulo $s$) multiplied by a constant, more precisely by a power of $\varepsilon_m^{ti}$. The exact value of this constant will not be important for us in what follows, the only thing which should be pointed out is that it is different from zero. So, keeping this in mind, we will tacitly carry on this abuse of notation and continue to write $\beta^k v_i$ even if $k\not\in\{0,\dots,s-1\}$.
\end{rem}

Our next aim will be to describe the superpotential of $Q_G$. Actually, this superpotential will be twisted, where the twist is induced by the tensor product with the representation $\det_V$. Hence we will start by describing explicitly this automorphism.

\subsection{The twist in $Q_G$}\label{sec_twist}

Set $c:=\sum_{k=0}^{s-1}r^k$ and note that it is a fixed point. We may observe that the elements $\alpha$ and $\beta$ act on $\det_V$ respectively as the multiplication by $\epsilon_m^c$ and by $(-1)^{s-1}\epsilon_m^t$. Since $t$ is a fixed point, we have that $r^kt \equiv t \pmod m$ for all $k$, thus $tc \equiv ts \pmod m$. Hence, putting $d_s:= \left\{\begin{array}{cc} 0 & \text{ if $s>2$,} \\ 1 & \text{ if $s=2$}\end{array}\right.$, we get that $\lambda_{c,d_s} = \epsilon_m^{\frac{t}{s}c}\epsilon_s^{d_s} = \epsilon_m^t(-1)^{s-1}$ and so $\det_V = T_c^{(d_s)}$. It is easily checked that the maps
\begin{equation}\label{eq_twist-non-fp}
\begin{split}
T_i \otimes \det_V & \to T_{i+c} \\
v_i \otimes w_c^{(d_s)} & \mapsto v_{i+c}
\end{split}
\end{equation}
and, if $i$ is a fixed point,
\begin{equation}\label{eq_twist-fp}
\begin{split}
T_i^{(\ell)} \otimes \det_V & \to T_{i+c}^{(\ell+d_s)} \\
w_i^{(\ell)} \otimes w_c^{(d_s)} & \mapsto w_{i+c}^{(\ell+d_s)},
\end{split}
\end{equation}
are isomorphisms (for the second map, note that $\lambda_{i,\ell}\lambda_{c,d_s}=\lambda_{i+c,\ell+d_s}$).

From now on, we will need to make the following assumption.
\begin{assum}\label{ass_repclosed}
The set of representatives $\cal D$ is closed under the sum by $c$.
\end{assum}

Hence we have that the twist acts on vertices by $\tau(i)=i+c$ if $i\in\cal D\smallsetminus\cal F$, and $\tau(i^{(\ell)})=(i+c)^{(\ell+d_s)}$ if $i\in\cal F$, $0\leq \ell\leq s-1$.

The following lemma can be easily proved using the isomorphisms \eqref{eq_twist-non-fp} and \eqref{eq_twist-fp}.

\begin{lem}\label{lem_action-tau}
The twist $\tau$ sends each arrow in $Q_G$ to a non-zero scalar multiple of another arrow. More precisely, we have:
\begin{enumerate}[label=(\arabic*)]
\item if $i,j\in \cal D\smallsetminus\cal F$, then $\tau(x_{p,a}^i)=\lambda_{c,d_s}^{-a}x_{p,a}^{i+c}$;
\item if $i\in \cal D\smallsetminus\cal F$ and $j\in\cal F$, then $\tau(x_{p,0}^{i(\ell)}) = \lambda_{c,d_s}^{s-1} x_{p,0}^{i(\ell+d_s)}$;
\item if $i\in\cal F$ and $j\in \cal D\smallsetminus\cal F$, then $\tau(x_{p,0}^{(\ell)i}) = \lambda_{c,d_s}^{1-s} x_{p,0}^{(\ell+d_s)i}$.
\end{enumerate}
\end{lem}

In view of the above lemma, we have that $\tau$ induces in a natural way an automorphism $\tau'$ of $Q_G$, as we observed in the discussion above Lemma~\ref{lem_cyclic-permutation}.

\subsection{The superpotential of $Q_G$}

Now we are ready to begin the study of the superpotential $\omega_G$. Our aim will be to prove Theorem~\ref{thm_Psi(omegaG)cont_inPhi(omegaA)}, which says that every path in $\supp(\omega_G)$ is induced from a path in $\supp(\omega_A)$ (we will make this statement more precise by introducing a new quiver $\tilde{Q}_G$, see Definition~\ref{def_Qtilde}).
We start by proving two lemmas which describe explicitly the map $T_i\to V^{\otimes u}\otimes T_j$ associated to a path of length $u$ in $Q_G$. This will be done only when such a path satisfies some technical assumptions, since, as we will see in the proof of Theorem~\ref{thm_Psi(omegaG)cont_inPhi(omegaA)}, the general case can be traced back to that of paths of this kind. Finally, in Theorem~~\ref{thm_Psi(omegaG)=Phi(omegaA)-1fp} we will see that a converse of Theorem~\ref{thm_Psi(omegaG)cont_inPhi(omegaA)} holds for paths which contain at most one fixed point.

\begin{lem}\label{lem_comp_non-f.p.}
For $i_1,i_2,\dots,i_{u+1}\in\Z/m\Z$, let $a_1,\dots,a_u,p_1,\dots,p_u$ be integers which satisfy $r^{a_h}i_{h+1} \equiv i_h-r^{p_h} \pmod m$ for each $h=1,\dots,u$. Then:
\begin{enumerate}[label=(\arabic*)]
\item the composition
\[
T_{i_1} \xrightarrow{x^{i_1}_{p_1,a_1}} V \otimes T_{i_2} \xrightarrow{1_V \otimes x^{i_2}_{p_2,a_2}} \dots \xrightarrow{1_{V^{\otimes u-1}} \otimes x^{i_u}_{p_u,a_u}} V^{\otimes u} \otimes T_{i_{u+1}}
\]
sends $v_{i_1}$ to the element
$$\beta^{p_1}v_1 \otimes \beta^{a_1+p_2}v_1 \otimes \beta^{a_1+a_2+p_3}v_1 \otimes \dots \otimes \beta^{a_1+\dots+a_{u-1}+p_u}v_1 \otimes \beta^{a_1+\dots+a_u}v_{i_{u+1}};$$
\item we have
\begin{equation}\label{eq_sum_i_k}
r^{a_1+\dots+a_u}i_{u+1} \equiv i_1-r^{p_1}-r^{a_1+p_2}-\dots-r^{a_1+\dots+a_{u-1}+p_u} \pmod m.
\end{equation}
\end{enumerate}
\end{lem}
\begin{proof}
(1) We proceed by induction on $u$. The case $u=1$ is clear by the definition of the $x_{p,a}^i$'s. Now suppose $u>1$. By the induction hypothesis we have
\begin{align*}
&(1_{V^{\otimes u-1}} \otimes x^{i_u}_{p_u,a_u}) \circ \dots \circ (x^{i_1}_{p_1,a_1}) (v_{i_1}) = \\
&= (1_{V^{\otimes u-1}} \otimes x^{i_u}_{p_u,a_u}) (\beta^{p_1}v_1 \otimes \beta^{a_1+p_2}v_1 \otimes \dots \otimes \beta^{a_1+\dots+a_{u-2}+p_{u-1}}v_1 \otimes \beta^{a_1+\dots+a_{u-1}}v_{i_{u}}) \\
&= \beta^{p_1}v_1 \otimes \beta^{a_1+p_2}v_1 \otimes \dots \otimes \beta^{a_1+\dots+a_{u-2}+p_{u-1}}v_1 \otimes \beta^{a_1+\dots+a_{u-1}} x^{i_u}_{p_u,a_u}(v_{i_{u}}) \\
&= \beta^{p_1}v_1 \otimes \beta^{a_1+p_2}v_1 \otimes \dots \otimes \beta^{a_1+\dots+a_{u-2}+p_{u-1}}v_1 \otimes \beta^{a_1+\dots+a_{u-1}} (\beta^{p_u}_1v_1 \otimes \beta^{a_u}v_{i_{u+1}}) \\
&= \beta^{p_1}v_1 \otimes \beta^{a_1+p_2}v_1 \otimes \beta^{a_1+a_2+p_3}v_1 \otimes \dots \otimes \beta^{a_1+\dots+a_{u-1}+p_u}v_1 \otimes \beta^{a_1+\dots+a_u}v_{i_{u+1}}.
\end{align*}

(2) We proceed again by induction on $u$. The case $u=1$ is clear since $r^{a_1}i_2 \equiv i_1-r^{p_1}$. Now suppose that $r^{a_1+\dots+a_{u-1}}i_u \equiv i_1-r^{p_1}-r^{a_1+p_2}-\dots-r^{a_1+\dots+a_{u-2}+p_{u-1}}$. By adding on both sides $-r^{a_1+\dots+a_{u-1}+p_u}$ we see that is enough to prove that $r^{a_1+\dots+a_{u-1}} i_u - r^{a_1+\dots+a_{u-1}+p_u} \equiv r^{a_1+\dots+a_u} i_{u+1}$, but this follows immediately from the fact that $r^{a_u}i_{u+1} \equiv i_u-r^{p_u}$.
\end{proof}

\begin{lem}\label{lem_comp_f.p.}
Let $i_1,i_2,\dots,i_{u+1}\in\cal D$ and let $a_1,\dots,a_u,p_1,\dots,p_u$ be integers which satisfy $r^{a_j}i_{j+1} \equiv i_j-r^{p_j} \pmod m$ for each $j=1,\dots,u$. Suppose that the only fixed points among the $i_j$'s are $i_{f_1},\dots,i_{f_h}$ for some integers $f_1,\dots,f_h\in \{2,\dots,u\}$, and that $a_{f_j-1}=a_{f_j}=0$ for any $j$. We also assume that $f_{j+1}-f_j\geq 2$ for each $j=2,\dots,h-1$.

Fix integers $\ell_1,\dots,\ell_h\in\{0,\dots,s-1\}$. We introduce the following simplified notation: for each $j=1,\dots,u$, we set $i_j' := i_{f_j-1}$, $p_j' := p_{f_j-1}$, $i_j'' := i_{f_j}$ and $p_j'' := p_{f_j}$.

Consider the following path of length $u$:
\begin{equation}\label{eq_path_fp}
i_1 \xrightarrow{x_{p_1,a_1}^{i_1}} \dots \to i_j' \xrightarrow{x_{p_j',0}^{i_j'(\ell_j)}} i_j''^{(\ell_j)} \xrightarrow{x_{p_j'',0}^{(\ell_j)i_j''}} i_{f_j+1} \to \dots \xrightarrow{x_{p_u,a_u}^{i_u}} i_{u+1}.
\end{equation}
Then the corresponding composition $T_{i_1} \to \dots \to V^{\otimes u} \otimes T_{i_{u+1}}$ sends $v_{i_1}$ to
\begin{equation}\label{eq_composition_fp}
\frac{1}{s^h} \sum_{k_1,\dots,k_h=0}^{s-1} \left( \prod_{j=1}^h \epsilon_s^{-ti_{f_j}}\lambda_{i_{f_j},\ell_j}^{s-k_j} \right) \beta^{q_1}v_1 \otimes \dots \otimes \beta^{q_u}v_1 \otimes \beta^{q_{u+1}}v_{i_{u+1}},
\end{equation}
where $f_0:=0$ and
\begin{align*}
q_1 & := p_1, \\
\dots & \\
q_{f_1-1} & := a_1+\dots+a_{f_1-2}+p_{f_1-1}, \\
q_{f_1} & := k_1+a_1+\dots+a_{f_1-1}+p_{f_1}, \\
\dots & \\
q_{f_2-1} & := k_1+a_1+\dots+a_{f_2-2}+p_{f_2-1}, \\
q_{f_2} & := k_1+k_2+a_1+\dots+a_{f_2-1}+p_{f_2}, \\
\dots & \\
q_{f_h-1} & := k_1+\dots+k_{h-1}+a_1+\dots+a_{f_h-2}+p_{f_h-1}, \\
q_{f_h} & := k_1+\dots+k_h+a_1+\dots+a_{f_h-1}+p_{f_h}, \\
\dots \\
q_u & := k_1+\dots+k_h+a_1+\dots+a_{u-1}+p_u \\
q_{u+1} & := k_1+\dots+k_h+a_1+\dots+a_u.
\end{align*}
\end{lem}

\begin{rem}
Note that the assertion of the lemma makes sense also for $h=0$. In this case the element \eqref{eq_composition_fp} is equal to
$$\beta^{p_1}v_1 \otimes \beta^{a_1+p_2}v_1 \otimes \dots \otimes \beta^{a_1+\dots+a_{u-1}+p_u}v_1 \otimes \beta^{a_1+\dots+a_u}v_{i_{u+1}},$$
and so the result follows from Lemma~\ref{lem_comp_non-f.p.}.
\end{rem}

\begin{proof}[Proof of Lemma~\ref{lem_comp_f.p.}]
We proceed by induction on the number $h$ of fixed points. The case $h=0$ is clear by the previous remark.

Now fix an $h\geq1$ and suppose that the statement is valid for every path of any length which contains strictly less than $h$ fixed points. So let $\bf p$ be the path \eqref{eq_path_fp} and consider the subpath
$$ \bf p'\colon i_1 \to \dots \to i_{f_h-1}. $$
Clearly $\bf p'$ contains $h-1$ fixed points, and neither its starting nor its ending points are among them. So $\bf p'$ satisfies the induction hypothesis and we have that the composition $T_{i_1} \to \dots \to V^{\otimes f_h-2} \otimes T_{f_h-1}$ sends $v_{i_1}$ to
\begin{equation*}
\frac{1}{s^{h-1}} \sum_{k_1,\dots,k_{h-1}=0}^{s-1} \left( \prod_{j=1}^{h-1} \epsilon_s^{-ti_{f_j}}\lambda_{i_{f_j},\ell_j}^{s-k_j} \right) \beta^{q_1}v_1 \otimes \dots \otimes \beta^{q_{f_h-2}}v_1 \otimes \beta^{q'_{f_h-1}}v_{i_{f_h-1}},
\end{equation*}
where $q'_{f_h-1}:=k_1+\dots+k_{h-1}+a_1+\dots+a_{f_h-2}$.
If we apply to this element the map
$$ V^{\otimes (f_h-2)} \otimes T_{f_h-1} \xrightarrow{\left(\id_{V^{\otimes (f_h-2)}}\right) \otimes x_{p_h',0}^{i_h'(\ell_h)}} V^{\otimes (f_h-1)} \otimes T_{f_h}^{(\ell_h)} \xrightarrow{\left(\id_{V^{\otimes (f_h-1)}}\right) \otimes x_{p_h'',0}^{(\ell_h)i_h''}} V^{\otimes f_h} \otimes T_{f_h+1} $$
we get
\begin{equation}\label{eq_sommagrossa}
\frac{1}{s^{h-1}} \sum_{k_1,\dots,k_{h-1}=0}^{s-1} \left( \prod_{j=1}^{h-1} \epsilon_s^{-ti_{f_j}}\lambda_{i_{f_j},\ell_j}^{s-k_j} \right) \beta^{q_1}v_1 \otimes \dots \otimes \beta^{q_{f_h-2}}v_1 \otimes \beta^{q'_{f_h-1}} \theta(v_{i_{f_h-1}}),
\end{equation}
where $\theta := \left( \id_V \otimes x_{p_h'',0}^{(\ell_h)i_h''} \right) \circ x_{p_h',0}^{i_h'(\ell_h)}$. By Lemma~\ref{lem_calculations-arrows} we have that
$$ \theta(v_{i_{f_h-1}}) = \frac{\epsilon_s^{-ti_{f_h}}}{s} \sum_{k_h=0}^{s-1} \lambda_{i_{f_h},\ell_h}^{s-k_h} \beta^{p_{f_h-1}} v_1 \otimes \beta^{k_h+p_{f_h}} v_1 \otimes \beta^{k_h} v_{i_{f_h+1}},$$
so the \eqref{eq_sommagrossa} becomes
\begin{equation*}
\frac{1}{s^h} \sum_{k_1,\dots,k_h=0}^{s-1} \left( \prod_{j=1}^{h-1} \epsilon_s^{-ti_{f_j}}\lambda_{i_{f_j},\ell_j}^{s-k_j} \right) \beta^{q_1}v_1 \otimes \dots \otimes \beta^{q_{f_h-1}}v_1 \otimes \beta^{q_{f_h}}v_1\otimes \beta^{q'_{f_h-1}+k_h}v_{i_{f_h+1}},
\end{equation*}
since $q_{f_h-1} = k_1+\dots+k_{h-1}+a_1+\dots+a_{f_h-2}+p_{f_h-1} = q'_{f_h-1}+p_{f_{h-1}}$, $q_{f_h} = k_1+\dots+k_h+a_1+\dots+a_{f_h-1}+p_{f_h} = q'_{f_h-1}+a_{f_h-1}+k_h+p_{f_h} = q'_{f_h-1}+k_h+p_{f_h}$. Hence, if we go on calculating the exponents in the same way as in Lemma~\ref{lem_comp_non-f.p.}, we obtain the result we wanted to prove.
\end{proof}

Let us illustrate in an example how one can compute the element \eqref{eq_composition_fp} of Lemma~\ref{lem_comp_f.p.}.

\begin{example}\label{ex_sec6}
Let $G$ be the metacyclic group associated to $m=21$, $r=4$, $s=3$, $t=0$ (see Example~\ref{ex_s=3} for further details about this example). Choose $\cal D = \{0,4,7,8,9,12,13,14,17\}\subseteq\Z/21\Z$ as a set of representatives for the $G/A$-action. For an $\ell\in\{0,1,2\}$, consider the path
$$\bf p\colon 12 \xrightarrow{x^{12}_{0,1}} 8 \xrightarrow{x^{8(\ell)}_{0,0}} 7^{(\ell)} \xrightarrow{x^{(\ell)7}_{2,0}} 12$$
in $Q_G$, so, according to the notation of Lemma~\ref{lem_comp_f.p.}, we have $i_1=12$, $i_2=8$, $i_3=7$, $i_4=12$, $a_1=1$, $a_2=a_3=0$, $p_1=p_2=0$, $p_3=2$, $h=1$, $f_1=3$, $\lambda_{i_3,\ell}=\epsilon_3^{\ell+1}$. In this case the element \eqref{eq_composition_fp} becomes
\begin{align*}
 & \frac{1}{3} \sum_{k=0}^2 \lambda_{i_3,\ell} \beta^{p_1}v_1 \otimes \beta^{a_1+p_2}v_1 \otimes \beta^{k+a_1+a_2+p_3}v_1 \otimes \beta^{k+a_1+a_2+a_3}v_1 \\
= & \frac{1}{3} \left( \epsilon_3^{3(\ell+1)}v_1 \otimes \beta v_1 \otimes v_1 \otimes \beta v_1 + \epsilon_3^{2(\ell+1)}v_1 \otimes \beta v_1 \otimes \beta v_1 \otimes \beta^2 v_1 + \epsilon_3^{\ell+1}v_1 \otimes \beta v_1 \otimes \beta^2 v_1 \otimes v_1 \right).
\end{align*}
Note that if we apply the antisymmetrizer to it we get
$$\frac{1}{3} \epsilon_3^{\ell+1}v_1 \wedge \beta v_1 \wedge \beta^2 v_1 \otimes v_1.$$
This implies that the coefficient $c_{\bf p}$ of $\bf p$ in $\omega_G$ is equal to $\frac{1}{3}\epsilon_3^{\ell+1}$, and in particular we have that $\bf p\in\supp(\omega_G)$.
\end{example}

The following proposition, which will be used in the proof of Theorem~\ref{thm_Psi(omegaG)=Phi(omegaA)-1fp}, gives a necessary and sufficient condition for a path without fixed points to lie in $\supp(\omega_G)$.

\begin{prop}\label{prop_contrib_non-f.p.}
Let
$$\bf p \colon i_1 \xrightarrow{x^{i_1}_{p_1,a_1}} \dots \xrightarrow{x^{i_s}_{p_s,a_s}} i_{s+1}$$
be a path in $Q_G$ which contains no fixed points. Then $\bf p$ lies in the support of $\omega_G$ if and only if $\{p_1,a_1+p_2,\dots,a_1+\dots+a_{s-1}+p_s\}$ is a complete set of representatives of the integers modulo $s$.
\end{prop}
\begin{proof}
The path $\bf p$ can be identified with the composition
\[
T_{i_1} \to \dots \to V^{\otimes s} \otimes T_{i_{s+1}}.
\]
Composing this map with the antisymmetrizer we obtain a morphism
\[
T_{i_1} \to \det_V \otimes T_{i_{s+1}},
\]
which, by Lemma~\ref{lem_comp_non-f.p.}, sends $v_{i_1}$ to $\beta^{p_1}v_1 \wedge \beta^{a_1+p_2}v_1 \wedge \dots \wedge \beta^{a_1+\dots+a_{s-1}+p_s}v_1 \otimes \beta^{a_1+\dots+a_s}v_{i_{s+1}}$. Then it is clear that this element is different from zero if and only if the integers  $p_1,a_1+p_2,\dots,a_1+\dots+a_{s-1}+p_s$ are pairwise different modulo $s$.
\end{proof}

\begin{example}\label{ex_sec6_2}
Let us consider again the case where $m=21$, $r=4$, $s=3$, $t=0$. The path
$$12 \xrightarrow{x^{12}_{2,0}} 17 \xrightarrow{x^{17}_{0,1}} 4 \xrightarrow{x^{4}_{0,2}} 12$$
is in $\supp(\omega_G)$ because $\{p_1,a_1+p_2,a_1+a_2+p_3\}=\{2,0,1\}$.

On the other end, we have that
$$12 \xrightarrow{x^{12}_{2,0}} 17 \xrightarrow{x^{17}_{2,2}} 4 \xrightarrow{x^{4}_{0,2}} 12$$
is not in $\supp(\omega_G)$, since $\{p_1,a_1+p_2,a_1+a_2+p_3\}=\{2,2,2\}$.
\end{example}

Proposition~\ref{prop_contrib_non-f.p.} tells us exactly when a path containing no fixed points is in the support of $\omega_G$. However, in view of Lemma~\ref{lem_comp_f.p.}, obtaining a similar statement for paths containing fixed points seems more difficult and we will not prove such a result. Nevertheless, we will be able to show that every path in the support of $\omega_G$ comes, in a certain sense, from a path in the support of $\omega_A$, and this will be sufficient for our purposes. In order to make this more precise, we give the following definition.

\begin{defn}\label{def_Qtilde}
Let $\tilde{Q}_G$ be the quiver defined in the following way. Its set of vertices is $\cal D$, while for the arrows $i \to j$ we have the following three possibilities:
\begin{enumerate}[label=(\arabic*)]
\item $i,j\in \cal D \smallsetminus \cal F$: in this case the arrows between $i$ and $j$ in $\tilde{Q}_G$ are the same as in $Q_G$;
\item $i\in \cal D \smallsetminus \cal F$ and $j\in \cal F$: we put an arrow $x_{p,0}^i$ in $\tilde{Q}_G$ whenever $j\equiv i-r^p$ for some $p$;
\item $i\in \cal F$ and $j\in \cal D \smallsetminus \cal F$: we put an arrow $x_{p,0}^i$ in $\tilde{Q}_G$ whenever $j\equiv i-r^p$ for some $p$.
\end{enumerate}
\end{defn}
We will see later that $\tilde{Q}_G$ can be seen as the quotient of $Q_A$ by an action of $G/A$.

We now define two morphisms of quivers
$$ \Phi\colon Q_A \to \tilde{Q}_G, \qquad \Psi\colon Q_G \to \tilde{Q}_G. $$

For $i\in(Q_A)_0$ we put $\Phi(i)=\ul i$. Given an arrow $x^i_q\colon i\to i-r^q$ in $Q_A$, we set $\Phi(x^i_q)=x^{\ul i}_{p,a}\colon \ul i \to \ul{i-r^q}$, where
\[
(p,a) = \left\{ \begin{array}{cc}
(q-\kappa_i, \kappa_{i-r^q}-\kappa_i) & \text{ if } \ul i, \ul{i-r^q} \in \cal D \smallsetminus \cal F; \\
(q-\kappa_i, 0) & \text{ if } \ul i \in \cal D \smallsetminus \cal F,\; \ul{i-r^q} \in \cal F; \\
(q-\kappa_{i-r^q}, 0) & \text{ if } \ul i \in \cal F,\; \ul{i-r^q} \in \cal D \smallsetminus \cal F.
\end{array} \right.
\]

We define $\Psi(i)=i$ if $i\in \cal D \smallsetminus \cal F$ and $\Psi(i^{(\ell)})=i$ if $i\in\cal F$, $0\leq\ell\leq s-1$. Moreover we put $\Psi(x^i_{p,a})=x^i_{p,a}$, $\Psi(x^{i(\ell)}_{p,0})=x^i_{p,0}$, $\Psi(x^{(\ell)i}_{p,0})=x^i_{p,0}$ whenever the notation makes sense. So basically we can consider $\Psi$ simply as the map which ``forgets'' about the splitting of fixed points (see Figure~\ref{fig_Psi}).

Consider the subquivers $Q_G \smallsetminus \cal F$ and $\tilde{Q}_G \smallsetminus \cal F$ of, respectively, $Q_G$ and $\tilde{Q}_G$, which are obtained by removing the fixed points and the arrows adjacent to them. Then it is clear that $\Psi \mid_{Q_G \smallsetminus \cal F} \colon Q_G \smallsetminus \cal F \to \tilde{Q}_G \smallsetminus \cal F$ is an isomorphism. So the fact that we used the same names to indicate vertices and arrows in these subquivers will cause no confusion: actually, we will often treat them as if they were the same quiver.

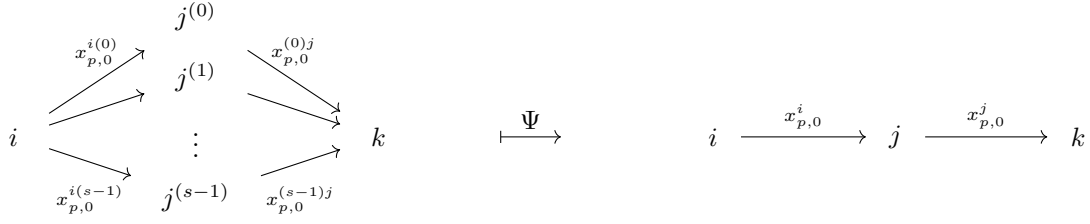
\begin{figure}
\begin{tikzpicture}[scale=0.8,inner sep=4mm]

\node (A) at (-3,0){$i$};
\node (B0) at (0,2){$j^{(0)}$};
\node (B1) at (0,1){$j^{(1)}$};
\node (B2) at (0,0){$\vdots$};
\node (B3) at (0,-1){$j^{(s-1)}$};
\node (C) at (3,0){$k$};

 \draw [->] (A) -> (B0) node[midway,above,yshift=-2mm] {{\scriptsize $x^{i(0)}_{p,0}$}};
 \draw [->] (A) -> (B1);
 \draw [->] (A) -> (B3) node[midway,below,yshift=1mm] {{\scriptsize $x^{i(s-1)}_{p,0}$}};

 \draw [->] (B0) -> (C) node[midway,above,yshift=-2mm] {{\scriptsize $x^{(0)j}_{p,0}$}};
 \draw [->] (B1) -> (C);
 \draw [->] (B3) -> (C) node[midway,below,yshift=1mm] {{\scriptsize $x^{(s-1)j}_{p,0}$}};

\begin{scope}[xshift=11.5cm,inner sep=3mm]
\node (D) at (-3,0){$i$};
\node (E) at (0,0){$j$};
\node (F) at (3,0){$k$};

 \draw [->] (D) -> (E) node[midway,above,yshift=-2mm] {{\scriptsize $x^i_{p,0}$}};
 \draw [->] (E) -> (F) node[midway,above,yshift=-2mm] {{\scriptsize $x^j_{p,0}$}};
\end{scope}

\begin{scope}[xshift=5cm]
 \draw [|->] (0,0) -> (1,0) node[midway,above,yshift=-3mm] {$\Psi$};
\end{scope}

\end{tikzpicture}
\caption{The local behaviour of $\Psi$ at fixed points.}\label{fig_Psi}
\end{figure}

Recall that $G/A$ acts on the vertices of $Q_A$ via the automorphism $\varphi$ given by the multiplication by $r$. We can extend this to an automorphism of $Q_A$ by setting $\varphi(x^i_q)=x^{ri}_{q+1}$.

Consider now the orbit quiver $Q_A/(G/A)$. We will denote by $[i]$ the orbit of $i\in(Q_A)_0$ and by $[x^i_q]$ the orbit of $x^i_q\in(Q_A)_1$.

\begin{prop}\label{prop_Q_A/(G/A)iso-tilde(Q)_G}
The morphism $\Phi$ induces an isomorphism of quivers $\tilde{\Phi}\colon Q_A/(G/A) \to \tilde{Q}_G$.
\end{prop}
\begin{proof}
For each $i\in(Q_A)_0$ we have $\Phi(\varphi(i))=\Phi(ri)=\ul{ri}=\ul{i}$, so $\Phi$, as a map between vertices, factors through the action of $G/A$. Hence we get a map $(Q_A/(G/A))_0 \to (\tilde{Q}_G)_0$, $[i]\mapsto \Phi(i)=\ul{i}$, which is obviously a bijection because $(\tilde{Q}_G)_0=\cal D$ is a set of representatives of the $G/A$-orbits.

Now we consider the arrows. For each $x^i_q\in(Q_A)_1$ we have $\Phi(\varphi(x^i_q)) = \Phi(x^{ri}_{q+1}) = x^{\ul{ri}}_{p,a} = x^{\ul{i}}_{p,a}$, where
\[
(p,a) = \left\{ \begin{array}{cc}
(q+1-\kappa_{ri}, \kappa_{r(i-r^q)}-\kappa_{ri}) & \text{ if } \ul{ri}, \ul{r(i-r^q)} \in \cal D \smallsetminus \cal F; \\
(q+1-\kappa_{ri}, 0) & \text{ if } \ul{ri} \in \cal D \smallsetminus \cal F,\; \ul{r(i-r^q)} \in \cal F; \\
(q+1-\kappa_{r(i-r^q)}, 0) & \text{ if } \ul{ri} \in \cal F,\; \ul{r(i-r^q)} \in \cal D \smallsetminus \cal F.
\end{array} \right.
\]
It is easy to check that $\ul{ri}=\ul{i}$, $\ul{r(i-r^q)}=\ul{i-r^q}$, $\kappa_{ri}=\kappa_i+1$, $\kappa_{r(i-r^q)}=\kappa_{i-r^q}+1$, so it becomes clear from the definition of $\Phi$ that $\Phi(\varphi(x^i_q))=\Phi(x^i_q)$. Hence $\Phi$ factors through the action of $G/A$ and we get a map $\tilde{\Phi}\colon (Q_A/(G/A))_1 \to (\tilde{Q}_G)_1$, $[x^i_q]\mapsto\Phi(x^i_q)$.

To show that this map is surjective, take an arrow $x^i_{q,b}\colon i \to j$ in $(\tilde{Q}_G)_1$, so $r^b j \equiv i-r^q$ holds. Note that $\kappa_i=0$ and $\kappa_{i-r^q}=b$. By definition, $\Phi(x^i_q)=x^i_{p,a}$, where
\[
(p,a) = \left\{ \begin{array}{cc}
(q-\kappa_i, \kappa_{i-r^q}-\kappa_i)=(q,b) & \text{ if } i, j \in \cal D \smallsetminus \cal F; \\
(q-\kappa_i, 0)=(q,b) & \text{ if } i \in \cal D \smallsetminus \cal F,\; j \in \cal F; \\
(q-\kappa_{i-r^q}, 0)=(q,b) & \text{ if } i \in \cal F,\; j \in \cal D \smallsetminus \cal F.
\end{array} \right.
\]
This means that $\Phi(x^i_q)=x^i_{q,b}$, so we showed that $\Phi$ (and consequently $\tilde{\Phi}$) is surjective on arrows.

Now we prove that $\tilde{\Phi}$ is injective. Let $x^i_p$, $x^j_q$ be two arrows in $Q_A$ and suppose that $\Phi(x^i_p)=\Phi(x^j_q)$: we want to show that $x^i_p$ and $x^j_q$ lie in the same $G/A$-orbit. Clearly we have that $\ul i=\ul j$ and $\ul{i-r^p}=\ul{j-r^q}$, so there exist $a,b\in \{0,\dots,s-1\}$ such that $j \equiv r^ai \pmod m$ and $j-r^q \equiv r^b(i-r^p) \pmod m$ (we take $a=0$ and $b=0$ respectively when $i$ and $i-r^p$ are fixed points). Note that $r^{\kappa_i}\ul{j} \equiv r^{\kappa_i}\ul{i} \equiv i \equiv r^{-a}j$, so $\kappa_j=\kappa_i-a$. Similarly we have $\kappa_{j-r^q}=\kappa_{i-r^p}-b$. By definition of $\Phi$, we write $\Phi(x^i_p)=x^{\ul i}_{p',a'}$ and $\Phi(x^j_q)=x^{\ul j}_{q',b'}$, where
\[
(p',a') = \left\{ \begin{array}{cc}
(p-\kappa_i, \kappa_{i-r^p}-\kappa_i) & \text{ if } \ul i, \ul{i-r^p} \in \cal D \smallsetminus \cal F; \\
(p-\kappa_i, 0) & \text{ if } \ul i \in \cal D \smallsetminus \cal F,\; \ul{i-r^p} \in \cal F; \\
(p-\kappa_{i-r^p}, 0) & \text{ if } \ul i \in \cal F,\; \ul{i-r^p} \in \cal D \smallsetminus \cal F
\end{array} \right.
\]
and
\[
(q',b') = \left\{ \begin{array}{cc}
(q-\kappa_j, \kappa_{j-r^q}-\kappa_j) & \text{ if } \ul j, \ul{j-r^q} \in \cal D \smallsetminus \cal F; \\
(q-\kappa_j, 0) & \text{ if } \ul j \in \cal D \smallsetminus \cal F,\; \ul{j-r^q} \in \cal F; \\
(q-\kappa_{j-r^q}, 0) & \text{ if } \ul j \in \cal F,\; \ul{j-r^q} \in \cal D \smallsetminus \cal F.
\end{array} \right.
\]
By hypothesis $(p',a') = (q',b')$. Hence, in the first two cases we have $p-\kappa_i=q-\kappa_j=q-\kappa_i+a$, so $q=p+a$. This means that $x^j_q=x^{r^ai}_{p+a}=\varphi^a(x^i_p)$ and so $[x^i_p]=[x^j_q]$. In the last case we have $p-\kappa_{i-r^p}=q-\kappa_{j-r^q}=q-\kappa_{i-r^p}+b$, so $q=p+b$. Moreover, both $i$ and $j$ are fixed points, so we can write $j \equiv r^bi \pmod m$. Hence we have that $x^j_q=x^{r^bi}_{p+b}=\varphi^b(x^i_p)$ and so $[x^i_p]=[x^j_q]$.
\end{proof}

\begin{example}
We shall illustrate the behaviour of $\Phi$ in the case of Example~\ref{ex_sec6}. The quivers $Q_A$, $Q_G$ and $\tilde{Q}_G$ are depicted, respectively, in Figures~\ref{fig_quiver-s=3}, \ref{fig_cut C_1^1} and \ref{fig_quiver cut C^1_1}. Recall that their set of vertices are
$$(Q_A)_0 = \Z/21\Z,\quad (Q_G)_0 = (\tilde{Q}_G)_0 = \{0,4,7,8,9,12,13,14,17\}.$$
The map $\Phi$ is given on vertices by:
\begin{eqnarray*}
\Phi(0)=0, & \Phi(1)=\Phi(4)=\Phi(16)=4, & \Phi(2)=\Phi(8)=\Phi(11)=8, \\
\Phi(7)=7, & \Phi(9)=\Phi(15)=\Phi(18)=9, & \Phi(3)=\Phi(6)=\Phi(12)=12, \\
\Phi(14)=14, & \Phi(10)=\Phi(13)=\Phi(19)=13, & \Phi(5)=\Phi(17)=\Phi(20)=17.
\end{eqnarray*}
We now describe how $\Phi$ behaves only on some of the arrows of $Q_G$:
\begin{eqnarray*}
\Phi(x_0^{17}\colon 17 \to 16)=\Phi(x_1^{5}\colon 5 \to 1)=\Phi(x_2^{20}\colon 20 \to 4)=x_{0,1}^{17}\colon 17 \to 4, \\
\Phi(x_0^{5}\colon 5 \to 4)=\Phi(x_1^{20}\colon 20 \to 16)=\Phi(x_2^{17}\colon 17 \to 1)=x_{2,2}^{17}\colon 17 \to 4, \\
\Phi(x_0^{0}\colon 0 \to 20)=\Phi(x_1^{0}\colon 0 \to 17)=\Phi(x_2^{0}\colon 0 \to 5)=x_{1,0}^{0}\colon 0 \to 17.
\end{eqnarray*}
\end{example}

There is a natural way to define an automorphism of $\tilde{Q}_G$ which is compatible with both the twists of $Q_A$ and $Q_G$ via the morphisms $\Phi$ and $\Psi$. More precisely, we have the following proposition/definition.

\begin{prop}\label{prop_def-tau}
For each $i\in (Q_A)_0$ and each $x^i_q\in (Q_A)_1$, set $\tau([i]):=[\tau(i)]=[i+c]$ and $\tau([x^i_q]):=[\tau(x^i_q)]=[x^{i+c}_q]$. Then this assignments induce a well defined automorphism of $Q_A/(G/A)$, which in turn induces an automorphism of $\tilde{Q}_G$ under the isomorphism of Proposition~\ref{prop_Q_A/(G/A)iso-tilde(Q)_G}: by an abuse of notation, we will again denote both these maps by $\tau$. Moreover these maps are compatible with the other twists, in the sense that the following diagram commutes:
\[
\begin{tikzcd}
Q_{A} \arrow[r] \arrow[d,"\tau"] & Q_{A}/(G/A) \arrow[r,"\tilde{\Phi}"] \arrow[d,"\tau"] & \tilde{Q}_{G} \arrow[d,"\tau"] & Q_{G} \ar[l,swap,"\Psi"] \arrow[d,"\tau'"] \\
Q_{A} \arrow[r] & Q_{A}/(G/A) \arrow[r,"\tilde{\Phi}"] & \tilde{Q}_G & Q_G \ar[l,swap,"\Psi"]
\end{tikzcd}
\]


\end{prop}
\begin{proof}
For each $i\in (Q_A)_0$ and each $x^i_q\in (Q_A)_1$, we have $\tau([ri]) = [ri+c] = [ri+rc] = [i+c] = \tau([i])$ and similarly $\tau([x^{ri}_{q+1}]) = [x^{ri+rc}_{q+1}] = \tau([x^i_q])$, so $\tau$ is well defined on $Q_A/(G/A)$.

For the commutativity of the diagram, note that the two squares on the left commute by definition, so we only have to deal with the square on the right. For $i\in \cal D \smallsetminus \cal F$ we have $\tau(\Psi(i))=\tau(i)={i+c}=\Psi({i+c})=\Psi(\tau'(i))$, while for $i\in \cal F$ we have  $\tau(\Psi(i^{(\ell)}))=\tau(i)=i+c=\Psi((i+c)^{(\ell+d_s)})=\Psi(\tau'(i^{(\ell)}))$. Hence the result follows.
\end{proof}

\begin{thm}\label{thm_Psi(omegaG)cont_inPhi(omegaA)}
Let $\bf p$ be a path in $Q_G$ which lies in $\supp(\omega_G)$. Then there exists a path $\tilde{\bf p}\in\supp(\omega_A)$ such that $\Psi(\bf p)=\Phi(\tilde{\bf p})$.
\end{thm}
\begin{proof}
We will first consider a particular case and then the general one.

\paragraph*{\textit{Case 1. }} Suppose that $\bf p$ is in the form described in Lemma~\ref{lem_comp_f.p.}, so we retain all the notation from there. The fact that $\bf p$ is in $\supp(\omega_G)$ implies that there exist $k_1,\dots,k_h$ such that the integers $q_1,\dots,q_s$ are pairwise different modulo $s$. Otherwise, there would be two linearly dependent tensor factors for each element of the sum \eqref{eq_composition_fp}, so we would get zero after applying the antisymmetrizer. To simplify the notation, for the rest of the proof we will set $N_k := r^{q_1}+\dots+r^{q_k}$ for each $k=1,\dots,s$ and $N_0:=0$. We also set $f_0:=0$ and $f_{h+1}:=s+1$. Note that what follows will make sense also for $h=0$.

Define $\tilde{\bf p}$ to be the path
$$\tilde{\bf p}\colon i_1 \to i_1-N_1 \to \dots \to i_1-N_s$$
in $Q_A$. Since we assumed the $q_j$'s to be pairwise different, by Proposition~\ref{prop_quiver_A} we deduce that $\tilde{\bf p}\in\supp(\omega_A)$. So it remains to show that $\Psi(\bf p)=\Phi(\tilde{\bf p})$.

Firstly, we claim that
\begin{equation}\label{eq_claim_proof_path}
i_1-N_k \equiv r^{q_{k+1}-p_{k+1}} i_{k+1}
\end{equation}
for all $k=0,\dots,s$. We will prove this by induction on $k$, the case $k=0$ being clear. Suppose that $k\geq 1$, then there exists a $j\in\{1,\dots,h+1\}$ such that $f_{j-1} \leq k \leq f_j-1$. If $k \neq f_j-1$, we have that $q_{k+1}-p_{k+1}=k_1+\dots+k_{j-1}+a_1+\dots+a_k$, so by induction hypothesis we obtain
\begin{align*}
r^{q_{k+1}-p_{k+1}} i_{k+1} & = r^{k_1+\dots+k_{j-1}+a_1+\dots+a_{k-1}} (r^{a_k} i_{k+1}) \equiv r^{q_k-p_k} (i_k-r^{p_k}) \\
 & \equiv i_1-N_{k-1}-r^{q_k} = i_1-N_k.
\end{align*}
If $k = f_j-1$, using the fact that $i_{k+1}$ is a fixed point, we have
\begin{align*}
r^{q_{k+1}-p_{k+1}} i_{k+1} & \equiv r^{q_k-p_k} i_{k+1} \equiv r^{q_k-p_k} (i_k-r^{p_k}) \\
 & \equiv i_1-N_{k-1}-r^{q_k} = i_1-N_k.
\end{align*}
Hence the claim follows.

Now if we apply $\Phi$ to $\tilde{\bf p}$ we obtain the path $\Phi(\tilde{\bf p})\colon i_1 \to \dots \to i_{s+1}$, because $\ul{i_1-N_k}=i_{k+1}$ by the equation \eqref{eq_claim_proof_path}. By definition of $\Phi$, the arrows in this path are given by $x^{i_k}_{p'_k,a'_k}\colon \ul{i_1-N_{k-1}} \to \ul{i_1-N_k}$, where
\[
(p'_k,a'_k) = \left\{ \begin{array}{cc}
(q_k-\kappa_{i_1-N_{k-1}},\kappa_{i_1-N_k}-\kappa_{i_1-N_{k-1}}) & \text{ if } \ul{i_1-N_{k-1}},\; \ul{i_1-N_k} \in \cal D \smallsetminus \cal F; \\
(q_k-\kappa_{i_1-N_{k-1}},0) & \text{ if } \ul{i_1-N_{k-1}} \in \cal D \smallsetminus \cal F,\; \ul{i_1-N_k} \in \cal F; \\
(q_k-\kappa_{i_1-N_k},0) & \text{ if } \ul{i_1-N_{k-1}} \in \cal F,\; \ul{i_1-N_k} \in \cal D \smallsetminus \cal F.
\end{array} \right.
\]
We may observe that in all cases we have $(p'_k,a'_k)=(p_k,a_k)$, because the \eqref{eq_claim_proof_path} implies that, for all $k$,
\[
\kappa_{i_1-N_k} = \left\{ \begin{array}{cc}
q_{k+1}-p_{k+1} & \text{ if } i_1-N_k \text{ is not a fixed point,} \\
0 & \text{ otherwise.}
\end{array} \right.
\]
Hence $\Psi(\bf p)=\Phi(\tilde{\bf p})$.

\paragraph*{\textit{Case 2. }} We now consider the general case. Let $\bf p\in\supp(\omega_G)$ and write $\bf p=\bf p_1\cdots \bf p_s$. By Proposition~\ref{prop_r_not_equiv_1} we can assume that no arrow $\bf p_i$ both starts and ends with a fixed point. Note that $s(\bf p)=\tau(t(\bf p))$, so $s(\bf p)$ is a fixed point if and only if $t(\bf p)$ is. If $s(\bf p)$ is not a fixed point, then $\bf p$ is in the form discussed in Case 1. Otherwise, suppose that $s(\bf p)=s(\bf p_s)$ is a fixed point and consider the path $\bf q:=\tau'(\bf p_s)\bf p_1\cdots \bf p_{s-1}$. By Lemma~\ref{lem_cyclic-permutation}, we have that $\bf q$ lies in $\supp(\omega_G)$; moreover neither $s(\bf q)=s(\bf p_{s-1})=t(\bf p_s)$ nor $t(\bf q)=t(\tau'(\bf p_s))=\tau(t(\bf p_s))$ are fixed points, because otherwise $\bf p_s$ would connect two fixed points. So $\bf q$ is as in case 1, and we can write $\Psi(\bf q)=\Phi(\tilde{\bf q})$ for some path $\tilde{\bf q}=\tilde{\bf q}_1\cdots\tilde{\bf q}_s \in \supp(\omega_A)$.
Now note that, by Proposition~\ref{prop_def-tau}, we have $\Psi(\bf q) = \Psi(\tau'(\bf p_s))\Psi(\bf p_1)\cdots \Psi(\bf p_{s-1}) = \tau(\Psi(\bf p_s))\Psi(\bf p_1)\cdots \Psi(\bf p_{s-1})$, so $\tau(\Psi(\bf p_s))=\Phi(\tilde{\bf q}_1)$, $\Psi(\bf p_i)=\Phi(\tilde{\bf q}_{i+1})$ for all $i=1,\dots,s-1$ and in particular $\Psi(\bf p_s) = \tau^-(\Phi(\tilde{\bf q}_1)) = \Phi(\tau^-(\tilde{\bf q}_1))$. Now set $\tilde{\bf p}:=\tilde{\bf q}_2\cdots\tilde{\bf q}_s\tau^-(\tilde{\bf q}_1)$: then $\Psi(\bf p)=\Phi(\tilde{\bf p})$ and, by Lemma~\ref{lem_cyclic-permutation}, $\tilde{\bf p}\in \supp(\omega_A)$, hence the result follows.
\end{proof}

The previous theorem tells us that $\Psi(\supp(\omega_G)) \subseteq \Phi(\supp(\omega_A))$, but we do not know if the equality holds. The following proposition shows that this happens at least when we restrict the support to paths containing at most one fixed point.

\begin{thm}\label{thm_Psi(omegaG)=Phi(omegaA)-1fp}
Let
$$\bf p \colon i \to i-r^{q_1} \to \dots \to i-r^{q_1}-\dots-r^{q_s}$$
be a path in $Q_A$ which passes through at most one fixed point and suppose that $\bf p\in\supp(\omega_A)$. Then there exists a path $\tilde{\bf p}$ in $Q_G$ such that $\tilde{\bf p}\in\supp(\omega_G)$ and $\Phi(\bf p)=\Psi(\tilde{\bf p})$.
\end{thm}
\begin{proof}
We consider separately the cases where $\bf p$ contains zero or one fixed point.

\paragraph*{\textit{Case 1. }} Suppose that $\bf p$ has no fixed points and consider the path $\Phi(\bf p)$ in $\tilde{Q}_G$. It is clear that we can lift it to a path
$$\tilde{\bf p}\colon \ul{i} \to \ul{i-r^{q_1}} \to \dots \to \ul{i-r^{q_1}-\dots-r^{q_s}}$$
in $Q_G$, since $\Psi$ acts as the identity outside the fixed points. Thus $\Phi(\bf p)=\Psi(\tilde{\bf p})$.

We are left to show that $\tilde{\bf p}\in\supp(\omega_G)$. In the following we will use the notation $i_h := i-r^{q_1}-\dots-r^{q_{h-1}}$. By definition of $\Phi$, the arrows of $\tilde{\bf p}$ are given by $x^{\ul{i_h}}_{p_h,a_h} \colon \ul{i_h} \to \ul{i_{h+1}}$ for each $h=1,\dots,s$, where $p_h:=q_h-\kappa_{i_h}$ and $a_h:=\kappa_{i_{h+1}}-\kappa_{i_h}$. Hence we have that
$$\{p_1,a_1+p_2,\dots,a_1+\dots+a_{s-1}+p_s\} = \{q_1-\kappa_i,q_2-\kappa_i,\dots,q_s-\kappa_i\}.$$
Note that $\bf p\in\supp(\omega_A)$ implies that $q_1,\dots,q_s$ are pairwise different modulo $s$, so the same is true for $p_1,a_1+p_2,\dots,a_1+\dots+a_{s-1}+p_s$. Hence it follows from Proposition~\ref{prop_contrib_non-f.p.} that the path $\tilde{\bf p}$ lies in the support of $\omega_G$.

\paragraph*{\textit{Case 2. }} Now we consider the case where $\bf p$ has exactly one fixed point. We keep the notation $i_h := i-r^{q_1}-\dots-r^{q_{h-1}}$ as in the previous case.

By an argument similar to the one in Case 2 of the proof of Theorem~\ref{thm_Psi(omegaG)cont_inPhi(omegaA)}, we can assume that the only fixed point in $\bf p$ is $i_s$. It is clear that $\Phi(\bf p)$ can be lifted to a path
$$\tilde{\bf p}\colon \ul{i_1} \to \ul{i_2} \to \dots \to \ul{i_{s-1}} \to \ul{i_s}^{(\ell)} \to \ul{i_{s+1}}$$
in $Q_G$, for an integer $0\leq\ell\leq s-1$. Thus $\Phi(\bf p)=\Psi(\tilde{\bf p})$ and the arrows in $\tilde{\bf p}$ are given by $x_{p_j,a_j}^{\ul{i_j}} \colon \ul{i_j} \to \ul{i_{j+1}}$ for $j=1,\dots,s-2$, $x_{p_{s-1},a_{s-1}}^{\ul{i_{s-1}}(\ell)} \colon \ul{i_{s-1}} \to \ul{i_s}^{(\ell)}$, $x_{p_s,a_s}^{(\ell)\ul{i_s}} \colon \ul{i_s}^{(\ell)} \to \ul{i_{s+1}}$, where
\[
(p_j,a_j) = \left\{ \begin{array}{cc}
(q_j-\kappa_{i_j}, \kappa_{i_{j+1}}-\kappa_{i_j}) & \text{ if } 1\leq j\leq s-2, \\
(q_j-\kappa_{i_j}, 0) & \text{ if } j=s-1, \\
(q_j-\kappa_{i_{j+1}}, 0) & \text{ if } j=s.
\end{array} \right.
\]
The path $\tilde{\bf p}$ induces a morphism $T_{\ul{i_1}} \to \bigwedge^s V \otimes T_{\ul{i_{s+1}}}$ which, by Lemma~\ref{lem_comp_f.p.}, sends $v_{\ul{i_1}}$ to
\[
\frac{1}{s} \sum_{k=0}^{s-1} \epsilon_s^{-t \ul{i_s}}\lambda_{\ul{i_s},\ell}^{s-k} \beta^{p_1}v_1 \wedge \beta^{p_1+a_2}v_1 \wedge \dots \wedge \beta^{a_1+\dots+a_{s-2}+p_{s-1}}v_1 \wedge \beta^{k+a_1+\dots+a_{s-1}+p_s}v_1 \otimes \beta^{k+a_1+\dots+a_s}v_{\ul{i_{s+1}}}.
\]
For all $j=0,\dots,s-2$ we have $a_1+\dots+a_j+p_{j+1} = q_{j+1}-\kappa_{i_1}$, while $k+a_1+\dots+a_{s-1}+p_s = k-\kappa_{i_1}+\kappa_{i_{s-1}}+q_s-\kappa_{i_{s+1}}$ and $k+a_1+\dots+a_s = k-\kappa_{i_1}+\kappa_{i_{s-1}}$. So the previous sum becomes
\[
\frac{1}{s} \sum_{k=0}^{s-1} \epsilon_s^{-t \ul{i_s}}\lambda_{\ul{i_s},\ell}^{s-k} \beta^{q_1-\kappa_{i_1}}v_1 \wedge \beta^{q_2-\kappa_{i_1}}v_1 \wedge \dots \wedge \beta^{q_{s-1}-\kappa_{i_1}}v_1 \wedge \beta^{k-\kappa_{i_1}+\kappa_{i_{s-1}}+q_s-\kappa_{i_{s+1}}}v_1 \otimes \beta^{k-\kappa_{i_1}+\kappa_{i_{s-1}}}v_{i_1}.
\]
Note that since $\bf p\in\supp(\omega_A)$ we must have $i_{s+1} \equiv i_1+c \pmod m$, and moreover $\ul{i_{s+1}} \equiv \ul{i_1}+c$ because $\cal D$ is closed under the twist, by Assumption~\ref{ass_repclosed}: this implies that $\kappa_{i_{s+1}}=\kappa_{i_1}$. Hence in the above sum all terms are zero except the one where $k=\kappa_{i_1}-\kappa_{i_{s-1}}$, and thus we obtain
\[
\frac{1}{s} \epsilon_s^{-t \ul{i_s}}\lambda_{\ul{i_s},\ell}^{s-\kappa_{i_1}+\kappa_{i_{s-1}}} \beta^{q_1-\kappa_{i_1}}v_1 \wedge \beta^{q_2-\kappa_{i_1}}v_1 \wedge \dots \wedge \beta^{q_{s-1}-\kappa_{i_1}}v_1 \wedge \beta^{q_s-\kappa_{i_1}}v_1 \otimes v_{i_1}.
\]
This is clearly a non-zero element, because $q_1,\dots,q_s$ are pairwise different modulo $s$, and so the coefficient in $\omega_G$ corresponding to the path $\tilde{\bf p}$ is non-zero. Hence the result follows.
\end{proof}

\begin{cor}\label{cor_supp-s=2,3}
For $s=2,3$ we have $\Psi(\supp(\omega_G)) = \Phi(\supp(\omega_A))$.
\end{cor}
\begin{proof}
Clearly in these cases every path of length $s$ can contain at most one fixed point, hence the result follows immediately from Theorem~\ref{thm_Psi(omegaG)=Phi(omegaA)-1fp}.
\end{proof}

\begin{example}
Let us retain the case of Example~\ref{ex_sec6}.
By Corollary~\ref{cor_supp-s=2,3} we can describe explicitly the paths in $\supp(\omega_G)$. We know that $\supp(\omega_A)$ consists in all the cyclic permutations of paths of type $i \xrightarrow{x^i_{0}} i-1 \xrightarrow{x^{i-1}_1} i-5 \xrightarrow{x^{i-5}_2} i$ and $i \xrightarrow{x^i_{0}} i-1 \xrightarrow{x^{i-1}_2} i-17 \xrightarrow{x^{i-17}_1} i$. Hence $\supp(\omega_G)$ is made of the paths which are induced by these ones via the procedure described in the proof of Theorem~\ref{thm_Psi(omegaG)=Phi(omegaA)-1fp}.

For example, given an $\ell\in\{0,1,2\}$, the path $12 \xrightarrow{x^{12}_{0}} 11 \xrightarrow{x^{11}_{1}} 7 \xrightarrow{x^7_{2}} 12$ in $\supp(\omega_A)$ induces a path
$$\bf p\colon 12 \xrightarrow{x^{12}_{0,1}} 8 \xrightarrow{x^{8(\ell)}_{0,0}} 7^{(\ell)} \xrightarrow{x^{(\ell)7}_{2,0}} 12$$
in $\supp(\omega_G)$ (note that this was already shown by a direct computation in Example~\ref{ex_sec6}). The reader should be careful that this is not the same path which is induced by $12 \xrightarrow{x^{12}_{1}} 8 \xrightarrow{x^{8}_{0}} 7 \xrightarrow{x^7_{2}} 12$, since the arrows $x^{12}_{0}\colon 12 \to 11$ and $x^{12}_{1}\colon 12 \to 8$ in $Q_A$ yield two different arrows from $12$ to $8$ in $Q_G$.
\end{example}

\subsection{Gradings of $Q_G$}

We will now illustrate a way to obtain gradings on $Q_G$ which make $\omega_G$ homogeneous.

\begin{prop}\label{prop_grading-A-->grading-G}
Let $g_A$ be a grading on $Q_A$ such that $\omega_A$ is homogeneous of degree $a$ and the morphism of quivers $\Phi$ is $g_A$-gradable. Then there exists a grading $g_G$ on $Q_G$ such that $\omega_G$ is homogeneous of degree $a$ with respect to it.
\end{prop}
\begin{proof}
By Proposition~\ref{prop_Q_A/(G/A)iso-tilde(Q)_G} we have that $\Phi$ is surjective on arrows: this, together with the fact that $\Phi$ is $g_A$-gradable, implies that we can define the grading $\Phi_*g_A$ on $\tilde{Q}_G$ (see Definition~\ref{def_graded-quiver-morphisms}). We now define a grading on $Q_G$ by $g_G:=\Psi^*\Phi_*g_A$. Note that with these definitions both $\Phi$ and $\Psi$ become morphisms of graded quivers.

Now we must show that $\omega_G$ is homogeneous of degree $a$ with respect to $g_G$. Let $\bf p\in\supp(\omega_G)$, then it is enough to prove that $g_G(\bf p)=a$. By Theorem \ref{thm_Psi(omegaG)cont_inPhi(omegaA)} there exists $\tilde{\bf p}\in\supp(\omega_A)$ such that $\Phi(\tilde{\bf p})=\Psi(\bf p)$, hence, since $g_A(\omega_A)=a$, we must have that
$$ g_G(\bf p) = (\Psi^*\Phi_*g_A)(\bf p) = (\Phi_*g_A)(\Psi(\bf p)) = (\Phi_*g_A)(\Phi(\tilde{\bf p})) = g_A(\tilde{\bf p}) = a. $$
\end{proof}

We will see in Section~\ref{sec_cuts} how one can find in practice gradings which fit the setting of Proposition~\ref{prop_grading-A-->grading-G}.

\subsection{Metacyclic groups embedded in $\SL(s+1,\C)$}\label{sec_metacyclic-embedded-in-SL}

Our final aim will be to obtain $(s-1)$-representation infinite algebras from the McKay quiver $Q_G$: however, by Corollary~\ref{cor_superpot-homog-->GP1}, this can be done when $G$ is contained in $\SL(s,\C)$. Nevertheless, if this condition is not satisfied, we can still use the results in \S~\ref{sec_embedding_gen} and get examples of $s$-representation infinite algebras by embedding $G$ in $\SL(s+1,\C)$.

Denote by $G'$ and $A'$ the images in $\SL(s+1,\C)$ of, respectively, $G$ and $A$ under this embedding, and call $Q_{G'}$, $Q_{A'}$, $\omega_{G'}$, $\omega_{A'}$ the corresponding McKay quivers and superpotentials. We now want to show that an analogue of Proposition~\ref{prop_grading-A-->grading-G} holds in this setting.

Recall that by Proposition~\ref{prop_emb-quiver} $Q_G$ (resp. $Q_A$) is a subquiver of $Q_{G'}$ (resp. $Q_{A'}$), and the latter is obtained from the former by adding all the arrows $i \to \tau(i)$. Now consider the automorphism $\tau$ of $\tilde{Q}_G$ defined in Proposition~\ref{prop_def-tau}. We define $\tilde{Q}_{G'}$ as the quiver obtained from $\tilde{Q}_G$ by adding an arrow $i \to \tau(i)$ for each vertex $i\in (\tilde{Q}_G)_0$.

Since the morphisms $\Phi$ and $\Psi$ are compatible with $\tau$, we can naturally extend them to morphisms $\Phi'$ and $\Psi'$, so that the following diagram commutes:
\[
\begin{tikzcd}
Q_{A'} \arrow[r,"\Phi'"] & \tilde{Q}_{G'} & Q_{G'} \ar[l,swap,"\Psi'"] \\
Q_A \arrow[u,hook] \ar[r,"\Phi"] & \tilde{Q}_G \arrow[u,hook] & Q_G \arrow[u,hook,u] \ar[l,swap,"\Psi"]
\end{tikzcd}
\]

Note also that the $G/A$-action on $Q_A$ can be extended to a $G/A$-action on $Q_{A'}$, and $\tilde{Q}_{G'}$ can be thought as the quotient of $Q_{A'}$ by this action.

\begin{prop}\label{prop_grading-A'-->grading-G'}
Let $g_A$ be a grading on $Q_{A'}$ such that $\omega_{A'}$ is homogeneous of degree $a$ and the morphism $\Phi'$ is $g_A$-gradable. Then there exists a grading $g_G$ on $Q_{G'}$ such that $\omega_{G'}$ is homogeneous of degree $a$.
\end{prop}
\begin{proof}
Note that $\Psi'(\supp(\omega_{G'})) \subseteq \Phi'(\supp(\omega_{A'}))$: indeed, every path in the support of $\omega_{G'}$ is, up to cyclic permutation, in the form $a_{\bf p} \cdot \bf p$ for a path $\bf p\in\supp(\omega_G)$, where $a_{\bf p}$ is the arrow $t(\bf p) \to \tau(t(\bf p))$. Hence it is enough to show that $a_{\bf p}$ is in the image of $\Phi'$, but this is true because both $\Phi$ and $\Psi$ are compatible with the twists.

Now it is easy to check that the same proof of Proposition~\ref{prop_grading-A-->grading-G} carries over if we just replace $Q_A$, $Q_G$, $\omega_A$, $\omega_G$, $\Phi$, $\Psi$ respectively by $Q_{A'}$, $Q_{G'}$, $\omega_{A'}$, $\omega_{G'}$, $\Phi'$, $\Psi'$.
\end{proof}

\section{Cuts}\label{sec_cuts}

In this section we will illustrate a method to define explicitly some gradings on the McKay quivers we studied so far. For this purpose, we will first describe $Q_A$ and $Q_{A'}$ using a construction of \cite{HIO}. In order to encompass both the cases of $A$ and $A'$, we will first do this in a more general setting.

Fix an integer $N\geq 2$. Let $\{e_0,\dots,e_{N-1}\}$ be the canonical basis of $\R^N$ and let $$E=\{(x_0,\dots,x_{N-1})\in\R^N \,|\, \sum_{i=0}^{N-1}x_i=0\}.$$ The set $\{e_i-e_j \in E \,|\, 0\leq i\neq j\leq N-1\}$ is a root system of type $\mathrm{A}_{N-1}$ (see for example \cite{H}). We take as simple roots $\alpha_i:=e_i-e_{i-1}$ for $i=1,\dots,N-1$, and we set $\alpha_0=\alpha_N:=e_0-e_{N-1}=-\sum_{i=1}^{N-1}\alpha_i$. Define the root lattice $L$ as the lattice in $E$ generated by the simple roots.

We define a quiver $Q=Q^{(N)}$ as follows. Its vertices are $Q_0:=L$; for each vertex $v\in Q_0$ and each $k=0,\dots,N-1$, we have an arrow $a_k^v\colon v \to v+\alpha_k$. Sometimes, when this does not cause any confusion, we will drop the superscript and write $a_k$ in place of $a_k^v$.

Let $m,r_1,\dots,r_N$ be positive integers such that $(r_i,m)=1$ for all $i$. Let $H$ be the subgroup of $\SL(N,\C)$ generated by the matrix
\[
\left(
\begin{array}{cccc}
\epsilon_m^{r_1} & 0 & \cdots & 0 \\
0 & \epsilon_m^{r_2} & \cdot & 0 \\
\vdots & \vdots & \ddots & \vdots \\
0 & 0 & \cdots & \epsilon_m^{r_N}
\end{array}\right).
\]
It is easy to see that $H$ is cyclic of order $m$. Let $\eta\colon L \to \Z/m\Z$ be the homomorphism of abelian groups defined by $\eta(\alpha_j)=-r_j$ for all $j=1,\dots,N$. It is clearly surjective because each $r_j$ is invertible modulo $m$: hence it induces an isomorphism (which we call again $\eta$)
$$\eta\colon L/B \xrightarrow{\sim} \Z/m\Z,$$
where $B:=\ker(\eta)$.

The subgroup $B\leq L$ acts on $L$ by translations and this extends naturally to an action on the quiver $Q$, so we can form the orbit quiver $Q/B$. We will denote by $\ol a_k^{\ol v}\colon \ol v \to \ol v+\ol\alpha_k$ the arrow in $Q/B$ corresponding to the orbit of $a_k^v$.

Now consider the McKay quiver $Q_H$. It has vertices $(Q_H)_0:=\Z/m\Z$ and an arrow $x_k^i\colon i \to i-r_k$ for all $i\in(Q_H)_0$, $k=1,\dots,N$. This follows from \cite[Corollary~4.1]{BSW}; alternatively, it can be proved in the same way of Proposition~\ref{prop_quiver_A}.

\begin{prop}\label{prop_Q/B-iso-Q_H}
The map $\eta$ extends to an isomorphism of quivers $\eta\colon Q/B \to Q_H$, which is given on arrows by $\eta(\ol a_k^{\ol v})=x_k^{\eta(\ol v)}$. Moreover the McKay relations in $Q_H$ described in Proposition~\ref{prop_quiver_A} correspond to the relations
$$ \{\ol a_h^{\ol v+\ol\alpha_k} \ol a_k^{\ol v} = \ol a_k^{\ol v+\ol\alpha_h} \ol a_h^{\ol v} \,|\, \ol v\in L/B, \; 0 \leq k,h \leq N-1\}$$
in $Q/B$.
\end{prop}
\begin{proof}
This follows immediately from Proposition~\ref{prop_quiver_A}.
\end{proof}

\begin{figure}
\centering
\begin{tikzpicture}[->,scale=0.8]
\begin{scope}[rotate=-15,inner sep=1.5mm]
\begin{scriptsize}

\def\b{1}							
\def\n{7}

\pgfmathtruncatemacro{\m}{3*\n}
\pgfmathtruncatemacro{\r}{3*\b+1}

\coordinate (A1) at (1,0,-1);
\coordinate (A2) at (-1,1,0);
\coordinate (A0) at (0,-1,1);

\node (beta0) at (0,0,0){};								
\node (beta1) at ($-\r*(A1) + (A2)$){};
\node (beta2) at ($\r*(A2) - (A0)$){};
\node (beta3) at ($(beta2) - (beta1)$){};

\filldraw[very nearly transparent,-] ($(beta0)$) -- ($(beta1)$) -- ($(beta2)$) -- ($(beta3)$);

\pgfmathtruncatemacro{\minx}{-\r-1}			
\pgfmathtruncatemacro{\maxx}{\r-2}
\pgfmathtruncatemacro{\miny}{-\b}
\pgfmathtruncatemacro{\maxy}{\r+1}

	\pgfmathtruncatemacro{\i}{Mod(-\minx+1+(-\maxy-1)*(\r+1),\m)}
	\node (D) at (\minx-1,\maxy+1,-\minx-\maxy) {$\i$};	
\foreach \y in {\miny,...,\maxy}
{
	\pgfmathtruncatemacro{\i}{Mod(-\minx+1+(-\y)*(\r+1),\m)}	
	\node (D) at (\minx-1,\y,-\minx-\y+1) {$\i$};		
	\node (P) at (\maxx,\y,-\maxx-\y) {};
	\node (Q) at (\maxx,\y+1,-\maxx-\y-1) {};
 \draw (Q) -- (P);												
}

\foreach \x in {\minx,...,\maxx}
{
\pgfmathtruncatemacro{\i}{Mod(-\x+(-\maxy-1)*(\r+1),\m)}	
\node (M) at (\x,\maxy+1,-\x-\maxy-1)  {$\i$};
\node (N) at (\x-1,\maxy+1,-\x-\maxy)  {};
 \draw (N) -- (M);											
	\foreach \y in {\miny,...,\maxy}
{
	\pgfmathtruncatemacro{\i}{Mod(-\x+(-\y)*(\r+1),\m)}
	\node (A) at (\x,\y,-\x-\y)  {$\i$};		
	\node (B) at (\x-1,\y+1,-\x-\y) {};
	\node (C) at (\x-1,\y,-\x-\y+1) {};
	  \draw (A) -- (B);
	  \draw (C) -- (A);
	  \draw (B) -- (C);
}
}
\end{scriptsize}
\end{scope}
\end{tikzpicture}
\caption{A part of the infinite quiver $Q^{(3)}$. Each vertex is labelled with its image under $\eta$, where we set $m=21$, $r_1=1$, $r_2=4$, $r_3=16$. The McKay quiver $Q_H$ is obtained by taking the vertices in the shadowed parallelogram and identifying the upper side with the lower one and the left side with the right one (note that in this way $Q_H$ can be naturally embedded in a real $2$-dimensional torus).}\label{fig_quiver-s=3}
\end{figure}
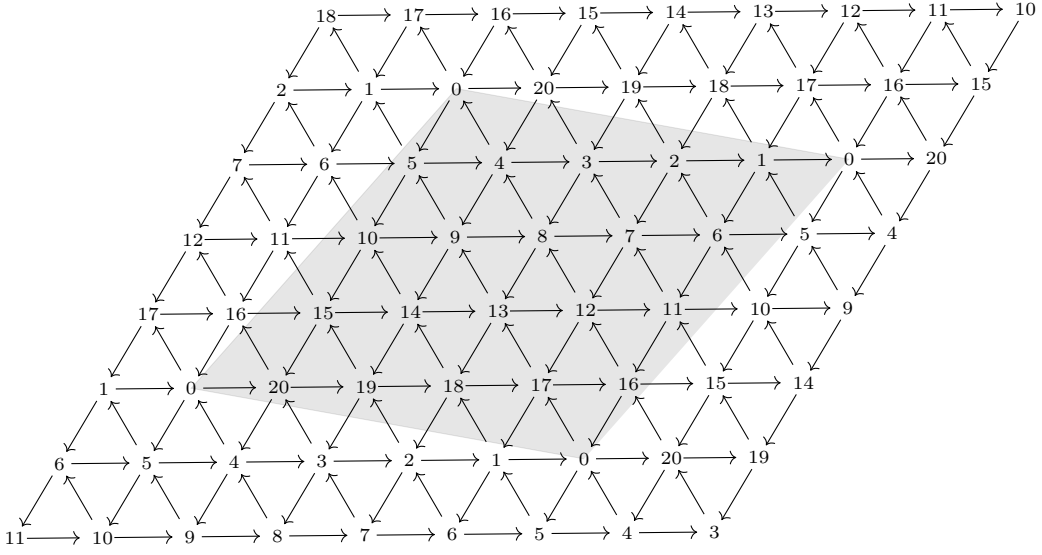

\begin{defn}
A subset $C$ of the arrows of $Q$ is called a \emph{cut} if every path of the form
$$a_{\sigma(0)}\cdots a_{\sigma(N-1)}\colon v\to v,$$
for a permutation $\sigma\in\mathfrak{S}_N$, contains exactly one arrow of $C$.

Given a cut $C$, we define a grading $g_C$ on $Q$ by setting 
\[
g_C(a) = \left\{ \begin{array}{cc}
1 & \text{ if $a\in C$,}\\
0 & \text{ otherwise.}
\end{array} \right.
\]
\end{defn}

We have the following result (cf. \cite[Theorem~5.6]{HIO}).

\begin{prop}\label{prop_cut-B-inv-->H-CYGP1}
Let $C$ be a cut on $Q$ which is invariant under the action of $B$. Then $g_C$ induces a grading on $Q_H$ such that the skew group algebra of $H$ becomes $N$-bimodule Calabi-Yau of Gorenstein parameter $1$.
\end{prop}
\begin{proof}
The projection morphism $Q\to Q/B$ is clearly surjective on arrows, and it is $g_C$-gradable because $C$ is $B$-invariant. Hence it induces a grading on $Q/B$, and in turn one on $Q_H$ via $\eta$: we denote the latter by $g_C^H$. It is clear, by Propositions~\ref{prop_quiver_A} and \ref{prop_Q/B-iso-Q_H}, that the fact that $C$ is a cut implies that the superpotential $\omega_H$ of $Q_H$ is homogeneous of degree $1$. Hence the statement follows by Corollary~\ref{cor_superpot-homog-->GP1}.
\end{proof}

We will now apply this construction to the setting of Section~\ref{sec_sga-of-metacyclic-groups}, in order to get an analogue of Proposition~\ref{prop_cut-B-inv-->H-CYGP1} for metacyclic groups.

Let $G$ be the metacyclic group associated to some integers $m,r,s,t$. From now on we will assume that all the conditions \ref{item_(m,r)=1},\dots,\ref{item_r=1(s)} hold.

In the following we will treat separately the cases where $G\subseteq \SL(s,\C)$ and $G\not\subseteq \SL(s,\C)$, and we will refer to them respectively as (SL) and (GL). Depending on the case, the objects we introduced at the beginning of this section will assume the following values:
\begin{itemize}
\item[(SL)] $H=A$, $N=s$, $r_i=r^{i-1}$ for all $i=1,\dots,s$;
\item[(GL)] $H=A'$ (according to the notation of Section~\ref{sec_metacyclic-embedded-in-SL}), $N=s+1$, $r_i=r^{i-1}$ for all $i=1,\dots,s$, $r_{s+1}=-c$.
\end{itemize}

Before going on we first give an example of a cut in $Q$ which will be important in the following in both the above cases.

\subsection{An example of cut}

We keep the notation introduced previously, but from now on we will consider it only in the cases (SL) and (GL).

Let $l$ be a positive integer. For $k\in\{1,\dots,l\}$, define $\gamma=\gamma_k\colon L \to \Z/sl\Z$ as the group homomorphism given by $\gamma(\alpha_i)=k$ for all $i=1,\dots,N-1$. Note that in this way we have $\gamma(\alpha_0)=-(N-1)k$.

If $x\in\Z/sl\Z$, we denote by $\ol{x}$ the unique representative of $x$ in $\{0,\dots,sl-1\}$. 

\begin{defn}\label{def_cut}
Given $k\in\{1,\dots,l\}$, we define the following subset of $Q_1$:
$$ C_k^{(l)} = C_k := \left\{ a_i\colon v \to v+\alpha_i \,|\, \ol{\gamma(v)} \geq \ol{\gamma(v+\alpha_i)},\, 0\leq i \leq N-1 \right\}. $$
\end{defn}

\begin{prop}\label{prop_example-cut}
Let $l\geq 1$, $1\leq k \leq l$.
\begin{enumerate}[label=(\alph*)]
\item In case (SL), $C_k$ is a cut in $Q$ for all $k$. In case (GL), $C_k$ is a cut in $Q$ for all $k<l$.
\item Every path in $Q$ of length greater or equal than $sl$ contains at least an arrow of $C_k$.
\item Suppose that $l$ divides both the integers $n$ and $b$ defined in conditions \ref{item_m=sn} and \ref{item_r=1(s)}. Then, for all $k=1,\dots,l$, $C_k$ is invariant under the action of $B$.
\end{enumerate}
\end{prop}
\begin{proof}
(a) Suppose we have a path $a_{i_0}\cdots a_{i_{N-1}}\colon v\to v$ in $Q$ with $\{i_0,\dots,i_{N-1}\}=\{0,\dots,N-1\}$: we have to show that there exist exactly one $j$ such that $a_{i_j}\in C_k$. Up to a cyclic permutation, we can assume that $i_0=0$. If $N=s$, we have $(N-1)k=sk-k<sl$; if $N=s+1$ then by hypothesis $k<l$, and so $(N-1)k=sk<sl$. In both cases $(N-1)k<sl$, so, since $\gamma(v+\alpha_0)=\gamma(v)-(N-1)k$, we have that $a_{i_0}\in C_k$ if and only if $\ol{\gamma(v)}-(N-1)k \geq 0$. Now put $v_j := v+\sum_{h=0}^j \alpha_{i_h}$ and note that, for $j=1,\dots,N-1$, we have $a_{i_j}\colon v_{j-1} \to v_j$ and $\gamma(v_{j})=\gamma(v_{j-1})+k$. Hence $a_{i_j}\not\in C_k$ if and only if $\ol{\gamma(v_{j})} = \ol{\gamma(v_{j-1})}+k$, and $a_{i_j}\in C_k$ if and only if $\ol{\gamma(v_{j-1})}+k \geq sl$.

Now suppose that $a_{i_0}\in C_k$. By the above discussion we have that $\ol{\gamma(v_0)} = \ol{\gamma(v)}-(N-1)k \geq 0$, so, if $1\leq j\leq N-1$, then $0<\ol{\gamma(v_0)}+jk = \ol{\gamma(v)}-(N-1-j)k \leq \ol{\gamma(v)}<sl$ and thus $\ol{\gamma(v_j)} = \ol{\gamma(v_0)+jk} = \ol{\ol{\gamma(v_0)}+jk} = \ol{\gamma(v_0)}+jk$. This means that, for all $j=1,\dots,N-1$, $\ol{\gamma(v_{j})} = \ol{\gamma(v_{j-1})}+k$ and hence $a_{i_j}\not\in C_k$.

Assume now that $a_{i_0}\not\in C_k$, so that we have $\ol{\gamma(v)}-(N-1)k < 0$. Suppose by absurd that no $a_{i_j}$ is in $C_k$. Then we must have $\ol{\gamma(v_{N-1})} = \ol{\gamma(v_0)}+(N-1)k$, but $v_{N-1}=v$ and so $\ol{\gamma(v_0)}+(N-1)k = \ol{\gamma(v_{N-1})} = \ol{\gamma(v)} < (N-1)k$, which is a contradiction because $\ol{\gamma(v_0)}\geq 0$. Hence there exist a $j\geq 1$ such that $a_{i_j}$ is in $C_k$: by an argument similar to above, it can be easily proved that such a $j$ must be unique.

(b) Consider a path $p=a_{i_1}\cdots a_{i_h}$ in $Q$. If no arrow in $p$ is contained in $C_k$, then we must have a chain of inequalities $\ol{\gamma(v)} < \ol{\gamma(v+\alpha_{i_1})} < \dots < \ol{\gamma(v+\sum_{j=1}^h \alpha_{i_j})}$ in $\{0,\dots,sl-1\}$, and this is clearly impossible if $h\geq sl$.

(c) Suppose that $v=\sum_{j=1}^{N-1} \mu_j\alpha_j$ is in $B$: this means that $\eta(v)=-\sum_{j=1}^{N-1} \mu_j r^{j-1} \equiv 0 \pmod m$. Since $l|n$, we have $sl|sn=m$ and so we get $\sum_{j=1}^{N-1} \mu_j r^{j-1} \equiv 0 \pmod{sl}$. Moreover we have that $sl|sb=r-1$, so $r \equiv 1 \pmod{sl}$ and thus $\sum_{j=1}^{N-1} \mu_j r^{j-1} \equiv \sum_{j=1}^{N-1} \mu_j \equiv 0 \pmod{sl}$. Hence we get $\gamma(v)=k\sum_{j=1}^{N-1} \mu_j=0$, which clearly implies that $C_k$ is $B$-invariant.
\end{proof}

\subsection{The case $G\subseteq\SL(s,\C)$}\label{sec_caseSL}

Recall from Section~\ref{sec_twist} that the determinants of the generators of $G$ are $\det(\alpha)=\epsilon_m^c$, where $c=\sum_{i=0}^{s-1} r^i$, and $\det(\beta)=(-1)^{s-1}\epsilon_m^t$. Hence, if $G\subseteq\SL(s,\C)$, we have that
\[c \equiv 0 \pmod m, \qquad
t= \left\{\begin{array}{cc}
n & \text{ if } s=2,\\
0 & \text{ if } s>2.
\end{array} \right.
\]

Clearly we have that $A$ is contained in $\SL(s,\C)$, too. The vertices of the quiver $Q = Q^{(s)}$ are the points of a root lattice of type $\mathrm  A_{s-1}$, with basis given by $\alpha_1,\dots,\alpha_{s-1}$. We have an isomorphism of quivers $\eta\colon Q/B \to Q_A$, where $B$ is the kernel of the group homomorphism $\eta\colon L\to \Z/m\Z$ given $\eta(\alpha_j)=-r^{j-1}$ for all $j=1,\dots,s-1$. Note that, since $c \equiv 0 \pmod m$, we have $\eta(\alpha_0)=\eta(\alpha_s)=-r^{s-1}$.

Recall that the action of $G/A$ on $A$ induces an automorphism $\varphi$ of $Q_A$, which is given on vertices by $\varphi(i)=ri$. Identifying $Q/B$ with $Q_A$ via $\eta$, this induces an automorphism of $Q/B$ given by $\varphi(\ol \alpha_j)=\ol \alpha_{j+1}$, $0\leq j\leq s-1$. Moreover it is easy to check that this automorphism lifts to an action on $Q$ defined by $\varphi(\alpha_j)=\alpha_{j+1}$.

\begin{thm}\label{thm_grading-invariant-cut}
Let $C$ be a cut in $Q$ which is invariant under both the actions of $B$ and $G/A$, and let $g_C$ be the grading on $Q$ induced by $C$. Then there exists a grading on $Q_G$ such that $\Pi_G$ is $s$-Calabi-Yau of Gorenstein parameter $1$ with the grading induced by it.

In particular, if the degree $0$ part $(\Pi_{G'})_0$ of $\Pi_{G'}$ is finite dimensional, then $(\Pi_{G'})_0$ is $s$-representation infinite.
\end{thm}
\begin{proof}
Since $C$ is $B$-invariant, by Proposition~\ref{prop_cut-B-inv-->H-CYGP1} it induces a grading $g_C^A$ on $Q_A$ such that the superpotential $\omega_A$ becomes homogeneous of degree $1$. Moreover, the fact that $C$ is $G/A$ invariant implies that the morphism $\Phi\colon Q_A\to \tilde{Q}_G$ is $g_C^A$-gradable: hence, by Proposition~\ref{prop_grading-A-->grading-G}, there exists a grading $g_C^G$ on $Q_G$ such that $\omega_G$ is homogeneous of degree $1$. So, applying Corollary~\ref{cor_superpot-homog-->GP1}, the result follows.
\end{proof}

\begin{cor}\label{cor_grading-invariant-cut}
Let $l$ be a positive integer which divides both $n$ and $b$ and let $k\in\{1,\dots,l\}$. Then the cut $C_k^{(l)}$ of Definition~\ref{def_cut} induces a grading on $\Pi_G$ such that $(\Pi_G)_0$ is $(s-1)$-representation infinite.
\end{cor}
\begin{proof}
By Proposition~\ref{prop_example-cut}(c) we have that $C_k^{(l)}$ is $B$-invariant. It is also $G/A$-invariant, since the action of $G/A$ permutes the set $\{\alpha_0,\dots,\alpha_{s-1}\}$. So the result follows from Theorem~\ref{thm_grading-invariant-cut} if we prove that $(\Pi_G)_0$ is finite dimensional. To achieve this, is enough to show that there exists an integer $M$ such that every path in $Q_G$ of length greater or equal than $M$ has degree $1$. Clearly it is enough to prove this for $\tilde{Q}_G$. Let $p$ be a path in $\tilde{Q}_G$ of length $h$. Then it is easy to see that $p$ lifts to a path of length $h$ in $Q$. Hence, by Proposition~\ref{prop_example-cut}(b), it is enough to take $M=sl$.
\end{proof}

\subsection{The case $G\not\subseteq\SL(s,\C)$}

We retain all the notation of \S~\ref{sec_metacyclic-embedded-in-SL}. So we embed $G$ and $A$ in $\SL(s+1,\C)$ and we denote their images by $G'$ and $A'$ respectively. In this case, the vertices of the quiver $Q = Q^{(s+1)}$ are the points of a root lattice of type $\mathrm  A_s$, with basis $\alpha_1,\dots,\alpha_s$.

We have an isomorphism $\eta\colon Q/B \to Q_{A'}$, where $B$ is the kernel of the group homomorphism $\eta\colon L\to \Z/m\Z$ given by $\eta(\alpha_j)=-r^{j-1}$ for $j=1,\dots,s$. Note that in this case the element $\alpha_0=\alpha_{s+1}$ is sent to $c$.

The action of $G/A$ on $Q_A$, which is given on vertices by $\varphi(i)=ri$, extends naturally to $Q_{A'}$. Identifying the latter with $Q/B$ via $\eta$, this induces an automorphism of $Q/B$ given by $\varphi(\ol \alpha_j)=\ol \alpha_{j+1}$ for $j=1,\dots,s-1$, $\varphi(\ol \alpha_s)=\ol \alpha_1$. Moreover, it is easy to check that this automorphism lifts to an action on $Q$ defined by $\varphi(\alpha_j)=\alpha_{j+1}$, $j=1,\dots,s-1$, $\varphi(\alpha_s)=\alpha_1$. It is also worth to point out that $\alpha_0$ is now fixed by this action.

\begin{thm}\label{thm_grading-invariant-cut-embed}
Let $C$ be a cut in $Q$ which is invariant under both the actions of $B$ and $G/A$, and let $g_C$ be the grading on $Q$ induced by $C$. Then there exists a grading on $Q_{G'}$ such that the induced grading on $\Pi_{G'}$ makes it $(s+1)$-Calabi-Yau of Gorenstein parameter $1$.

In particular, the degree $0$ part $(\Pi_{G'})_0$ of $\Pi_{G'}$ is $s$-representation infinite if it is finite dimensional.
\end{thm}
\begin{proof}
Since $C$ is $B$-invariant, by Proposition~\ref{prop_cut-B-inv-->H-CYGP1} it induces a grading $g_C^A$ on $Q_{A'}$ such that the superpotential $\omega_{A'}$ becomes homogeneous of degree $1$. Moreover, the fact that $C$ is $G/A$ invariant implies that the morphism $\Phi'\colon Q_{A'}\to \tilde{Q}_{G'}$ is $g_C^A$-gradable. Hence we can apply Proposition~\ref{prop_grading-A'-->grading-G'} to get a grading $g_C^G$ on $Q_{G'}$ such that $\omega_{G'}$ is homogeneous of degree $1$. So the result follows from Corollary~\ref{cor_superpot-homog-->GP1}.
\end{proof}

\begin{cor}\label{cor_grading-invariant-cut-embed}
Let $l$ be a positive integer which divides both $n$ and $b$ and let $1\leq k\leq l-1$. Then the cut $C_k^{(l)}$ of Definition~\ref{def_cut} induces a grading on $\Pi_{G'}$ such that $(\Pi_{G'})_0$ is $s$-representation infinite.
\end{cor}
\begin{proof}
Note that $C_k^{(l)}$ is $G/A$-invariant, since the action of $G/A$ permutes the set $\{\alpha_0,\dots,\alpha_s\}$. The rest of the proof is analogue to the one of Corollary~\ref{cor_grading-invariant-cut}.
\end{proof}

We end this section with the following easy observation, which gives a method of obtaining new gradings in both the cases (SL) and (GL).

\begin{rem}\label{rem_morecuts}
If a cut in $Q$ contains an arrow which starts or ends with a fixed point $j$, then all the corresponding splitting arrows in $Q_G$ (respectively $Q_{G'}$) will have degree $1$ with respect to the grading defined in Theorem~\ref{thm_grading-invariant-cut} (respectively Theorem~\ref{thm_grading-invariant-cut-embed}). In this case all the paths in the superpotential $\omega_G$ (respectively $\omega_{G'}$) passing through $j^{(\ell)}$ will contain a subpath of the form $i \xrightarrow{a} j^{(\ell)} \xrightarrow{b} k$, where one arrow among $a$ and $b$ has degree $1$ and the other has degree $0$. Hence it is easy to see that if we define a new grading in the quiver by swapping the degrees of $a$ and $b$ (see Figure~\ref{fig_more-cuts}), then the degree of the superpotential remains unchanged and so the algebra $\Pi_G$ (respectively $\Pi_{G'}$) has again Gorenstein parameter $1$. Moreover, if the degree $0$ part of $\Pi_G$ with the old grading is finite dimensional, then the same is true if we consider the new grading.
\end{rem}

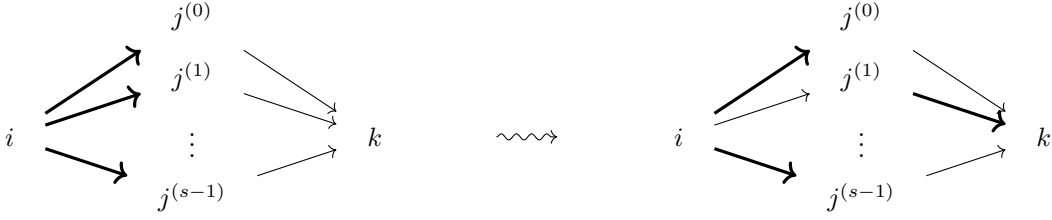
\begin{figure}
\begin{tikzpicture}[scale=0.8,inner sep=4mm]

\node (A) at (-3,0){$i$};
\node (B0) at (0,2){$j^{(0)}$};
\node (B1) at (0,1){$j^{(1)}$};
\node (B2) at (0,0){$\vdots$};
\node (B3) at (0,-1){$j^{(s-1)}$};
\node (C) at (3,0){$k$};

 \draw [->,very thick] (A) -> (B0) node[midway,above,yshift=-2mm] {};
 \draw [->,very thick] (A) -> (B1);
 \draw [->,very thick] (A) -> (B3) node[midway,below,yshift=1mm] {};

 \draw [->] (B0) -> (C) node[midway,above,yshift=-2mm] {};
 \draw [->] (B1) -> (C);
 \draw [->] (B3) -> (C) node[midway,below,yshift=1mm] {};



\begin{scope}[xshift=11cm]
\node (A) at (-3,0){$i$};
\node (B0) at (0,2){$j^{(0)}$};
\node (B1) at (0,1){$j^{(1)}$};
\node (B2) at (0,0){$\vdots$};
\node (B3) at (0,-1){$j^{(s-1)}$};
\node (C) at (3,0){$k$};

 \draw [->,very thick] (A) -> (B0) node[midway,above,yshift=-2mm] {};
 \draw [->] (A) -> (B1);
 \draw [->,very thick] (A) -> (B3) node[midway,below,yshift=1mm] {};

 \draw [->] (B0) -> (C) node[midway,above,yshift=-2mm] {};
 \draw [->,very thick] (B1) -> (C);
 \draw [->] (B3) -> (C) node[midway,below,yshift=1mm] {};
\end{scope}

\begin{scope}[xshift=5cm]
 \draw [->,decorate,decoration={snake,amplitude=.4mm,segment length=2mm,post length=0.5mm}] (0,0) -> (1,0) node[midway,above,yshift=-3mm] {};
\end{scope}

\end{tikzpicture}
\caption{On the left we have an example of a grading obtained as in Theorem~\ref{thm_grading-invariant-cut} in the local neighbourhood of a fixed point. Here the thick arrows have degree $1$, while the others have degree $0$. Applying the procedure described in Remark~\ref{rem_morecuts} for $\ell=1$ we obtain the grading illustrated on the right.}\label{fig_more-cuts}
\end{figure}

\section{Examples}\label{sec_examples}

We have seen so far that we can obtain examples of higher representation infinite algebras from skew group algebras of some finite groups which satisfy certain conditions, but we still don't know how rich the class of such groups is. The aim of this section is to show that we have indeed many examples of them.

We start by defining two families of metacyclic groups.

\begin{defn}\label{def_M(s,b)}
Let $s$ be a prime number.
\begin{enumerate}[label=(\alph*)]
\item For each $b\geq 1$ we define $M(s,b)$ to be the metacyclic group in $\GL(s,\C)$ associated to integers $m,r,s,t$, where we set
\[
r:=sb+1, \quad m:=\sum_{j=0}^{s-1} r^j, \quad t:= \left\{\begin{array}{cc}
m/s & \text{ if } s=2,\\
0 & \text{ if } s>2,
\end{array} \right.
\]
\item For each $b\geq 2$ we define $\hat{M}(s,b)$ to be the metacyclic group in $\GL(s,\C)$ associated to integers $m,r,s,t$, where we set
\[
r:=sb+1, \quad m:=b\sum_{j=0}^{s-1} r^j, \quad t:= \left\{\begin{array}{cc}
m/s & \text{ if } s=2,\\
0 & \text{ if } s>2,
\end{array} \right.
\]
\end{enumerate}
\end{defn}

\begin{prop}\label{prop_M(s,b)welldef}
The groups defined in Definition~\ref{def_M(s,b)} satisfy all the conditions \ref{item_(m,r)=1},\dots,\ref{item_r=1(s)}. Moreover, we have that $M(s,b) \subseteq \SL(s,\C)$ and $\hat{M}(s,b) \not\subseteq \SL(s,\C)$.
\end{prop}
\begin{proof}
(a) We consider first the case of $M(s,b)$, where $s$ is prime and $b\geq1$. Condition \ref{item_r=1(s)} is clear by definition of $r$. It is easy to see that $r^s-1=(r-1)m=sbm$, so \ref{item_r^s=1(m)} holds. For \ref{item_(m,r)=1}, we note that $m-1 = \sum_{j=1}^{s-1} r^j$ is a multiple of $r$, so clearly $(m,r)=1$. By \ref{item_r=1(s)} we have $m=\sum_{j=0}^{s-1} r^j \equiv \sum_{j=0}^{s-1} 1 \equiv 0 \pmod s$, so we have \ref{item_m=sn}. Now we can write $m=sn$, so $t= \left\{\begin{array}{cc}
n & \text{ if } s=2,\\
0 & \text{ if } s>2,
\end{array} \right.$. Hence \ref{item_(r-1)t=0(m)} is clear if $s>2$; for $s=2$ note that we have $r=2b+1$, $m=1+r=2(b+1)$, $n=b+1$, so $(r-1)t=2bn=2b(b+1) \equiv 0 \pmod m$. Condition \ref{item_r-not1} is clear because $1<r<m$. Finally, since $c=\sum_{j=0}^{s-1} r^j=m$, by the discussion at the beginning of Subsection~\ref{sec_caseSL} we have that $M(b,s)\subseteq\SL(s,\C)$.

(b) Now we consider $\hat{M}(s,b)$, so in this case $s$ is prime and $b\geq2$. Again, \ref{item_r=1(s)} is clear. It is easy to see that $r^s-1=sm$, so \ref{item_r^s=1(m)} holds. Now note that $m-b = b\sum_{j=1}^{s-1} r^j$ is a multiple of $r$, so $(m,r)=(b,r)=(b,sb+1)=1$ and \ref{item_(m,r)=1} holds. Condition \ref{item_m=sn} is clear, so we can write $m=sn$. Hence we have again $t= \left\{\begin{array}{cc}
n & \text{ if } s=2,\\
0 & \text{ if } s>2,
\end{array} \right.$, so \ref{item_(r-1)t=0(m)} is clear except for $s=2$: in this case $m=2b(b+1)$ and $n=b(b+1)$, thus $(r-1)t=2bn=2b^2(b+1) \equiv 0 \pmod m$. Condition \ref{item_r-not1} holds because $1<r<m$. Finally, since $b\geq2$ it is clear that $1<c<m$: this implies that $\det(\alpha)=\epsilon_m^c\neq1$ and thus $\hat{M}(s,b) \not\subseteq \SL(s,\C)$.
\end{proof}

In case (GL), in order to apply the results of the previous section, we need $G$ to satisfy Assumption~\ref{ass_repclosed}. This will not be always the case: in the following we will exhibit a sufficient condition for this to happen, and we will show that the groups of the form $\hat{M}(s,b)$ satisfy it. We will also see in Example~\ref{ex_bin-dih} that this condition is not necessary.

\begin{prop}\label{prop_D-invariant}
Let $r,m,s,t$ be integers which satisfy conditions \ref{item_(m,r)=1},\dots,\ref{item_r=1(s)}, and let $G$ be the associated metacyclic group. Consider the automorphism $\tau$ of $\Z/m\Z$ given by the sum by $c$, and call $u$ its order. Suppose that $(u,s)=1$, then there exists a complete set $\cal D$ of representatives for the action of $G/A$ on $\Z/m\Z$ which is closed under $\tau$.
\end{prop}

\begin{rem}
The integer $u$ we just defined is the smallest positive integer which is a solution to the equation $cx \equiv 0 \pmod m$, and hence it is equal to $\frac{m}{(c,m)}$.
\end{rem}

\begin{proof}[Proof of Proposition~\ref{prop_D-invariant}]
Clearly any set $\cal D$ of representatives contains the set $\cal F$ of fixed points. Moreover, $\cal F$ is closed under $\tau$, so we only have to show a suitable way to choose the elements in $\cal D \minus \cal F$.

We have already seen in Proposition~\ref{prop_def-tau} that $\tau$ induces an action on the $G/A$-orbits of $\Z/m\Z$ given by $\tau([i])=[\tau(i)]=[i+c]$, where $[i]$ denotes the $G/A$-orbit of $i$. So take an $i_1\in\Z/m\Z$ and call $k_1$ the smallest positive integer such that $\tau^{k_1}([i_1])=[i_1]$. Let $\cal D_1 := \{i_1,\tau(i_1),\dots,\tau^{k_1-1}(i_1)\}$: the elements of this set clearly provide representatives of different $G/A$-orbits. Now we want to show that $\cal D_1$ is invariant under $\tau$, which is equivalent to say that $i_1+k_1c \equiv i_1 \pmod m$.

The fact that $[i_1+k_1c]=\tau^{k_1}([i_1])=[i_1]$ implies that there exists an $h\in\{0,\dots,s-1\}$ such that $i_1+k_1c \equiv r^h i_1 \pmod m$. Now an easy induction shows that $r^{lh}i_1 \equiv i_1 + lk_1c \pmod m$ for all $l\geq1$. In particular we have that $r^{uh}i_1 \equiv i_1+uk_1c \pmod m$, and the latter is equivalent to $i_1$ modulo $m$ because $cu \equiv 0 \pmod m$. Since $i_1$ is not in $\cal F$, we must have that $uh \equiv 0 \pmod s$, but this implies that $h \equiv 0 \pmod s$ because we have assumed that $(u,s)=1$. Hence $h=0$ and so $i_1+k_1c \equiv i_1 \pmod m$, which shows that $\cal D_1$ is invariant under $\tau$.

Now we can choose an element $i_2\in\Z/m\Z$ which does not belong to any of the orbits of the elements in $\cal D_1$, and applying the same argument as above we can construct a set $\cal D_2 = \{i_2,\tau(i_2),\dots,\tau^{k_2-1}(i_2)\}$ which is invariant under $\tau$ and such that the orbits of elements in $\cal D_1$ are disjoint from the ones of the elements in $\cal D_2$. Repeating this procedure until we can, we obtain a sequence of sets $\cal D_1,\dots,\cal D_N$, and $\cal D := \bigcup_{j=1}^{N} \cal D_j \cup \cal F$ provides a complete set of representatives for the $G/A$-action which is invariant under $\tau$.
\end{proof}

We now show that we can always obtain a higher representation infinite algebra from the examples we discussed above.

\begin{cor}\label{cor_RI-from-metacyclic}
Let $s$ be a prime number.
\begin{enumerate}[label=(\alph*)]
\item For each integer $b\geq 1$, there exists an $(s-1)$-representation infinite algebra which is the degree 0 part of $\Pi_G$, where $G=M(s,b)$.
\item For each integer $b\geq 2$ such that $(b,s)=1$, there exists an $s$-representation infinite algebra which is the degree 0 part of $\Pi_{G'}$, where $G=\hat{M}(s,b)$ and $G'$ is its embedding in $\SL(s+1,\C)$.
\end{enumerate}
\end{cor}
\begin{proof}
(a) Take a positive integer $l$ which divides both $n$ and $b$ and an integer $1\leq k\leq l$ (for example, we could choose $l=k=1$). Then the result follows by applying Corollary~\ref{cor_grading-invariant-cut}.

(b) Recall that in this case we have $m=bc$, so $u=b$. Hence we have $(u,s)=1$, so by Proposition~\ref{prop_D-invariant} there exists a set of representatives $\cal D$ which is invariant under $\tau'$. Now take an integer $l\geq 2$ which divides both $n$ and $b$ and an integer $1\leq k<l$. Note that $m=bc$ and $c$ is a multiple of $s$, so $n=b\frac{c}{s}$: hence we can choose, for example, $l=b\geq 2$ and $k=1$, and so it is always possible to find $l,k$ which satisfy the above properties. Then the result follows by applying Corollary~\ref{cor_grading-invariant-cut-embed}.
\end{proof}

In the following we will give some examples for $s=2,3$.

\subsection{Examples for $G\subseteq\SL(s,\C)$}

\begin{example}
Let $s=2$. We now want to describe all metacyclic groups in $\SL(2,\C)$ which satisfy conditions \ref{item_(m,r)=1},\dots,\ref{item_r=1(s)}.

By \ref{item_m=sn} we can take $m=2n$ for an integer $n\geq 2$. Since we want our group to be in $\SL(2,\C)$, it is clear that $r=m-1$ is the unique (modulo $m$) possible choice of $r$. Also, as we already observed previously, we must take $t=n$. Hence the corresponding metacyclic group $G$ is generated by the matrices
\[
\alpha = \left(
\begin{array}{cc}
\epsilon_{2n} & 0 \\
0 & \epsilon_{2n}^{-1}
\end{array}\right),
\qquad
\beta = \left(
\begin{array}{cc}
0 & -1 \\
1 & 0
\end{array}\right).
\]

The quiver $Q$ is the preprojective quiver of type $\tilde{\mathrm A}_\infty^\infty$ and can be drawn on a line:
\[
\begin{tikzcd}[every arrow/.append style={shift left}]
\,\arrow[dotted,thick,-,shift right]{r} & 1 \arrow{r}{a_1} & 0 \arrow{r}{a_1} \arrow{l}{a_0} & 2n-1 \arrow{l}{a_0} \arrow[dotted,thick,-,shift right]{r} & \arrow{r}{a_1} & n \arrow{l}{a_0} \arrow{r}{a_1} & \arrow{l}{a_0} \arrow[dotted,thick,-,shift right]{r} & 1 \arrow{r}{a_1} & 0 \arrow{l}{a_0} \arrow[dotted,thick,-,shift right]{r} &\,
\end{tikzcd}
\]
Here we have labelled every vertex by its image in $\Z/m\Z$ under the morphism $\eta$ described in Section~\ref{sec_cuts}. Then $Q_A$, which we recall being isomorphic to the orbit quiver $Q/A$, is obtained by taking the subquiver of the above quiver given by some $m+1$ consecutive vertices $0,1,\dots,2n-1,0$ and identifying the two vertices labelled by $0$. Note that in this way we obtain the preprojective quiver of type $\tilde{\mathrm A}_{2n-1}$:
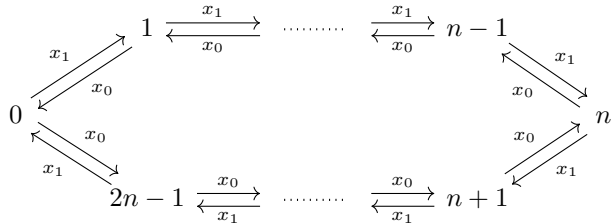
\[
\begin{tikzcd}[every arrow/.append style={shift left}]
 & 1 \arrow{ld}{x_0} \arrow{r}{x_1} & \arrow{l}{x_0} \arrow[dotted,thick,-,shift right]{r} & \arrow{r}{x_1} & n-1 \arrow{l}{x_0} \arrow{rd}{x_1} & \\
0 \arrow{ru}{x_1} \arrow{rd}{x_0} & & & & & n \arrow{ld}{x_1}  \arrow{lu}{x_0}  \\
 & 2n-1 \arrow{lu}{x_1} \arrow{r}{x_0} & \arrow{l}{x_1} \arrow[dotted,thick,-,shift right]{r} & \arrow{r}{x_0} & n+1 \arrow{l}{x_1} \arrow{ru}{x_0} & 
\end{tikzcd} 
\]
By Proposition~\ref{prop_quiver_A}, its path algebra modulo the preprojective relations is isomorphic to the skew group algebra $\C[V]*A$.

Now consider the action of $G/A \cong \Z/2\Z$, which sends $i$ to $-i$. On $Q$ it is given by rotating the quiver of $180^\circ$ around the vertex $0$, while on $Q_A$ it is given by reflecting with respect to the horizontal line passing through $0$ and $n$. It is easy to see that $\cal D = \{0,\dots,n\}$ is a complete set of representatives in $(Q_A)_0$ of this action, and that the only fixed points are $0$ and $n$. Hence the quiver $Q_G$ is given by
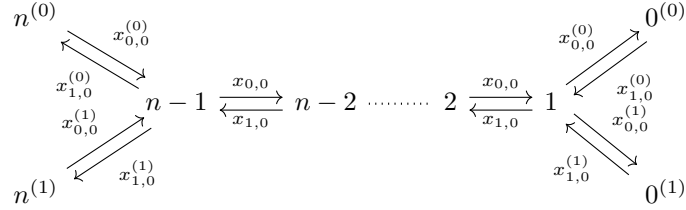
\[
\begin{tikzcd}[every arrow/.append style={shift left}]
n^{(0)} \arrow{rd}{x_{0,0}^{(0)}} & & & & & 0^{(0)} \arrow{ld}{x_{1,0}^{(0)}} \\
 & n-1 \arrow{ld}{x_{1,0}^{(1)}} \arrow{r}{x_{0,0}} \arrow{lu}{x_{1,0}^{(0)}} & n-2 \arrow{l}{x_{1,0}} \arrow[dotted,thick,-,shift right]{r} & 2 \arrow{r}{x_{0,0}} & 1 \arrow{l}{x_{1,0}} \arrow{ru}{x_{0,0}^{(0)}} \arrow{rd}{x_{0,0}^{(1)}} & \\
n^{(1)} \arrow{ru}{x_{0,0}^{(1)}} & & & & & 0^{(1)} \arrow{lu}{x_{1,0}^{(1)}}
\end{tikzcd}
\]
which is the preprojective quiver of type $\tilde{\mathrm D}_{n+2}$.

Using Corollary~\ref{cor_supp-s=2,3} and the description of the superpotential of $Q_A$, we can see that the support of $\omega_G$ is given by all paths of the form $x_{0,0}x_{1,0}$, $x_{1,0}x_{0,0}$, $x_{0,0}^{(\ell)}x_{1,0}^{(\ell)}$, $x_{1,0}^{(\ell)}x_{0,0}^{(\ell)}$, $\ell=0,1$. Moreover, any cut in $Q$ which is invariant under the actions of $B$ and $G/A$ induces an orientation of the graph underlying $Q_G$. So if we consider the corresponding induced grading on $\Pi_G$ and take the degree $0$ part, we obtain an hereditary representation infinite algebra of type $\tilde{\mathrm D}_{n+2}$. Using Remark~\ref{rem_morecuts} it is easy to see that all these algebras are obtainable in this way.

We may note that the relations we get from the superpotential in our case are different from the classical preprojective relations. However, we still have that $\Pi_G$ is isomorphic to a preprojective algebra of type $\tilde{\mathrm D}_{n+2}$, because we know, by Theorem~\ref{thm_n-RI_CYGP1}, that it is isomorphic to the preprojective algebra of $(\Pi_G)_0$, and we saw that $(\Pi_G)_0$ is the path algebra of a quiver of type $\tilde{\mathrm D}_{n+2}$.
\end{example}

\begin{figure}[b]
\centering
\tikzset{->-/.style={decoration={
  markings,
  mark=at position #1 with {\arrow{latex}}},postaction={decorate}},
  ->-/.default=0.6,}

\begin{tikzpicture}[->,scale=0.8]
\begin{scope}[rotate=-15,inner sep=1.5mm]
\begin{scriptsize}

\def\b{1}							
\def\n{7}

\pgfmathtruncatemacro{\m}{3*\n}
\pgfmathtruncatemacro{\r}{3*\b+1}

\coordinate (A1) at (1,0,-1);
\coordinate (A2) at (-1,1,0);
\coordinate (A0) at (0,-1,1);

\node (beta0) at (0,0,0){};								
\node (beta1) at ($-\r*(A1) + (A2)$){};
\node (beta2) at ($\r*(A2) - (A0)$){};
\node (beta3) at ($(beta2) - (beta1)$){};

\node (fix0) at ($ 0.333*(beta0) + 0.333*(beta1) + 0.333*(beta2) $){};
\node (fix1) at ($ 0.333*(beta0) + 0.333*(beta3) + 0.333*(beta2) $){};

\filldraw[very nearly transparent,-] ($(beta0)$) -- ($(fix0)$) -- ($(beta2)$) -- ($(fix1)$);
\draw[thick,->-,-,opacity=0.5] (beta0) -- (fix0);
\draw[thick,->-,-,opacity=0.5] (beta2) -- (fix0);
\draw[thick,->-,-,opacity=0.5] (beta2) -- (fix1);
\draw[thick,->-,-,opacity=0.5] (beta0) -- (fix1);


\pgfmathtruncatemacro{\minx}{-\r-1}			
\pgfmathtruncatemacro{\maxx}{\r-2}
\pgfmathtruncatemacro{\miny}{-\b}
\pgfmathtruncatemacro{\maxy}{\r+1}

	\pgfmathtruncatemacro{\i}{Mod(-\minx+1+(-\maxy-1)*(\r+1),\m)}
	\node (D) at (\minx-1,\maxy+1,-\minx-\maxy) {$\i$};	
\foreach \y in {\miny,...,\maxy}
{
	\pgfmathtruncatemacro{\i}{Mod(-\minx+1+(-\y)*(\r+1),\m)}	
	\node (D) at (\minx-1,\y,-\minx-\y+1) {$\i$};		
	\node (P) at (\maxx,\y,-\maxx-\y) {};
	\pgfmathtruncatemacro{\iright}{Mod(-\maxx+(-\y)*(\r+1),\m)}
	\pgfmathtruncatemacro{\cuttt}{Mod(\iright,3)}
	\node (Q) at (\maxx,\y+1,-\maxx-\y-1) {};
 \ifnum \cuttt=0 \draw[very thick] (Q) -- (P);
 [\else \draw (Q) -- (P);] \fi												
}

\foreach \x in {\minx,...,\maxx}
{
\pgfmathtruncatemacro{\i}{Mod(-\x+(-\maxy-1)*(\r+1),\m)}	
\pgfmathtruncatemacro{\cutt}{Mod(\i,3)}
\node (M) at (\x,\maxy+1,-\x-\maxy-1)  {$\i$};
\node (N) at (\x-1,\maxy+1,-\x-\maxy)  {};
 \ifnum \cutt=0 \draw[very thick] (N) -- (M);											
 [\else \draw (N) -- (M);] \fi

	\foreach \y in {\miny,...,\maxy}
{
	\pgfmathtruncatemacro{\i}{Mod(-\x+(-\y)*(\r+1),\m)}
	\pgfmathtruncatemacro{\cut}{Mod(\i,3)}
	\node (A) at (\x,\y,-\x-\y)  {$\i$};		
	\node (B) at (\x-1,\y+1,-\x-\y) {};
	\node (C) at (\x-1,\y,-\x-\y+1) {};
	  \ifnum \cut=1 \draw[very thick] (A) -- (B);
	  [\else \draw (A) -- (B);] \fi
	  \ifnum \cut=0 \draw[very thick] (C) -- (A);
	  [\else \draw (C) -- (A);] \fi
	  \ifnum \cut=2 \draw[very thick] (B) -- (C);
	  [\else \draw (B) -- (C);] \fi
}
}
\end{scriptsize}
\end{scope}
\end{tikzpicture}
\caption{The cut $C^{(1)}_1$ for $m=21$, $r=4$, $s=3$ (the arrows in the cut are represented by thick lines). A complete set of representatives $\cal D$ for the $G/A$-action is given by the vertices contained in the shadowed rhombus (the two $0$'s are identified), and the fixed points (i.e. $0$, $7$ and $14$) are the vertices of the latter. The quiver $\tilde{Q}_G$ is obtained by identifying the edges of the rhombus whose adjacent vertices have the same name, according to the orientation depicted. For a picture of $Q_G$ see Figure~\ref{fig_quiver cut C^1_1}.} \label{fig_cut C_1^1}
\end{figure}
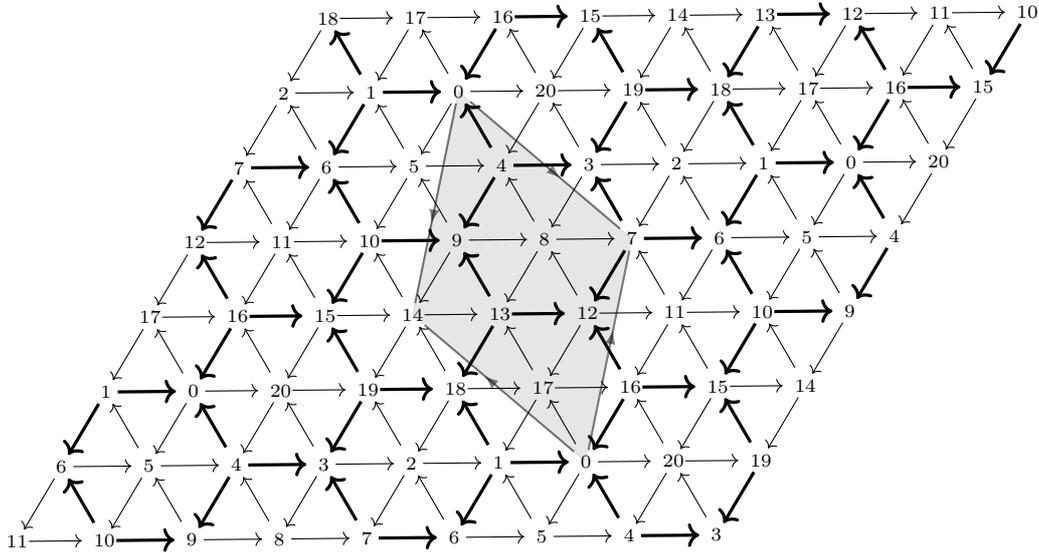

\begin{example}\label{ex_s=3}
Now we exhibit an example where $s=3$. Let $G=M(3,1)$, so we have $m=21$, $r=4$, $t=0$. Then $G$ is generated by the matrices
\[
\alpha = \left(
\begin{array}{ccc}
\epsilon_{21} & 0 & 0 \\
0 & \epsilon_{21}^4 & 0 \\
0 & 0 & \epsilon_{21}^{16}
\end{array}\right),
\qquad
\beta = \left(
\begin{array}{ccc}
0 & 0 & 1 \\
1 & 0 & 0 \\
0 & 1 & 0
\end{array}\right).
\]

We have already depicted the quivers $Q$ and $Q_A$ in Figure~\ref{fig_quiver-s=3}. In this case, the $G/A$-action on $Q$ is given by an anticlockwise rotation of $120^\circ$ around the origin $0$, and it is easy to see that the induced action on $Q_A$ can be realized as an anticlockwise rotation of $120^\circ$ around each fixed point.

In Figure~\ref{fig_cut C_1^1} we show a way to realize the quiver $\tilde{Q}_G$, together with the cut $C^{(1)}_1$. The quiver $Q_G$ and the grading induced by this cut are illustrated in Figure~\ref{fig_quiver cut C^1_1}. Note that in this case Corollary~\ref{cor_supp-s=2,3} holds, so the paths in the support of the superpotential $\omega_G$ are exactly the ones induced by paths in $\supp(\omega_A)$.

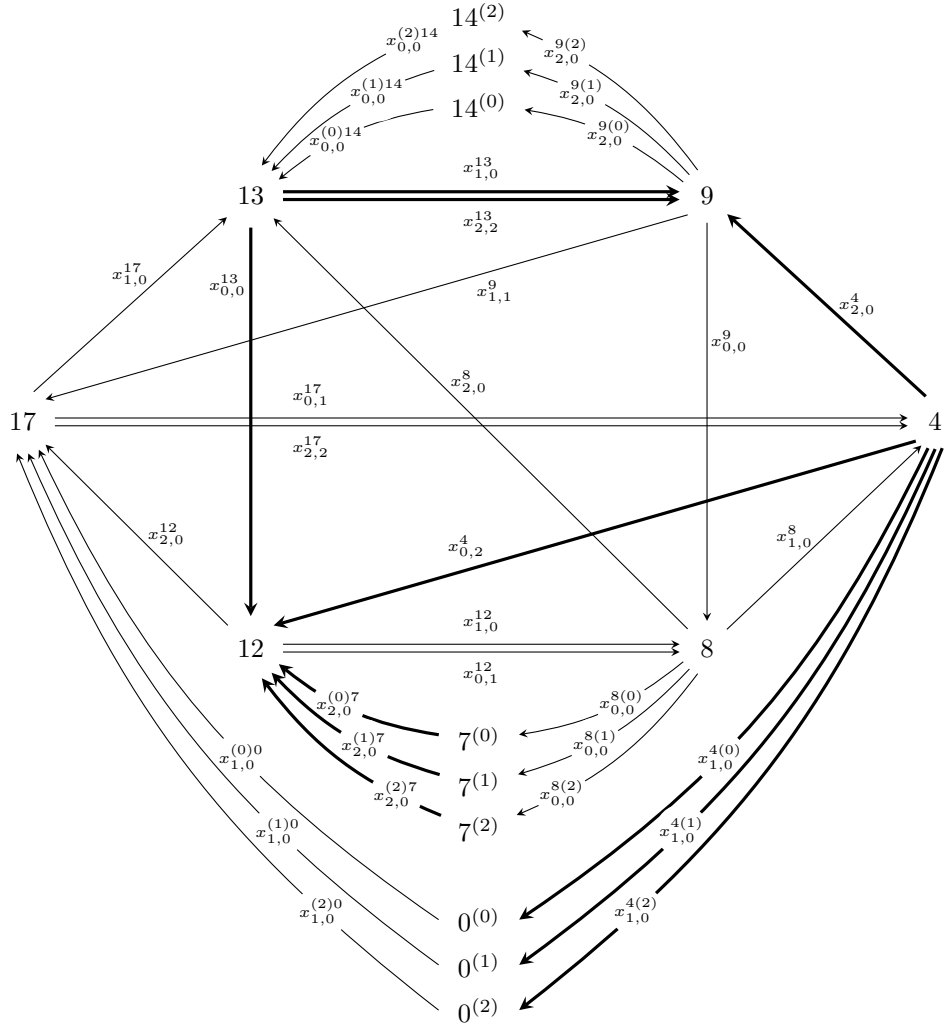
\begin{figure}
\centering
\begin{tikzpicture}[->,scale=0.6,inner sep=1.5mm,shape=circle,>=stealth]

\node (17) at (-10,0){$17$};
\node (13) at (-5,5){$13$};
\node (12) at (-5,-5){$12$};
\node (9) at (5,5){$9$};
\node (8) at (5,-5){$8$};
\node (4) at (10,0){$4$};
\begin{scope}[yshift=7cm]
\node (140) at (0,0){$14^{(0)}$};
\node (141) at (0,1){$14^{(1)}$};
\node (142) at (0,2){$14^{(2)}$};
\end{scope}
\begin{scope}[yshift=-7cm]
\node (70) at (0,0){$7^{(0)}$};
\node (71) at (0,-1){$7^{(1)}$};
\node (72) at (0,-2){$7^{(2)}$};
\end{scope}
\begin{scope}[yshift=-11cm]
\node (00) at (0,0){$0^{(0)}$};
\node (01) at (0,-1){$0^{(1)}$};
\node (02) at (0,-2){$0^{(2)}$};
\end{scope}

\begin{tiny}
\draw (17.70) -- node[midway,above,inner sep=2] {$x^{17}_{1,0}$} (13);
\draw[transform canvas={yshift=0.2em}] (17) -- node[midway,above,inner sep=0,pos=0.3] {$x^{17}_{0,1}$} (4);
\draw[transform canvas={yshift=-0.2em}] (17) -- node[midway,below,inner sep=0,pos=0.3] {$x^{17}_{2,2}$} (4);
\draw[very thick] (13) -- node[midway,left,inner sep=0,pos=0.15] {$x^{13}_{0,0}$} (12);
\draw[transform canvas={yshift=0.2em},very thick] (13) -- node[midway,above,inner sep=0] {$x^{13}_{1,0}$} (9);
\draw[transform canvas={yshift=-0.2em},very thick] (13) -- node[midway,below,inner sep=0] {$x^{13}_{2,2}$} (9);
\draw (12) -- node[midway,right,inner sep=2] {$x^{12}_{2,0}$} (17);
\draw[transform canvas={yshift=0.2em}] (12) -- node[midway,above,inner sep=0] {$x^{12}_{1,0}$} (8);
\draw[transform canvas={yshift=-0.2em}] (12) -- node[midway,below,inner sep=0] {$x^{12}_{0,1}$} (8);
\draw (9.225) -- node[midway,below,inner sep=0,pos=0.3] {$x^{9}_{1,1}$} (17.45);
\draw (9) -- node[midway,right,inner sep=0,pos=0.3] {$x^{9}_{0,0}$} (8);
\draw (9.150) to[bend right=15] node[midway,inner sep=0,fill=white,pos=0.5] {$x^{9(0)}_{2,0}$} (140);
\draw (9.130) to[bend right=15] node[midway,inner sep=0,fill=white,pos=0.7] {$x^{9(1)}_{2,0}$} (141);
\draw (9.110) to[bend right=15] node[midway,inner sep=0,fill=white,pos=0.8] {$x^{9(2)}_{2,0}$} (142);
\draw (8) -- node[midway,right,inner sep=2,pos=0.6] {$x^{8}_{2,0}$} (13);
\draw (8) -- node[midway,left,inner sep=2] {$x^{8}_{1,0}$} (4.240);
\draw (8.210) to[bend left=15] node[midway,inner sep=0,fill=white,pos=0.4] {$x^{8(0)}_{0,0}$} (70);
\draw (8.230) to[bend left=15] node[midway,inner sep=0,fill=white,pos=0.6] {$x^{8(1)}_{0,0}$} (71);
\draw (8.250) to[bend left=15] node[midway,inner sep=0,fill=white,pos=0.8] {$x^{8(2)}_{0,0}$} (72);
\draw[very thick] (4.110) -- node[midway,right,inner sep=2] {$x^{4}_{2,0}$} (9);
\draw[very thick] (4.225) -- node[midway,above,inner sep=0,pos=0.7] {$x^{4}_{0,2}$} (12.45);
\draw[very thick] (4.255) to[bend left=15] node[midway,inner sep=0,fill=white,pos=0.6] {$x^{4(0)}_{1,0}$} (00.east);
\draw[very thick] (4.270) to[bend left=15] node[midway,inner sep=0,fill=white,pos=0.7] {$x^{4(1)}_{1,0}$} (01.east);
\draw[very thick] (4.285) to[bend left=15] node[midway,inner sep=0,fill=white,pos=0.8] {$x^{4(2)}_{1,0}$} (02.east);
\draw (140) to[bend right=15] node[midway,inner sep=0,fill=white,pos=0.6] {$x^{(0)14}_{0,0}$} (13.30);
\draw (141) to[bend right=15] node[midway,inner sep=0,fill=white,pos=0.3] {$x^{(1)14}_{0,0}$} (13.50);
\draw (142) to[bend right=15] node[midway,inner sep=0,fill=white,pos=0.1] {$x^{(2)14}_{0,0}$} (13.70);
\draw[very thick] (70) to[bend left=15] node[midway,inner sep=0,fill=white,pos=0.6] {$x^{(0)7}_{2,0}$} (12.330);
\draw[very thick] (71) to[bend left=15] node[midway,inner sep=0,fill=white,pos=0.4] {$x^{(1)7}_{2,0}$} (12.310);
\draw[very thick] (72) to[bend left=15] node[midway,inner sep=0,fill=white,pos=0.2] {$x^{(2)7}_{2,0}$} (12.290);
\draw (00.west) to[bend left=15] node[midway,inner sep=0,fill=white,pos=0.4] {$x^{(0)0}_{1,0}$} (17.300);
\draw (01.west) to[bend left=15] node[midway,inner sep=0,fill=white,pos=0.3] {$x^{(1)0}_{1,0}$} (17.280);
\draw (02.west) to[bend left=15] node[midway,inner sep=0,fill=white,pos=0.2] {$x^{(2)0}_{1,0}$} (17.260);
\end{tiny}

\end{tikzpicture}
\caption{The quiver $Q_G$ for $m=21$, $r=4$, $s=3$. The thick arrows have degree $1$ with respect to the grading associated to the cut of Figure~\ref{fig_cut C_1^1}, so the quiver of the corresponding $2$-representation infinite algebra $(\Pi_G)_0$ is obtained by deleting these arrows.}\label{fig_quiver cut C^1_1}
\end{figure}
\end{example}

Now we will give an example where the conditions \ref{item_m=sn} and \ref{item_r=1(s)} do not hold, and we will show that in this case we have no invariant cuts.

\begin{example}
Set $m=7$, $r=2$, $s=3$, $t=0$, and let $G$ be the corresponding metacyclic group. The quivers $Q$, $Q_A$ and $\tilde{Q}_G$ are described in Figure~\ref{fig_bad cut}. For a picture of $Q_G$, see \cite[Example~5.5]{BSW}. Note that the quiver $Q_A$ can be embedded in a torus which carries a $G/A$-action with three fixed points. However, among these fixed points only one corresponds to a vertex of the quiver: the others are located in the ``barycentre of a triangle'', as we can see in the figure. Now consider the cyclic path $2\to 1 \to 4\to 2$: the action of $G/A$ sends each arrow in it to the arrow which precedes it in the path. Since a cut must contain exactly one arrow from this path, it is clear that we cannot have cuts which are $G/A$-invariant.
\end{example}

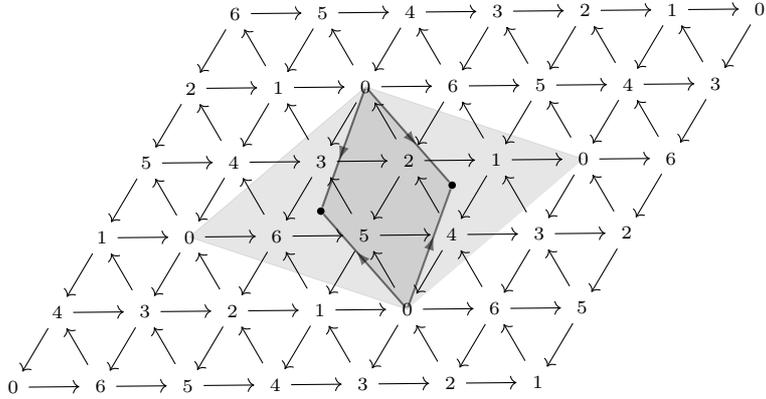
\begin{figure}
\centering
\tikzset{->-/.style={decoration={
  markings,
  mark=at position #1 with {\arrow{latex}}},postaction={decorate}},
  ->-/.default=0.6,}

\begin{tikzpicture}[->,scale=0.8]
\begin{scope}[rotate=-15,inner sep=2mm]
\begin{scriptsize}

\def\r{2}							
\def\m{7}

\pgfmathtruncatemacro{\b}{(\r-1)*0.333}

\coordinate (A1) at (1,0,-1);
\coordinate (A2) at (-1,1,0);
\coordinate (A0) at (0,-1,1);

\node[inner sep=0] (beta0) at (0,0,0){};								
\node[inner sep=0] (beta1) at ($-\r*(A1) + (A2)$){};
\node[inner sep=0] (beta2) at ($\r*(A2) - (A0)$){};
\node[inner sep=0] (beta3) at ($(beta2) - (beta1)$){};

\node[inner sep=1,fill,shape=circle] (fix0) at ($ 0.333*(beta0) + 0.333*(beta1) + 0.333*(beta2) $){};
\node[inner sep=1,fill,shape=circle] (fix1) at ($ 0.333*(beta0) + 0.333*(beta3) + 0.333*(beta2) $){};

\filldraw[very nearly transparent,-] ($(beta0)$) -- ($(beta1)$) -- ($(beta2)$) -- ($(beta3)$);

\filldraw[very nearly transparent,-] ($(beta0)$) -- ($(fix0)$) -- ($(beta2)$) -- ($(fix1)$);
\draw[thick,->-,-,opacity=0.5] (beta0) -- (fix0);
\draw[thick,->-,-,opacity=0.5] (beta2) -- (fix0);
\draw[thick,->-,-,opacity=0.5] (beta2) -- (fix1);
\draw[thick,->-,-,opacity=0.5] (beta0) -- (fix1);

\pgfmathtruncatemacro{\minx}{-\r-1}			
\pgfmathtruncatemacro{\maxx}{\r}
\pgfmathtruncatemacro{\miny}{-\b-1}
\pgfmathtruncatemacro{\maxy}{\r+1}

	\pgfmathtruncatemacro{\i}{Mod(-\minx+1+(-\maxy-1)*(\r+1),\m)}
	\node (D) at (\minx-1,\maxy+1,-\minx-\maxy) {$\i$};	
\foreach \y in {\miny,...,\maxy}
{
	\pgfmathtruncatemacro{\i}{Mod(-\minx+1+(-\y)*(\r+1),\m)}	
	\node (D) at (\minx-1,\y,-\minx-\y+1) {$\i$};		
	\node (P) at (\maxx,\y,-\maxx-\y) {};
	\node (Q) at (\maxx,\y+1,-\maxx-\y-1) {};
 \draw (Q) -- (P);												
}

\foreach \x in {\minx,...,\maxx}
{
\pgfmathtruncatemacro{\i}{Mod(-\x+(-\maxy-1)*(\r+1),\m)}	
\node (M) at (\x,\maxy+1,-\x-\maxy-1)  {$\i$};
\node (N) at (\x-1,\maxy+1,-\x-\maxy)  {};
 \draw (N) -- (M);											
	\foreach \y in {\miny,...,\maxy}
{
	\pgfmathtruncatemacro{\i}{Mod(-\x+(-\y)*(\r+1),\m)}
	\node (A) at (\x,\y,-\x-\y)  {$\i$};		
	\node (B) at (\x-1,\y+1,-\x-\y) {};
	\node (C) at (\x-1,\y,-\x-\y+1) {};
	  \draw (A) -- (B);
	  \draw (C) -- (A);
	  \draw (B) -- (C);
}
}
\end{scriptsize}
\end{scope}
\end{tikzpicture}
\caption{The quiver $Q^{(3)}$ with $m=7$, $r=2$. Fundamental domains for the quivers $Q_A$ and $\tilde{Q}_G$ are given by, respectively, the light and the dark shadowed regions. The vertices of the latter are the fixed points for the $G/A$ action on the torus.}\label{fig_bad cut}
\end{figure}

\subsection{Examples for $G\not\subseteq\SL(s,\C)$}

\begin{example}\label{ex_bin-dih}
Let $G=M(2,2)$, so we have  $s=2$, $r=5$, $m=12$ and $t=6$. The group $G$ is generated by the matrices
\[
\alpha = \left(
\begin{array}{cc}
\epsilon_{12} & 0 \\
0 & \epsilon_{12}^5
\end{array}\right),
\qquad
\beta = \left(
\begin{array}{cc}
0 & -1 \\
1 & 0
\end{array}\right)
\]
and is not contained in $\SL(2,\C)$, because $c=1+r=6$. Thus its image $G'$ under the embedding in $\SL(3,\C)$ is generated by
\[
\alpha = \left(
\begin{array}{ccc}
\epsilon_{12} & 0 \\
0 & \epsilon_{12}^5 & 0 \\
0 & 0 & \epsilon_{12}^6
\end{array}\right),
\qquad
\beta = \left(
\begin{array}{ccc}
0 & -1 & 0 \\
1 & 0 & 0 \\
0 & 0 & 1
\end{array}\right).
\]

The quivers $Q_{A'}$ and $Q_{G'}$, together with some cuts, are represented in Figures~\ref{fig_bin-dih-A}, \ref{fig_bin-dih-A2} and \ref{fig_bin-dih-G}. Note that the arrows of type $i\to \tau(i)$ are the ones which point in south-west direction.

We end this example by remarking that in this case $u=s=2$, and so the condition $(u,s)=1$ in Proposition~\ref{prop_D-invariant} is not satisfied. However, it is still possible to choose a set of representatives $\cal D$ which is invariant under $\tau'$ (see Figure~\ref{fig_bin-dih-G}), so such condition is not necessary. Moreover, in Figure~\ref{fig_bin-dih-A2} we exhibit a cut which is not induced by one of the form $C^{(l)}_k$, but it is easily checked that it still yields a grading on $\Pi_G$ such that $(\Pi_G)_0$ is $2$-representation infinite.

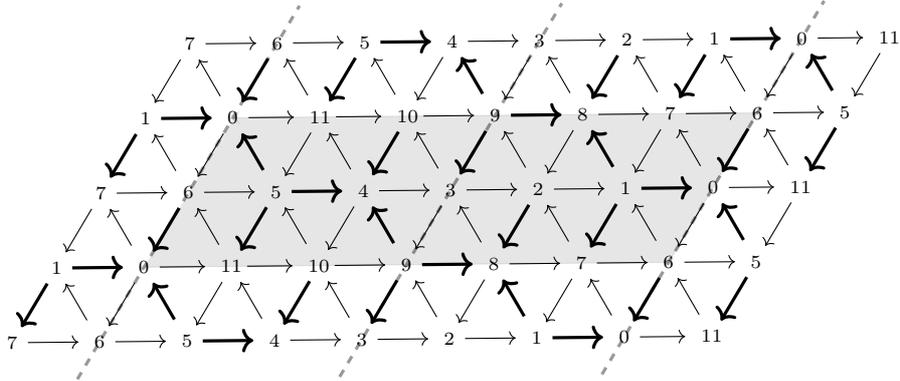
\begin{figure}
\centering
\begin{tikzpicture}[->,scale=0.8]
\begin{scope}[rotate=-15,inner sep=2mm]
\begin{scriptsize}

\def\b{2}							
\def\n{6}

\pgfmathtruncatemacro{\m}{2*\n}
\pgfmathtruncatemacro{\r}{2*\b+1}
\pgfmathtruncatemacro{\c}{1+\r}
\pgfmathtruncatemacro{\u}{\m/(gcd(\c,\m))};
\pgfmathtruncatemacro{\q}{\m/(gcd(\r-1,\m))};

\coordinate (A1) at (1,0,-1);
\coordinate (A2) at (-1,1,0);
\coordinate (A0) at (0,-1,1);

\node (beta0) at (0,0,0){};								
\node (beta1) at ($2*\q*(A1)$){};
\node (beta2) at ($(beta1)+\u*(A2)+\u*(A1)$){};
\node (beta3) at ($(beta2) - (beta1)$){};

\filldraw[very nearly transparent] ($(beta0)$) -- ($(beta1)$) -- ($(beta2)$) -- ($(beta3)$);

\pgfmathtruncatemacro{\minx}{0}			
\pgfmathtruncatemacro{\maxx}{2*\q+1}
\pgfmathtruncatemacro{\miny}{-1}
\pgfmathtruncatemacro{\maxy}{\u}

\draw[very thick,dashed,-,opacity=0.4] (\minx,\miny-0.5,-\minx-\miny+0.5) -- (\minx,\maxy+1.5,-\minx-\maxy-1.5);
\draw[very thick,dashed,-,opacity=0.4] (\minx+3,\miny-0.5,-\minx-\miny-2.5) -- (\minx+3,\maxy+1.5,-\minx-\maxy-4.5);
\draw[very thick,dashed,-,opacity=0.4] (\minx+6,\miny-0.5,-\minx-\miny-5.5) -- (\minx+6,\maxy+1.5,-\minx-\maxy-7.5);


	\pgfmathtruncatemacro{\i}{Mod(-\minx+1+(-\maxy-1)*(\r+1),\m)}
	\node (D) at (\minx-1,\maxy+1,-\minx-\maxy) {$\i$};	
\foreach \y in {\miny,...,\maxy}
{
	\pgfmathtruncatemacro{\i}{Mod(-\minx+1+(-\y)*(\r+1),\m)}	
	\pgfmathtruncatemacro{\cut}{Mod(-\i,4)}
	\pgfmathtruncatemacro{\cutt}{Mod(-\i+2,4)}
	\node (D) at (\minx-1,\y,-\minx-\y+1) {$\i$};		
	\node (P) at (\maxx,\y,-\maxx-\y) {};
	\node (Q) at (\maxx,\y+1,-\maxx-\y-1) {};
	\ifnum \cutt>\cut \draw[very thick] (Q) -- (P);
    [\else \draw (Q) -- (P);] \fi											
}

\foreach \x in {\minx,...,\maxx}
{
\pgfmathtruncatemacro{\i}{Mod(-\x+(-\maxy-1)*(\r+1),\m)}	
	\pgfmathtruncatemacro{\cut}{Mod(-\i,4)}
	\pgfmathtruncatemacro{\cutt}{Mod(-\i-1,4)}
\node (M) at (\x,\maxy+1,-\x-\maxy-1)  {$\i$};
\node (N) at (\x-1,\maxy+1,-\x-\maxy)  {};
 \ifnum \cutt>\cut \draw[very thick] (N) -- (M);
 [\else \draw (N) -- (M);] \fi											
	\foreach \y in {\miny,...,\maxy}
{
	\pgfmathtruncatemacro{\i}{Mod(-\x+(-\y)*(\r+1),\m)}
	\pgfmathtruncatemacro{\cutA}{Mod(-\i,4)}
	\pgfmathtruncatemacro{\cutB}{Mod(-\i+1,4)}
	\pgfmathtruncatemacro{\cutC}{Mod(-\i-1,4)}
	\node (A) at (\x,\y,-\x-\y)  {$\i$};		
	\node (B) at (\x-1,\y+1,-\x-\y) {};			
	\node (C) at (\x-1,\y,-\x-\y+1) {};			
	  \ifnum \cutA>\cutB \draw[very thick] (A) -- (B);
      [\else \draw (A) -- (B);] \fi
	  \ifnum \cutC>\cutA \draw[very thick] (C) -- (A);
      [\else \draw (C) -- (A);] \fi
	  \ifnum \cutB>\cutC \draw[very thick] (B) -- (C);
      [\else \draw (B) -- (C);] \fi
}
}
\end{scriptsize}
\end{scope}
\end{tikzpicture}
\caption{The quiver $Q^{(3)}$, where the vertices are labelled with their image under the isomorphism $\eta$ according to the setting of Example~\ref{ex_bin-dih}. The quiver $Q_{A'}$ is given by the vertices contained in the shaded parallelogram, where the vertices on opposite sides with the same label are identified. Hence $Q_{A'}$ can be embedded on a torus. The action of $G/A$ on $Q_{A'}$ is given by reflecting along the dashed lines, and the vertices contained in the latter are the fixed points. The thick arrows represent the cut $C^{(2)}_1$.}\label{fig_bin-dih-A}
\end{figure}

\begin{figure}
\centering
\begin{tikzpicture}[->,scale=0.8]
\begin{scope}[rotate=-15,inner sep=2mm]
\begin{scriptsize}

\def\b{2}							
\def\n{6}

\pgfmathtruncatemacro{\m}{2*\n}
\pgfmathtruncatemacro{\r}{2*\b+1}
\pgfmathtruncatemacro{\c}{1+\r}
\pgfmathtruncatemacro{\u}{\m/(gcd(\c,\m))};
\pgfmathtruncatemacro{\q}{\m/(gcd(\r-1,\m))};

\coordinate (A1) at (1,0,-1);
\coordinate (A2) at (-1,1,0);
\coordinate (A0) at (0,-1,1);

\node (beta0) at (0,0,0){};								
\node (beta1) at ($2*\q*(A1)$){};
\node (beta2) at ($(beta1)+\u*(A2)+\u*(A1)$){};
\node (beta3) at ($(beta2) - (beta1)$){};

\filldraw[very nearly transparent] ($(beta0)$) -- ($(beta1)$) -- ($(beta2)$) -- ($(beta3)$);

\pgfmathtruncatemacro{\minx}{0}			
\pgfmathtruncatemacro{\maxx}{2*\q+1}
\pgfmathtruncatemacro{\miny}{-1}
\pgfmathtruncatemacro{\maxy}{\u}

\draw[very thick,dashed,-,opacity=0.4] (\minx,\miny-0.5,-\minx-\miny+0.5) -- (\minx,\maxy+1.5,-\minx-\maxy-1.5);
\draw[very thick,dashed,-,opacity=0.4] (\minx+3,\miny-0.5,-\minx-\miny-2.5) -- (\minx+3,\maxy+1.5,-\minx-\maxy-4.5);
\draw[very thick,dashed,-,opacity=0.4] (\minx+6,\miny-0.5,-\minx-\miny-5.5) -- (\minx+6,\maxy+1.5,-\minx-\maxy-7.5);


	\pgfmathtruncatemacro{\i}{Mod(-\minx+1+(-\maxy-1)*(\r+1),\m)}
	\node (D) at (\minx-1,\maxy+1,-\minx-\maxy) {$\i$};	
\foreach \y in {\miny,...,\maxy}
{
	\pgfmathtruncatemacro{\i}{Mod(-\minx+1+(-\y)*(\r+1),\m)}	
	\node (D) at (\minx-1,\y,-\minx-\y+1) {$\i$};		
	\node (P) at (\maxx,\y,-\maxx-\y) {};
	\node (Q) at (\maxx,\y+1,-\maxx-\y-1) {};
 \ifnum\i=7 \draw[very thick] (Q) -- (P);	
 [\else \draw (Q) -- (P);] \fi							
}

\foreach \x in {\minx,...,\maxx}
{
\pgfmathtruncatemacro{\i}{Mod(-\x+(-\maxy-1)*(\r+1),\m)}	
\node (M) at (\x,\maxy+1,-\x-\maxy-1)  {$\i$};
\node (N) at (\x-1,\maxy+1,-\x-\maxy)  {};
 \ifnum\ifnum\i=4 1\else\ifnum\i=11 1\else0\fi\fi=1 \draw[very thick] (N) -- (M);
	  [\else \draw (N) -- (M);] \fi						
	\foreach \y in {\miny,...,\maxy}
{
	\pgfmathtruncatemacro{\i}{Mod(-\x+(-\y)*(\r+1),\m)}
	\node (A) at (\x,\y,-\x-\y)  {$\i$};		
	\node (B) at (\x-1,\y+1,-\x-\y) {};
	\node (C) at (\x-1,\y,-\x-\y+1) {};
	  \ifnum\ifnum\i=0 1\else\ifnum\i=1 1\else\ifnum\i=2 1\else0\fi\fi\fi=1 \draw[very thick] (A) -- (B);
	  [\else \draw (A) -- (B);] \fi
	  \ifnum\ifnum\i=3 1\else\ifnum\i=5 1\else\ifnum\i=10 1\else\ifnum\i=6 1\else\ifnum\i=7 1\else\ifnum\i=8 1
	  \else0\fi\fi\fi\fi\fi\fi=1 \draw[very thick] (B) -- (C);
	  [\else \draw (B) -- (C);] \fi
	  \ifnum\ifnum\i=9 1\else\ifnum\i=4 1\else\ifnum\i=11 1\else0\fi\fi\fi=1 \draw[very thick] (C) -- (A);
	  [\else \draw (C) -- (A);] \fi
}
}
\end{scriptsize}
\end{scope}
\end{tikzpicture}
\caption{A cut in the setting of Example~\ref{ex_bin-dih} which is not of the form $C^{(l)}_k$ for any $k,l$.}\label{fig_bin-dih-A2}
\end{figure}
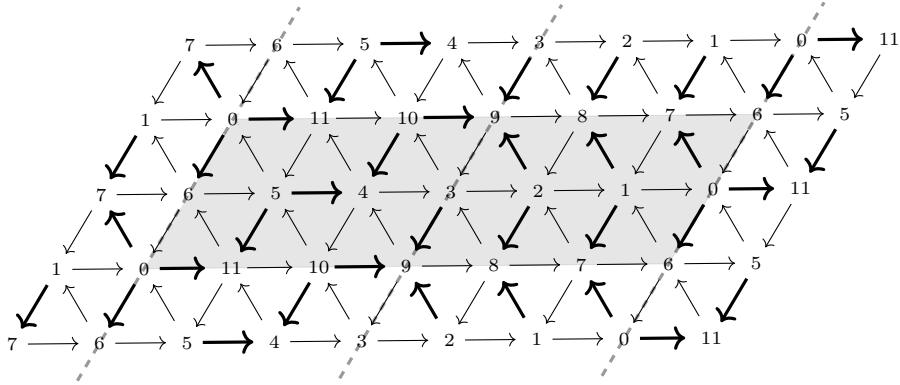

\begin{figure}
\centering
\begin{tikzpicture}[->,scale=1]
\begin{scope}[rotate=-15,inner sep=2mm]
\begin{scriptsize}

\def\b{2}							
\def\n{6}

\pgfmathtruncatemacro{\m}{2*\n}
\pgfmathtruncatemacro{\r}{2*\b+1}
\pgfmathtruncatemacro{\c}{1+\r}
\pgfmathtruncatemacro{\u}{\m/(gcd(\c,\m))};
\pgfmathtruncatemacro{\q}{\m/(gcd(\r-1,\m))};

\coordinate (A1) at (1,0,-1);
\coordinate (A2) at (-1,1,0);
\coordinate (A0) at (0,-1,1);

\coordinate (lshift) at (0,2,1);			
\coordinate (rshift) at (0,-1.5,-1);		

\node (beta0) at (0,0,0){};								
\node (beta1) at ($2*\q*(A1)$){};
\node (beta2) at ($(beta1)+\u*(A2)+\u*(A1)$){};
\node (beta3) at ($(beta2) - (beta1)$){};

\pgfmathtruncatemacro{\minx}{1}			
\pgfmathtruncatemacro{\maxx}{\q}
\pgfmathtruncatemacro{\miny}{0}
\pgfmathtruncatemacro{\maxy}{\u-1}

	\pgfmathtruncatemacro{\i}{Mod(-\minx+1+(-\maxy-1)*(\r+1),\m)}
	\node (D-) at (\minx-1,\maxy+1,-\minx-\maxy) {\scriptsize $\i^-$};	
	\node (D+) at ($(\minx-1,\maxy+1,-\minx-\maxy)+(lshift)$) {\scriptsize $\i^+$};	
	\node (D) at (\minx,\maxy+1,-\minx-\maxy-1) {};
	\draw[very thick] (D+) -- (D);

\foreach \y in {\miny,...,\maxy}
{
	\pgfmathtruncatemacro{\i}{Mod(-\minx+1+(-\y)*(\r+1),\m)}	
	\node (D-) at (\minx-1,\y,-\minx-\y+1) {\scriptsize $\i^-$};		
	\node (D+) at ($(\minx-1,\y,-\minx-\y+1)+(lshift)$) {\scriptsize $\i^+$};		
	\node (E+) at ($(\minx-1,\y+1,-\minx-\y)+(lshift)$) {};
	\node (F) at (\minx,\y,-\minx-\y) {};	
	\node (P-) at (\maxx,\y,-\maxx-\y) {};
	\node (Q-) at (\maxx,\y+1,-\maxx-\y-1) {};
	\node (P+) at ($(\maxx,\y,-\maxx-\y)+(rshift)$) {};
	\node (Q+) at ($(\maxx,\y+1,-\maxx-\y-1)+(rshift)$) {};
 \ifnum\i=0	 \draw[very thick] (Q-) -- (P-);
 			 \draw[very thick] (Q+) -- (P+);
 [\else \draw (Q-) -- (P-);
 		\draw (Q+) -- (P+);] \fi							
 \ifnum\i=6	 \draw[very thick] (E+) -- (D+);
 [\else \draw (E+) -- (D+);] \fi												
 \draw (F) -- (E+);
 \ifnum\i=0	 \draw[very thick] (D+) -- (F);
 [\else \draw (D+) -- (F);] \fi
}

\pgfmathtruncatemacro{\maxxmone}{\maxx-1}
\foreach \x in {\minx,...,\maxxmone}
{
\pgfmathtruncatemacro{\i}{Mod(-\x+(-\maxy-1)*(\r+1),\m)}	
\node (M) at (\x,\maxy+1,-\x-\maxy-1)  {\scriptsize $\i$};
\node (N) at (\x-1,\maxy+1,-\x-\maxy)  {};
 \ifnum\ifnum\i=4 1\else\ifnum\i=11 1\else0\fi\fi=1 \draw[very thick] (N) -- (M);
	  [\else \draw (N) -- (M);] \fi											
	\foreach \y in {\miny,...,\maxy}
{
	\pgfmathtruncatemacro{\i}{Mod(-\x+(-\y)*(\r+1),\m)}
	\node (A) at (\x,\y,-\x-\y)  {\scriptsize $\i$};		
	\node (B) at (\x-1,\y+1,-\x-\y) {};
	\node (C) at (\x-1,\y,-\x-\y+1) {};
	  \ifnum\ifnum\i=0 1\else\ifnum\i=1 1\else\ifnum\i=2 1\else0\fi\fi\fi=1 \draw[very thick] (A) -- (B);
	  [\else \draw (A) -- (B);] \fi
	  \ifnum\ifnum\i=3 1\else\ifnum\i=5 1\else\ifnum\i=10 1\else\ifnum\i=6 1\else\ifnum\i=7 1\else\ifnum\i=8 1
	  \else0\fi\fi\fi\fi\fi\fi=1 \draw[very thick] (B) -- (C);
	  [\else \draw (B) -- (C);] \fi
	  \ifnum\ifnum\i=9 1\else\ifnum\i=4 1\else\ifnum\i=11 1\else0\fi\fi\fi=1 \draw[very thick] (C) -- (A);
	  [\else \draw (C) -- (A);] \fi
}
}

\pgfmathtruncatemacro{\i}{Mod(-\maxx+(-\maxy-1)*(\r+1),\m)}	
\node (M+) at (\maxx,\maxy+1,-\maxx-\maxy-1)  {\scriptsize $\i^+$};
\node (M-) at ($(\maxx,\maxy+1,-\maxx-\maxy-1)+(rshift)$)  {\scriptsize $\i^-$};
\node (N) at (\maxx-1,\maxy+1,-\maxx-\maxy)  {};
 \draw[very thick] (N) -- (M+);											
 \draw[very thick] (N) -- (M-);											

	\foreach \y in {\miny,...,\maxy}
{
	\pgfmathtruncatemacro{\i}{Mod(-\maxx+(-\y)*(\r+1),\m)}
	\node (A+) at (\maxx,\y,-\maxx-\y)  {\scriptsize $\i^+$};		
	\node (A-) at ($(\maxx,\y,-\maxx-\y)+(rshift)$)  {\scriptsize $\i^-$};		
	\node (B) at (\maxx-1,\y+1,-\maxx-\y) {};
	\node (C) at (\maxx-1,\y,-\maxx-\y+1) {};
	  \draw (A+) -- (B);
	  \draw (A-) -- (B);
	  \ifnum\i=9	 \draw[very thick] (C) -- (A-);
 			 \draw[very thick] (C) -- (A+);
 	  [\else \draw (C) -- (A-);
 		\draw (C) -- (A+);] \fi
	  \ifnum\i=3	 \draw[very thick] (B) -- (C);
 		[\else \draw (B) -- (C);] \fi
}

\end{scriptsize}
\end{scope}
\end{tikzpicture}
\caption{The quiver $Q_{G'}$ obtained from Figure~\ref{fig_bin-dih-A2} by choosing as set of representatives $\cal D$ the vertices contained in the left-side half of the parallelogram. Note that now only the upper and lower sides are identified. The thick arrows are induced by an invariant cut in $Q$.}\label{fig_bin-dih-G}
\end{figure}
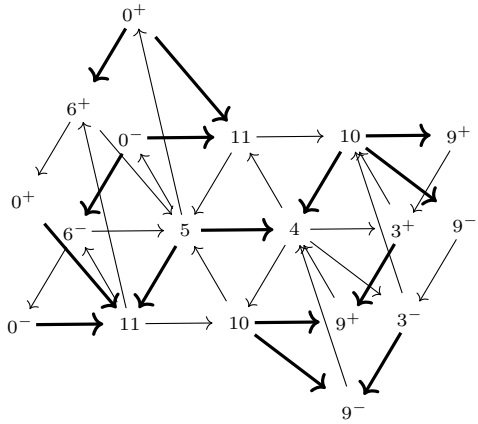

\end{example}

\end{document}